\pdfoutput=1
\documentclass[10pt,reqno]{amsart}
\usepackage[tt=false]{libertine}

\usepackage{kbordermatrix}
\renewcommand{\kbldelim}{(}
\renewcommand{\kbrdelim}{)}

\newcommand\s[1]{$\smash{#1}$}
\catcode`\$=\active
\gdef$#1${\s{#1}}

\usepackage{mathrsfs}

\usepackage[varbb]{newpxmath}

\usepackage{bm}

\usepackage{enumerate}
\usepackage[margin=1in]{geometry}

\usepackage[usenames,dvipsnames]{xcolor}

\usepackage[colorlinks=true,
	linkcolor=violet,
	citecolor=SeaGreen]{hyperref}
\hypersetup{bookmarksopen=true}

\usepackage{graphicx}
\usepackage[font=small,justification=centering]{caption}
\usepackage[font=footnotesize,justification=centering]{subcaption}

\newtheorem{thm}{Theorem}[section]
\newtheorem{lem}[thm]{Lemma}
\newtheorem{cor}[thm]{Corollary}
\newtheorem{ppn}[thm]{Proposition}

\theoremstyle{definition}
\newtheorem{dfn}[thm]{Definition}
\newtheorem{rmk}[thm]{Remark}

\def\vec{\underline}

\newcommand{\beq}{\begin{equation}}
\newcommand{\eeq}{\end{equation}}
\newcommand{\I}{\mathbf{1}}
\newcommand{\Ind}[1]{\mathbf{1}\{#1\}}
\newcommand{\f}{\frac}
\renewcommand{\log}{\ln}
\newcommand{\set}[1]{\{#1\}}
\renewcommand{\P}{\mathbb{P}}
\newcommand{\E}{\mathbb{E}}

\newcommand{\Bin}{\mathrm{Bin}}
\newcommand{\RPC}{\mathrm{RPC}}
\newcommand{\PP}{\textup{\textsf{P}}}
\newcommand{\relent}{\mathcal{H}}
\DeclareMathOperator{\Var}{Var}
\DeclareMathOperator{\Cov}{Cov}
\DeclareMathOperator{\supp}{supp}
\newcommand{\st}{\textup{\textsf{t}}}
\newcommand{\bemph}[1]{\textbf{\textup{#1}}}

\newcommand{\hDEL}{\hat{\nabla}}
\newcommand{\dDEL}{\dot{\nabla}}
\newcommand{\lit}{\TTT{L}}
\newcommand{\ulit}{\vec{\smash{\TTT{L}}}}

\newcommand{\TTT}[1]{\textup{\texttt{#1}}}
\newcommand{\zro}{\TTT{0}}
\newcommand{\one}{\TTT{1}}
\newcommand{\free}{\TTT{f}}
\newcommand{\UETA}{\vec{\smash{\ETA}}}
\newcommand{\ux}{\vec{x}}
\newcommand{\ETA}{\mathfrak{y}}

\newcommand{\ww}{\TTT{w}}

\newcommand{\vv}{\TTT{v}}
\newcommand{\rr}{\TTT{r}}
\newcommand{\sw}{\TTT{s}}
\newcommand{\hs}{\hat{\sw}}
\newcommand{\ds}{\dot{\sw}}
\newcommand{\dv}{\dot{\vv}}
\newcommand{\hv}{\hat{\vv}}
\newcommand{\dr}{\dot{\rr}}

\newcommand{\dw}{\dot{\ww}}
\newcommand{\hw}{\hat{\ww}}
\newcommand{\uw}{\vec{\smash{\ww}}}
\newcommand{\udw}{\vec{\smash{\dw}}}

\newcommand{\uhw}{\vec{\hw}}

\newcommand{\hq}{\hat{q}}
\newcommand{\dq}{\dot{q}}
\newcommand{\dz}{\dot{z}} 
\newcommand{\dfz}{\dot{\mathfrak{z}}} 
\newcommand{\hfz}{\hat{\mathfrak{z}}} 
\newcommand{\efz}{\bar{\mathfrak{z}}} 
\newcommand{\vph}{\varphi}
\newcommand{\dph}{\dot{\vph}}
\newcommand{\hph}{\hat{\vph}}
\newcommand{\eph}{\bar{\vph}}

\newcommand{\GG}{\mathscr{G}} 
\newcommand{\HH}{\mathscr{H}} 
\newcommand{\SOL}{\textsf{S}} 
\newcommand{\LOC}{\textsf{Q}} 

\newcommand{\asat}{\alpha_\textup{sat}}
\newcommand{\aG}{\alpha_\textup{Ga}}
\newcommand{\aubd}{\alpha_\textup{ubd}}
\newcommand{\amax}{\alpha_{\max}}

\newcommand{\EE}{\textup{\textsf{E}}}
\newcommand{\EMIN}{\EE_{\min}} 
\newcommand{\ee}{\textup{\textsf{e}}}

\newcommand{\emin}{\ee_{\min}}

\newcommand{\einf}{\ee_\bullet}
\newcommand{\esup}{\ee^\bullet}
\newcommand{\pinf}{\pgs^\bullet}
\newcommand{\psup}{\pgs_\bullet}

\newcommand{\pgs}{\textup{\textsf{p}}}
\newcommand{\pmax}{\textup{\textsf{p}}_{\max}}
\newcommand{\pubd}{\textup{\textsf{p}}_\textup{ubd}}
\newcommand{\eone}{\ee_{\onersb}}
\newcommand{\ealg}{\ee_\textup{alg}}
\newcommand{\elbd}{\ee_\textup{lbd}}

\newcommand{\CC}{\mathbf{c}}
\newcommand{\MM}{\mathcal{M}}
\newcommand{\MMnf}{\mathcal{M}^\bullet}
\newcommand{\MMstar}{\mathcal{M}^\star}

\newcommand{\WP}{\textsf{\footnotesize WP}}
\newcommand{\hWP}{\hat{\WP}}
\newcommand{\dWP}{\dot{\WP}}

\newcommand{\SP}{\textup{\textsf{\footnotesize SP}}}
\newcommand{\dSP}{\hat{\SP}}
\newcommand{\hSP}{\dot{\SP}}

\newcommand{\tree}{\tau}

\newcommand{\tq}{\tilde{q}}
\newcommand{\AM}{\textup{\textsc{am}}}
\newcommand{\GM}{\textup{\textsc{gm}}}

\newcommand{\AAA}{\mathbf{A}}
\newcommand{\GGG}{\mathbf{G}}
\newcommand{\PPP}{\mathbf{P}}
\newcommand{\QQQ}{\mathbf{Q}}
\newcommand{\SSS}{\mathbf{S}}
\newcommand{\RRR}{\mathbf{R}}

\newcommand{\binLAM}{\bar{L}_\AM}
\newcommand{\binLGM}{\bar{L}_\GM}

\newcommand{\LAM}{L_\AM}
\newcommand{\LGM}{L_\GM}
\newcommand{\lam}{\ell_\AM}
\newcommand{\lgm}{\ell_\GM}
\newcommand{\ldam}{\ell_{d,\AM}}
\newcommand{\ldgm}{\ell_{d,\GM}}
\newcommand{\pam}{p_\AM}
\newcommand{\pgm}{p_\GM}

\newcommand{\LdAM}{L_{d,\AM}}
\newcommand{\LdGM}{L_{d,\GM}}
\newcommand{\binLdAM}{\bar{L}_{d,\AM}}
\newcommand{\binLdGM}{\bar{L}_{d,\GM}}

\newcommand{\branch}{\bm{b}} 
\newcommand{\ep}{\epsilon}
\newcommand{\bep}{\bm{\ep}}
\newcommand{\bup}{\bm{\Upsilon}}
\newcommand{\bde}{\bm{\delta}}
\newcommand{\bpi}{\bm{\pi}}
\newcommand{\bta}{\bm{\tau}}
\newcommand{\bxi}{\bm{\xi}}
\newcommand{\bvp}{\bm{\varpi}}
\newcommand{\bsi}{\bm{\sigma}}
\newcommand{\bB}{\bm{B}}
\newcommand{\hB}{\hat{\bm{B}}}
\newcommand{\dB}{\dot{\bm{B}}}
\newcommand{\bN}{\bm{N}}
\newcommand{\hN}{\hat{\bm{N}}}
\newcommand{\dN}{\dot{\bm{N}}}
\newcommand{\bPi}{\bm{\Pi}}
\newcommand{\bGa}{\bm{\Gamma}}
\newcommand{\bXi}{\bm{\Xi}}
\newcommand{\bTe}{\bm{\Theta}}

\newcommand{\PROJ}{\bm{P}}
\newcommand{\onersb}{\textup{\oldstylenums{1}\textsc{rsb}}}
\newcommand{\tworsb}{\textup{\oldstylenums{2}\textsc{rsb}}}
\newcommand{\frsb}{\textup{\textsc{frsb}}}

\newcommand{\FF}{\mathfrak{F}}
\newcommand{\SIGMA}{\mathfrak{S}}
\newcommand{\RSFF}{\mathfrak{F}^\textup{\textsc{rs}}}
\newcommand{\PhiOne}{\Phi^{\onersb}}
\newcommand{\PhiTwo}{\Phi^{\tworsb}}
\newcommand{\WW}{\mathcal{W}}
\newcommand{\FG}{\mathcal{G}}
\newcommand{\NU}{\nu}
\newcommand{\ZETA}{\zeta}

\title{Breaking of 1RSB in random MAX-NAE-SAT}

\author[Z.\ Bartha]{Zsolt Bartha$^\star$}
\author[N.\ Sun]{Nike Sun$^{\circ\star}$}
\author[Y.\ Zhang]{Yumeng Zhang$^\bullet$}
\thanks{$^{\star}$Statistics Department, Berkeley;
$^{\circ}$Mathematics Department, MIT; 
$^{\bullet}$Statistics Department, Stanford.}

\newcommand{\II}{\textup{\ref{II}}}
\newcommand{\pQ}{\bar{Q}}
\newcommand{\PROB}{\mathscr{P}}

\newcommand{\dF}{\dot{F}}
\newcommand{\dH}{\dot{H}}
\newcommand{\dZ}{\dot{Z}}
\newcommand{\hZ}{\hat{Z}}
\newcommand{\eZ}{\bar{Z}}
\newcommand{\dL}{\dot{L}}
\newcommand{\hJ}{\hat{J}}
\newcommand{\av}[1]{\langle#1\rangle}
\newcommand{\bph}{\acute{\vph}}

\begin{document}

\begin{abstract}
For several models of random constraint satisfaction problems, it was conjectured by physicists and later proved that a sharp satisfiability transition occurs. For random $k$-\textsc{sat} and related models it happens at clause density $\alpha=\alpha_\textup{sat}\asymp 2^k$. Just below the threshold, further results suggest that the solution space has a ``$\onersb$'' structure of a large bounded number of near-orthogonal clusters inside $\set{0,1}^N$.

In the \emph{unsatisfiable} regime $\alpha>\asat$, it is natural to consider the problem of \emph{max-satisfiability}: violating the least number of constraints. This is a combinatorial optimization problem on the random energy landscape defined by the problem instance. For a simplified variant, the \emph{strong refutation} problem, there is strong evidence that an algorithmic transition occurs around $\alpha=N^{k/2-1}$. For $\alpha$ \emph{bounded} in $N$, a very precise estimate of the max-sat value was obtained by Achlioptas, Naor, and Peres (2007), but it is not sharp enough to indicate the nature of the energy landscape. Later work (Sen, 2016; Panchenko, 2016) shows that for $\alpha$ very large (roughly,  $\Omega(64^k)$) the max-sat value approaches the \emph{mean-field} (complete graph) limit: this is conjectured to have an ``$\frsb$'' structure where near-optimal configurations form clusters within clusters, in an ultrametric hierarchy of infinite depth inside $\set{0,1}^N$. A stronger form of \textsc{frsb} was shown in several recent works to have algorithmic implications (again, in complete graphs). Consequently we find it of interest to understand how the model transitions from $\onersb$ near $\asat$, to (conjecturally) $\frsb$ for large $\alpha$. In this paper we show that in the random regular $k$-\textsc{nae-sat} model, the $\onersb$ description breaks down already above $\alpha\asymp 4^k/k^3$. This is proved by an explicit perturbation in the $\tworsb$ parameter space. The choice of perturbation is inspired by the ``bug proliferation'' mechanism proposed by physicists (Montanari and Ricci-Tersenghi, 2003; Krzakala, Pagnani, and Weigt, 2004), corresponding roughly to a percolation-like threshold for a subgraph of dependent variables.
\end{abstract}

\maketitle

\section{Introduction}

A \bemph{random constraint satisfaction problem} (random \textsc{csp}), broadly construed, is any problem specified by $N$ variables subject to $M$ random constraints. We shall consider a prototypical example, \bemph{random regular $k$-\textsc{nae-sat}}, where an instance $\GG_N$ involves $N$ binary variables $x_i\in\set{\zro,\one}$, subject to $M=N\alpha$ random constraints such that each constraint involves a subset of $k$ variables (the formal definition is below). In the \bemph{satisfiable} regime $0\le\alpha\le\asat$, with high probability the solution space is a nonempty (random) subset $\SOL(\GG_N)\subseteq\set{\zro,\one}^N$. It is predicted by physicists \cite{krzakala2007gibbs} to undergo a precise series of sharp structural transitions as $\alpha$ increases between zero and $\asat$. Several of these predictions have now been supported by rigorous results: for example, we point to works on solution geometry \cite{achlioptas2006solution, achlioptas2008algorithmic, MR2823097}, the exact satisfiability threshold $\asat$ \cite{achlioptas2002asymptotic,coja2012catching,MR3440193}, the number of solutions \cite{ssz}, and associated inference problems \cite{MR3818090}. In particular it is known that $\asat = 2^{k-1}\log2 - O(1)$.

In this paper we consider the \bemph{unsatisfiable} regime $\alpha>\asat$, where with high probability the solution space $\SOL(\GG_N)$ is empty. It then becomes natural to study the \bemph{max-satisfiable value}
(or \bemph{ground state energy})
	\[\emin(\GG_N)
	\equiv \f1N \min
	\bigg\{
	\#\Big\{\textup{constraints violated by $\ux$}\Big\}
	: \ux\in\set{\zro,\one}^N\bigg\}\,.
	\]
The computer science literature on this problem has primarily focused on the regime where $\alpha=\alpha_N$ \bemph{diverges} in $N$. In this regime, an easy union bound gives $\emin(\GG_N) = (1-o_N(1))\alpha_N/2^{k-1}$, which allows for a simple phrasing of the so-called \emph{strong refutation} problem \cite{MR2121179}: is there an efficiently computable bound 
$\ealg(\GG_N)\le\emin(\GG_N)$
(for \emph{any} $\GG_N$) which satisfies
$\ealg(\GG_N)= (1+o_N(1))\alpha_N/2^{k-1}$ with high probability for \emph{random} $\GG_N$? An efficient (spectral) strong refutation algorithm  exists above $\alpha_N \approx N^{k/2-1}$ (\cite{MR2286509}, and extended by \cite{allen2015refute}). On the other hand, within a large family of convex programming algorithms (as defined by the \emph{sum-of-squares hierarchy}) it has been shown that many problems of this kind are solvable in subexponential but not polynomial time for $1\ll \alpha_N\ll N^{k/2-1}$ \cite{MR1832812,schoenebeck2008linear,MR3678176,kothari2017sum}.

In the regime where $\alpha$ does \bemph{not} diverge with $N$, very strong bounds on $\emin(\GG_N)$ are given by \cite{MR2295994}, as we will review below. However, the bounds are not quite precise enough to give information about the nature of the energy landscape. More recent results in the spin glass literature \cite{MR3854043,MR3783209} show that for $\alpha$ very large (roughly,  $\Omega(64^k)$) the max-sat value approaches the \bemph{mean-field} (complete graph) limit, which is given by a Parisi-type variational formula \cite{MR3737919} (in the physics literature see  \cite{leuzzi2001k,crisanti20023}). The solution of the mean-field variational formula is conjectured to be ``\bemph{full replica symmetry breaking}'' ($\frsb$), e.g.\ by analogy with the  zero-temperature Sherrington--Kirkpatrick model \cite{auffinger2017sk}. A stronger version of $\frsb$ has been shown in several recent works (in mean-field settings) to have algorithmic implications \cite{addario2018algorithmic,subag2018following,montanari2018optimization}. By contrast, results near the satisfiability threshold \cite{MR3440193,ssz} are consistent only with ``\bemph{one-step replica symmetry breaking}'' ($\onersb$). This is to say that as $\alpha$ increases from $\asat$ to $\infty$, the model must transition from $\onersb$ to $\frsb$; and one may even speculate further on whether the $N^{k/2-1}$ threshold in the algorithmic literature relates to a transition in the type of $\frsb$.

In this paper we study a phenomenon which is proposed in the physics literature as the first transition
beyond $\asat$ in the type of \textsc{rsb}.
It is predicted to occur at an explicit value $\aG$
\cite{MRT,KPW} (termed the \bemph{Gardner transition}, after \cite{gardner1985spin}) --- by a mechanism of \bemph{bug proliferation}, which we describe below. A simple consequence of this prediction is that the ground state energy would coincide with the $\onersb$ value $\eone$ up to $\aG$, but not thereafter. Our main result is a rigorous upper bound on this transition:

\begin{thm}\label{t:main}
For all $k\ge k_0$ (where $k_0$ is an absolute constant), if $\GG_N$ is an instance of random regular $k$-\textup{\textsc{nae-sat}} on $N$ variables subject to $N\alpha$ constraints
(Definition~\ref{d:naesat}), and $\E$ is expectation over $\GG_N$, then the quantity
	\beq\label{e:diff}
	\liminf_{N\to\infty}
	\bigg\{\E[\emin(\GG_N)] - \eone(\alpha)\bigg\}
	\eeq
is well-defined and nonnegative for all $\asat\le\alpha\le4^k/k$. It is strictly positive for all $\aG\le \alpha\le 4^k/k$ where $\aG\asymp 4^k/k^3$.
The formal characterizations of $\eone(\alpha)$ and $\aG$ appear below in Propositions~\ref{p:eone.defn}~and~\ref{p:gardner.defn}.\end{thm}

\noindent We will see soon (\S\ref{ss:intro.anp}) that the ground state energy is naturally parametrized as
	\[\emin= \f{\alpha (1-\pmax)}{2^{k-1}}\]
for $0\le \pmax\le1$. The first assertion of the theorem, the nonnegativity of \eqref{e:diff}, improves on the best previous upper bound on $\pmax$ by a factor  $1-\Omega(x)$ where the correction $\Omega(x)$ reflects the typical sizes of \bemph{clusters} of near-max-satisfiable configurations. We give the basic intuition for this correction in \S\ref{ss:intro.clustering}, and show in \eqref{e:alpha.onersb.correction} that in the regime $2^k k^2 \ll_k\alpha \ll_k 4^k/k$ we expect a correction $x\ge\Omega(1/d^{1/2})$. In \S\ref{ss:comparison.first.mmt} (Corollary~\ref{c:alpha.onersb.correction}) we state a more precise bound for all $\asat\le\alpha\le 4^k/k$.
 
The result relies on an abstract ``interpolation bound'' proved in \cite{ssz}, which was adapted from a combination of  prior works \cite{MR1930572, MR1957729, MR2095932, MR3161470, MR3256814}. Its main consequence, for our purpose, is stated in Proposition~\ref{p:interpolation} below; it involves an optimization over parameters $0\le y_1\le y_2$ and over a large space of probability measures $Q$. We prove Theorem~\ref{t:main} by direct analysis of the bound in a specific region of $(y_1,y_2,Q)$. This seems to bear some resemblance to approaches of \cite{MR3783558,auffinger2017sk},  although only at a high level. Our explicit choice of perturbation is based on the ``bug proliferation'' mechanism proposed by physicists \cite{MRT,KPW}, which we detail in the introductory section below. We leave as an open question to prove the matching lower bound, i.e., to show that $\lim_N\E[\emin(\GG_N)]=\eone(\alpha)$ for all $\asat\le\alpha\le\aG$.

In the remainder of this introductory section we present some guiding heuristics for this model, leading to the formal definitions of $\eone(\alpha)$ and $\aG$. Our discussion is based primarily on \cite{MR2295994}, together with the two papers from the physics literature that describe the bug proliferation mechanism: of the latter, one studies a similar model as here for $k=3,4$ \cite{MRT}, while the other studies the $q$-coloring model \cite{KPW}. We will focus on the combinatorial intuition for $k$-\textsc{nae-sat} which simplifies when $k$ is large. At the end of this section we outline the proof of Theorem~\ref{t:main}. Before proceeding further, we formally define the model:

\begin{dfn}[random regular \textsc{nae-sat}]
\label{d:naesat}
Let $d,k$ be positive integers, and assume $N$ is a positive integer such that $M = Nd/k$ is also integer. A \bemph{random $d$-regular $k$-\textsc{nae-sat}} instance on $N$ variables is encoded by a random bipartite graph $\GG_N$. The vertex set of $\GG_N$ is partitioned into $V=\set{v_1,\ldots,v_N}$ (\bemph{variables}) and $F=\set{a_1,\ldots,a_M}$ (constraints or \bemph{clauses}). The two sets $V,F$ are joined by a set $E$ of random edges, generated according to the ``configuration model'': give $d$ half-edges to each $v\in V$, give $k$ half-edges to each $a\in F$, then take a uniformly random matching between the $V$-incident and $F$-incident half-edges to form a total of $Nd=Mk$ edges. Note that the sampling procedure can result in multi-edges, so $\GG_N$ is more precisely a multi-graph. Finally, assign to each $e\in E$ an independent label $\lit_e$ sampled uniformly from $\set{\zro,\one}$. We denote the instance as $\GG_N=(V,F,E,\ulit)$. For $e\in E$ we write $v(e)$ for the incident variable, and $a(e)$ for the incident clause. We write
	{\setlength{\jot}{0pt}\begin{align*}
	\hDEL e
	\equiv \delta a(e)\setminus e
	&=
	\set{\textup{edges
	incident to $e$ through a clause}}
	\,,\\
	\dDEL e
	\equiv \delta v(e)\setminus e
	&=\set{\textup{edges 
	incident to $e$ through a variable}}
	\,.
	\end{align*}}%
For any variable $v\in V$ we write $\delta v$ for the ordered $d$-tuple of edges incident to $v$, and $\partial v$ for the ordered $d$-tuple of clauses $(a(e))_{e\in\delta v}$. For any clause $a\in F$ we write $\delta a$ for the ordered $k$-tuple of edges incident to $a$, and $\partial a$ for the ordered $k$-tuple of variables $(v(e))_{d\in\delta a}$. If $a\in F$ and $v\in V$ are neighbors joined by a single edge $e$ (as will most often be the case) then we write $e\equiv (av)$. Given a variable assignment $\ux\in\set{\zro,\one}^N$, a clause $a\in F$ is \bemph{violated} if and only if the $k$-tuple $(\lit_e \oplus x_{v(e)})_{e\in\delta a}$ is \bemph{all equal} (all $\zro$ entries or all $\one$ entries). A \bemph{solution} of $\GG_N$ is a variable assignment $\ux\in\set{\zro,\one}^N$ that violates no clauses.
\end{dfn}

\label{d:maxnaesat}
\begin{dfn}[energy lanscape and max-satisfiable value]
Given an instance $\GG_N$
generated as in Definition~\ref{d:naesat}, its \bemph{energy landscape} or \bemph{Hamiltonian} is simply the total count of violated clauses: for $\ux\in\set{\zro,\one}^N$,
	\beq\label{e:ham}
	\HH_N(\ux)
	= \sum_{a\in F} \HH_a(\ux)
	= \sum_{a\in F}
	\I\bigg\{
	\#\Big\{
	e\in\delta a
	: \lit_e \oplus x_{v(e)}=\one
	\Big\} \in\set{0,k}
	\bigg\}\,.
	\eeq
Note that $\HH_N$ is a random function on 
$\set{\zro,\one}^N$ determined by the instance $\GG_N$.
The solutions of $\GG_N$ are precisely the zeroes of $\HH_N$. The \bemph{max-satisfiable value}
(\bemph{ground state energy}) of $\GG_N$ is
	\[\emin(\GG_N)
	\equiv
	\f{\EMIN(\GG_N)}{N}
	\equiv
	\f1N
	\min\bigg\{\HH_N(\ux)
	:\ux\in\set{\zro,\one}^N
	\bigg\}\,.
	\]
Note that $0\le\emin(\GG_N)\le \alpha\equiv d/k$, and $\emin(\GG_N)$ is positive if and only if $\GG_N$ has no proper solutions. Let
	\begin{align*}
	\f{\alpha(1-\pinf(\alpha))}{2^k}
	\equiv \einf(\alpha)
	&\equiv\liminf_{N\to\infty}\E[\emin(\GG_N)]\\
	&\le\limsup_{N\to\infty}\E[\emin(\GG_N)]
	\equiv\esup(\alpha)
	\equiv \f{\alpha(1-\psup(\alpha))}{2^k}
	\,.
	\end{align*}
(If the two sides are equal we write
$\ee_\star(\alpha)\equiv\einf(\alpha)=\esup(\alpha)$.)
We also write  ``$\amax(\pgs)\le\alpha$'' to mean that $\pinf(\alpha')<\pgs$ for all $\alpha'>\alpha$, and similarly ``$\amax(\pgs)\ge\alpha$'' to mean that $\psup(\alpha')>\pgs$ for all $\alpha'<\alpha$.
\end{dfn}

\begin{rmk} Physicists predict that a broad family of random \textsc{csp}s (including (\textsc{nae}-)\textsc{sat}, proper coloring, and independent set)
exhibit qualitatively similar phase diagrams (\cite{krzakala2007gibbs} and refs.\ therein). The existing rigorous literature has proved different aspects of these predictions in different models, including for at least six closely related variants of the model specified in Definition~\ref{d:naesat}: namely,
random regular $k$-\textsc{nae-sat},
random $k$-\textsc{nae-sat},
random regular $k$-\textsc{sat},
random $k$-\textsc{sat},
random regular $k$-hypergraph bicoloring,
random $k$-hypergraph bicoloring.
Throughout this introduction, to simplify the discussion we will (nonrigorously) transfer all existing results to the setting of random regular $k$-\textsc{nae-sat}. It is not unreasonable to expect that a result proved in any of the other models can be reproved in random regular $k$-\textsc{nae-sat}, which is mathematically the simplest of all the six. Certainly, however, none of our formal results relies on this assumption.
\end{rmk}

To explain the basic intuitions underlying this paper, in \S\ref{ss:intro.anp} we review the first moment bound of \cite{MR2295994} in the setting of random regular $k$-\textsc{nae-sat}. We then explain in \S\ref{ss:intro.clustering} why the first moment bound is loose, and a rough heuristic correction. In \S\ref{ss:intro.perc} we explain that when $\alpha$ is not too large the heuristic correction is a reasonable approximation, but it should fail beyond some threshold $\alpha \asymp 4^k/k^3$. In \S\ref{ss:intro.comb} and \S\ref{ss:intro.tree.formula} we explain the more refined heuristic provided by the $\onersb$ combinatorial framework. This leads to the formal definitions of $\eone(\alpha)$ and $\aG$, in \S\ref{ss:intro.formulas} and \S\ref{ss:intro.gardner} respectively. Finally, in \S\ref{ss:intro.interp} we state the interpolation bound and describe the proof approach.

\subsection{First moment bound}\label{ss:intro.anp} 
Throughout this paper we write $f_{n,k}\asymp g_{n,k}$ to indicate that $C^{-1} \le f_{n,k}/g_{n,k} \le C$ for a constant $C$ not depending on $n,k$. We write $f\ll_k g$ to indicate that $\lim_{k\to\infty} f/g=0$. We parametrize
	\beq\label{e:cc.parametrization}
	\CC\equiv \f{\alpha}{2^{k-1}\log2}\,,
	\quad
	\ee\equiv \f{ \alpha(1-\pgs)}{2^{k-1}}\,.
	\eeq
To explain the above parametrization of $\ee$,
consider an instance $\GG_N$ of $d$-regular random $k$-\textsc{nae-sat}, and let $\HH_N$ be its Hamiltonian defined by \eqref{e:ham} above. For any \emph{fixed} $\ux\in\set{\zro,\one}^N$, the number of constraints that it violates is distributed as $\HH_N(\ux) \sim \Bin(M, 1/2^{k-1})$, so $\E\HH_N(\ux)/N = \alpha/2^{k-1}$. Therefore it is certainly the case that $\E[\emin(\GG_N)] \le \alpha/2^{k-1}$, so it is natural to parametrize energies as in \eqref{e:cc.parametrization}. Now, following \cite{MR2295994}, for any given energy level $0\le\ee\le\alpha$, and any $0<\eta\le 1$, we can consider
	\beq\label{e:def.anp.reweighting}
	X_{\ee,\eta}
	= \sum_{\ux\in\set{\zro,\one}^N}
	\I\bigg\{\f{\HH_N(\ux)}{N} \le \ee\bigg\}
	\f{\eta^{\HH_N(\ux)}}{\eta^{N\ee}}
	\ge
	\#\bigg\{\ux \in\set{\zro,\one}^N:
	\f{\HH_N(\ux)}{N} \le \ee\bigg\}
	\equiv Y_\ee\,.\eeq
If $\E$ is expectation over the random instance $\GG_N$, then
	\beq\label{e:ANP.first.moment.f.eta}
	\E X_{\ee,\eta}
	\le \E
	\sum_{\ux\in\set{\zro,\one}^N}
	\f{\eta^{\HH_N(\ux)}}{\eta^{N\ee}}
	= \f{2^N}{\eta^{N\ee}}
	\bigg\{
	1-\f{2}{2^k}+\f{2\eta}{2^k}
	\bigg\}^{N\alpha}
	= \exp\bigg\{
	N \textsf{f}_\eta(\alpha,\ee)
	\bigg\}\,.\eeq
If $\alpha,\ee$ are fixed, a stationary point of 
$\textsf{f}$ as a function of $\eta$ is given by
	\beq\label{e:anp.eta.star}
	\eta
	=\eta(\alpha,\ee)
	=\f{\ee(2^{k-1}-1)}{\alpha-\ee}
	=
	\f{(1-\pgs)(2^{k-1}-1)}
	{2^{k-1}-(1-\pgs)}
	\equiv \eta(\pgs)\,.
	\eeq
Setting $\textsf{f}_{\eta(\pgs)}(\alpha,\ee)=0$ gives the relation $\alpha=\aubd(\pgs)
	\equiv \CC(\pgs) \cdot 2^{k-1}\log2$ where
	\beq\label{e:rs.alpha.ubd}
	\CC(\pgs)
	\equiv\f1{(2^{k-1}-(1-\pgs))
	\log \f{2^{k-1}-(1-\pgs)}{2^{k-1}-1}
	+(1-\pgs)\log(1-\pgs)}
	\le
	\f1{\pgs+(1-\pgs)\log(1-\pgs)}\,.
	\eeq
Note that $\CC(\pgs)$ is strictly decreasing with respect to $\pgs$, since
	\[
	\f{d}{d\pgs}\f1{\CC(\pgs)}
	= \log \f{2^{k-1}-(1-\pgs)}{(2^{k-1}-1)(1-\pgs)}
	>0\,.
	\]
We have $\CC(\pgs)\uparrow\infty$ as $\pgs\downarrow0$, and
	\[
	\CC(1)
	= \f1{-2^{k-1} \log(1-1/2^{k-1})}
	= 1 - \Theta(2^{-k})\,.
	\] 
Therefore the inverse function is well-defined for all $\alpha\ge \CC(1) 2^{k-1}\log 2$, and we denote it as
	\beq\label{e:anp.e.lbd}
	\elbd(\alpha)
	\equiv \f{\alpha(1-\pubd(\alpha))}{2^{k-1}}
	\equiv (\aubd)^{-1}(\alpha)\,.
	\eeq
Since $\textsf{f}_\eta$ is decreasing in $\alpha$,
we conclude that $\textsf{f}_{\eta(\pgs)}(\alpha,\ee)<0$ for all $\alpha>\aubd(\pgs)$. For any such $\alpha$, Markov's inequality gives that
$\P(Y_\ee>0)\le \E Y_\ee\le \E X_{\ee,\eta}$ is exponentially small with respect to $N$. We can summarize the above as

\begin{lem}\label{l:anp.ubd}
If $\GG_N$ is random regular $k$-\textup{\textsc{nae-sat}} on $N$ variables subject to $N\alpha$ constraints, then
	\beq\label{e:anp.ubd}
	\liminf_{N\to\infty}
	\E[\emin(\GG_N)]
	=\liminf_{N\to\infty}
	\E\bigg[
	\min\bigg\{
	\ee\ge0
	: Y_\ee >0
	\bigg\}\bigg]
	\ge 
	\elbd(\alpha)
	\eeq
as defined by \eqref{e:anp.e.lbd}. (In the shorthand of Definition~\ref{d:naesat}, we have $\einf(\alpha)\ge\elbd(\alpha)$.)
\end{lem}

Lemma~\ref{l:anp.ubd} is the first moment bound from \cite{MR2295994}, transferred to the setting of regular \textsc{nae-sat}. Recalling \eqref{e:cc.parametrization}, if $\CC$ is large then the expression
\eqref{e:anp.e.lbd} 
for $\pubd$ can be approximated by
	\[
	\f1\CC
	= g(\pubd)
	= \pubd+(1-\pubd)\log(1-\pubd)
	= \f{(\pubd)^2}{2}
	\bigg\{1 + O(\pubd)\bigg\}
	\,,
	\]
so that $\pubd = (2/\CC)^{1/2} + O(1/\CC)$. We point out that \cite{MR2295994} studies the more difficult model of random $k$-\textsc{sat}, and their main result is a much more challenging lower bound, which is done by the second moment method. Translating their full result to our model would give
	\beq\label{e:anp.result}
	\aubd(\pgs)
	\bigg\{1 + O\bigg( \f{k}{2^{k/2}}\bigg)
	\bigg\}\le \amax(\pgs) \le \aubd(\pgs)\,.
	\eeq
We will not seek to rigorously prove the lower bound in \eqref{e:anp.result}, since we expect it to be easier than the lower bound already achieved by \cite{MR2295994}. The more interesting open problem is to establish that $\eone(\alpha)$ is tight for $\alpha\le\aG$.

\subsection{Clustering of near-max-satisfiable configurations}\label{ss:intro.clustering} We next describe the intuition for why the first moment bound  \eqref{e:anp.ubd} cannot be exactly sharp. Suppose for the sake of argument that it is. Let $\ee=\elbd$ and $\eta=\eta(\pubd)$ as above. Any $\ux\in\set{\zro,\one}^N$ that contributes to $X_{\ee,\eta}$ will be max-satisfiable, so it certainly must satisfy the weaker condition of being \bemph{locally max-satisfiable}, in the sense that flipping any single variable $x_v$ cannot decrease the number of violated constraints. Explicitly, let $F_\zro$ be the number of clauses incident to $v$ which are satisfied only if $x_v=\zro$:
	\[F_\zro
	= \#\bigg\{
	e\in\delta v:
	\#\Big\{g\in\hDEL(e):
	\lit_e \oplus \lit_g \oplus x_{v(g)} = \one
	\Big\} = k-1
	\bigg\}\,,
	\]
and similarly $F_\one$. The spin $x_v=\zro$ is locally max-satisfiable if and only if $F_\zro\ge F_\one$. 
Let $X_{\ee,\eta}(x,\ell_\zro,\ell_\one)$
denote the contribution to $X_{\ee,\eta}$ from configurations $\ux$ with $(x_v,F_\zro,F_\one)=(x,\ell_\zro,\ell_\one)$. By taking expectation only over the edge labels $\lit_e$ around the clauses neighboring $v$, we find
	\beq\label{e:first.mmt.zero}
	\E\bigg\{ X_{\ee,\eta}
	(\zro,\ell_\zro,\ell_\one)\bigg\}
	=C_N
	\Ind{\ell_\zro\ge\ell_\one}
	\binom{d}{\ell_\zro,\ell_\one}
	\bigg(
	1-\f4{2^k}
	\bigg)^{d-\ell_\zro-\ell_\one}
	\bigg(\f2{2^k}\bigg)^{\ell_\zro}
	\bigg(\f{2\eta}{2^k}\bigg)^{\ell_\one}
	\eeq
where $C_N$ is a factor not depending on $\ell_\zro,\ell_\one$, and
for any $a_1+\ldots+a_t\le b$ we abbreviate
	\[
	\binom{b}{a_1,\ldots,a_t}
	= \f{b!}{a_1! 
	\cdots a_t! (b-a_1-\ldots-a_t)!}\,.
	\]
Summing \eqref{e:first.mmt.zero} over $\ell_\zro\ge\ell_\one$, we find that the total expected contribution to $X_{\ee,\eta}$ from configurations with $x_v=\zro$ is
	\begin{align*}
	\E\bigg\{
	X_{\ee,\eta}(x_v=\zro)\bigg\}
	&= C_N
	\sum_{0\le\ell\le d}
	\binom{d}{\ell}
	\bigg(
	1-\f4{2^k}
	\bigg)^{d-\ell}
	\bigg(\f{2(1+\eta)}{2^k}\bigg)^\ell
	P_{\eta,\ell}\,,\\
	P_{\eta,\ell}
	&= \P\bigg(\Bin\bigg(\ell,
	\f{\eta}{1+\eta}
	\bigg)\le\f{\ell}{2}\bigg)\,.\end{align*}
Simply using the crude bound $1/2 \le P_{\eta,\ell} \le 1$ gives
	\[
	\E\bigg\{X_{\ee,\eta}(x_v=\zro)\bigg\} 
	\asymp C_N
	\bigg(
	1-\f{4}{2^k} + \f{2(1+\eta)}{2^k}
	\bigg)^d\,.
	\]
Now note that if $F_\zro=F_\one$ then variable $v$ is \emph{free}, meaning that flipping $x_v$ alone does not change the total number of violated constraints. Summing \eqref{e:first.mmt.zero} over $\ell_\zro=\ell_\one=\ell/2$ gives
	\[
	\E\bigg\{ X_{\ee,\eta}
	(x_v=\zro,\textup{$v$ is free})\bigg\}
	=C_N
	\sum_{\ell\textup{ even}}
	\binom{d}{\ell}
	\bigg(
	1-\f4{2^k}
	\bigg)^{d-\ell}
	\bigg(\f{4\eta^{1/2}}{2^k}\bigg)^\ell
	P_\ell
	\]
where $P_\ell = \P(\Bin(\ell,1/2)=\ell/2)\asymp1/\ell^{1/2}$. Now assume that $\CC$ is large, so  $\pgs$ is small and we see from \eqref{e:anp.eta.star} that $\eta\asymp1$.  Without the factor $P_\ell$, the above sum is dominated by $\ell \asymp d\eta^{1/2}/2^k\asymp d/2^k$. Accounting for $P_\ell$ results in
	\[\E\bigg\{X_{\ee,\eta}
	(x_v=\zro,\textup{$v$ is free})\bigg\}
	\asymp
	\f{C_N}{(d/2^k)^{1/2}}
	\bigg(
	1-\f{4}{2^k} + \f{4\eta^{1/2}}{2^k}
	\bigg)^d\,.
	\]
(A more careful version of this calculation appears in Section~\ref{s:recursions}.) This would suggest the typical fraction of variables that are free is something like
	\begin{align*}
	\pi_\free
	&\asymp
	\f{\E[X_{\ee,\eta}(x_v=\zro,\textup{$v$ is free})]}
	{\E[X_{\ee,\eta}(x_v=\zro)]}
	\asymp
	\f1{(d/2^k)^{1/2}}
	\bigg(
	1-\f{4}{2^k} + \f{4\eta^{1/2}}{2^k}
	\bigg)^d
	\bigg/\bigg(
	1-\f{4}{2^k} + \f{2(1+\eta)}{2^k}
	\bigg)^d\\
	&=
	\f1{(d/2^k)^{1/2}}
	\bigg(
	1 - \f{2(1-\eta^{1/2})^2/2^k}
	{1-4/2^k + 2(1+\eta)/2^k}
	\bigg)^d
	\,,
	\end{align*}
If we assume that $k^2 \ll_k \CC \ll_k 2^k/k$, then the above simplifies to
	\begin{align*}
	\pi_\free
	&\asymp \f1{(d/2^k)^{1/2}}
	\bigg(
	1 - \f{2}{2^k\CC}
	\bigg[1 + 
	O\bigg(\f1{\CC^{1/2}}\bigg)
	\bigg]
	\bigg)^d\\
	&=\f1{(d/2^k)^{1/2}}
	\exp\bigg\{
	-\f{d}{2^{k-1}\CC}
	\bigg[1 + 
	O\bigg(\f1{\CC^{1/2}}\bigg)
	\bigg]
	+ O\bigg(\f{d}{4^k\CC^2}\bigg)
	\bigg\}
	\asymp \f1{d^{1/2}}\,.\end{align*}
Suppose the configuration $\ux$ has order $N/d^{1/2}$ free variables. Suppose for simplicity that they \emph{do not interact}, meaning that \emph{flipping any subset of free variables does not change the total number of violated constraints}. We will examine the validity of this supposition in \S\ref{ss:intro.perc}, but we simply grant it for now. This would mean that for a typical max-satisfiable configuration $\ux$ we can find at least $2^{N\pi_\free}$  nearby configurations $\ux'$ with $\HH_N(\ux)=\HH_N(\ux')$. But this would mean $\E X_{\ee,\eta}\ge2^{N\pi_\free}$,
in contradiction with our choice of $\ee=\elbd$
and $\eta=\eta(\pubd)$ which ensures  that
$\E X_{\ee,\eta}$ is exponentially small in $N$. This suggests that $\elbd$ (or equivalently its inverse $\aubd$) cannot be tight bounds; our main theorem verifies this by establishing the lower bound $\einf(\alpha)>\elbd(\alpha)$. The above calculation suggests that $\exp\{N \textsf{f}_\eta(\alpha,\ee)\}$ overestimates the typical value of $X_{\ee,\eta}$ by at least a factor $2^{N\pi_\free}$ where $\pi_\free\asymp1/d^{1/2}$, which suggests, in the regime $k^2\ll_k\CC\ll_k 2^k/k$ (equivalently $2^k k^2 \ll_k \alpha\ll 4^k/k$), that
	\beq\label{e:alpha.onersb.correction}
	\amax(\pgs)
	\le\bigg\{
	1-\Omega\bigg(\f{1}{d^{1/2}}\bigg)
	\bigg\}
	\aubd(\pgs)\,.
	\eeq
In \S\ref{ss:comparison.first.mmt} (Corollary~\ref{c:alpha.onersb.correction})
we prove a rigorous bound which covers the full regime
$\asat\le\alpha\le 4^k/k$, and agrees with 
\eqref{e:alpha.onersb.correction}
for $k^2\ll_k\CC\ll_k 2^k/k$. In fact in this regime we conjecture the estimate $\Omega(1/d^{1/2})$ to be tight.

\subsection{Percolation of dependent free variables}
\label{ss:intro.perc}

We now revisit the above assumption that the free variables do not interact. Take a clause $a$ with no incident multi-edges (as will be the case for most clauses), and suppose it neighbors two free variables $v\ne w$. If the values of $\lit_{au}\oplus x_u$ for $u\in\partial a \setminus \set{v,w}$ are all $\zro$ or all $\one$, then $x_v$ and $x_w$ are \bemph{linked}, meaning they cannot both be arbitrarily flipped without increasing the number of violated constraints. For a free variable $v$, the number of linked free variables $w$ sharing a clause with $v$ is (on average, heuristically)
	\beq\label{e:percrate}
	r =
	\bigg\{(d-1)(k-1)\bigg\} \times
	\pi_\free \times
	\f1{2^{k-2}}
	\asymp \f{d^{1/2}k}{2^k}\,,
	\eeq
where the factor $(d-1)(k-1)$ accounts for the branching factor of the underlying graph $\GG_N$. We view the process of linked free variables as a dependent percolation on $\GG_N$ spreading at rate $r$ given by \eqref{e:percrate}. As long as the rate is small, corresponding to $d \ll_k 4^k/k^2$ or 
	\[
	\alpha \ll_k \f{4^k}{k^3}\,,
	\]
we would expect the percolation to be \bemph{subcritical}, in the sense that the
free subgraph
---
the subgraph of $\GG_N$ induced by 
free variables and linking clauses
---
is mostly a \bemph{forest} of $\tilde{O}(1)$-sized trees. Moreover, roughly a $(1-r)$-fraction of free variables should be \bemph{isolated} (not linked to any other frees), so for small $r$ it is a reasonable approximation to assume that none of the free variables interact.

As we detail in \S\ref{ss:intro.comb} below, in the context of the current problem, the $\onersb$ framework is simply a convenient combinatorial model for the free subgraph, which captures the effect of free variables on the total energy in a well-organized manner. It yields the prediction that the limiting ground state energy is \emph{exactly} $\eone(\alpha)$, where $\eone(\alpha)$ is an explicit function defined below in Proposition~\ref{p:eone.defn}. The threshold $\aG$, given formally by Proposition~\ref{p:gardner.defn}, is an explicit prediction of the exact percolation threshold for the $\onersb$ combinatorial model. The derivation of $\eone(\alpha)$ relies crucially on the assumption that the free subgraph \emph{is} essentially a forest, which should not be the case beyond $\aG$. This is the basic intuition for our main result which verifies that $\eone(\alpha)$ is indeed incorrect beyond $\aG$.

We remark that it is a much more challenging problem to obtain a sharper estimate of the asymptotic ground state energy in the regime $\alpha>\aG$. The main result that we know of was obtained for the random $k$-\textsc{sat} model \cite{MR3783209} (see also \cite{MR3854043}) by comparison with mean-field limits \cite{leuzzi2001k,crisanti20023}; from the discussion in \cite{MR3783209} the estimate requires roughly $\alpha\ge\Omega(64^k)$. A related result was obtained for the max-cut problem by \cite{MR3630296}, for random  graphs of large degree. It remains a difficult challenge to understand the regime between the mean-field (i.e., complete graph) limits and $\aG$.

Having laid out the basic intuitions for the model, we next proceed to define the $\onersb$ combinatorial framework. We emphasize that the $\onersb$ model itself is a heuristic, which plays no formal role in the proof of our main result. We introduce it because it is the quickest way to motivate the exact definitions of $\eone$ and $\aG$. We point to \cite{MR2518205} for an introductory account and further references on the $\onersb$ framework.

\subsection{Combinatorial model of near-max-satisfiable clusters}\label{ss:intro.comb}

Following our earlier discussion, we now restrict attention to the subspace $\LOC(\GG_N)\subseteq\set{\zro,\one}^N$ of configurations that are  locally max-satisfiable.
Define a graph on vertex set $\LOC$ by putting an edge between $\ux$ and $\ux'$ if and only if they differ in a single coordinate and $\HH_N(\ux)=\HH_N(\ux')$. A \bemph{(locally max-satisfiable) cluster} is any subset $\omega\subseteq\LOC$ that constitutes a maximal connected component in that graph. The $\onersb$ heuristic models a cluster as follows:

\begin{dfn}[warning configurations]
Suppose $\GG_N=(V,F,E,\ulit)$ is any $d$-regular $k$-\textsc{nae-sat} problem instance. A \bemph{warning configuration} on $\GG_N$ is an element $\uw \in \set{\zro,\one,\free}^{2E}$ which assigns a pair $\ww_e\equiv(\dw_e,\hw_e)$ to each edge $e\in E$, satisfying conditions that we now specify. We take the convention throughout that  $x\oplus\free\equiv\free$. Define
	\[
	\ell_x(\hw_1,\ldots,\hw_{d-1})
	\equiv 
	\#\bigg\{1\le i\le d-1 : \hw_i=x\bigg\}.
	\]
Then $\uw$ is a valid warning configuration if and only if it satisfies variable relations
	\beq\label{e:wp.var}
	\dw_e
	=\dWP\Big(\hw_g : g\in\dDEL(e)\Big)
	=\begin{cases}
	\zro &
	\ell_\zro(\hw_g: g\in\dDEL(e))
	>\ell_\one(\hw_g : g\in\dDEL(e))\,,
	\\
	\one &
	\ell_\zro(\hw_g : g\in\dDEL(e))
	<\ell_\one(\hw_g : g\in\dDEL(e))
	\,,\\
	\free &
	\ell_\zro(\hw_g : g\in\dDEL(e))
	=\ell_\one(\hw_g : g\in\dDEL(e))\,,
	\end{cases}
	\eeq
as well as clause relations
	\beq\label{e:wp.clause}
	\hw_e
	=\hWP\Big(\dw_g : g\in\hDEL(e)\Big)
	=\begin{cases}
	\zro &
	\lit_g\oplus\dw_g=\lit_e\oplus\one
	\textup{ for all }g \in \hDEL(e)\,,\\
	\one &
	\lit_g\oplus\dw_g=\lit_e\oplus\zro
	\textup{ for all }g \in \hDEL(e)\,, \\
	\free &\textup{otherwise}\,,
	\end{cases}\eeq
for all $e\in E$.  We may write $\hWP\equiv\hWP_{k-1}$ and $\dWP\equiv\dWP_{d-1}$ to emphasize the number of arguments. Given $\uw$ let
	\[\UETA
	\equiv
	\UETA(\uw)\equiv
	\bigg(
	\dWP_d(\hw_e : e\in\delta v)
	\bigg)_{v\in V}
	\in\set{\zro,\one,\free}^N\,.
	\]
If $\#\set{v\in V : \ETA_v = \free} \le N/k^2$ then we say that $\uw$ is \bemph{near-frozen}.
\end{dfn}

Under the $\onersb$ heuristic, there is essentially a bijective correspondence
	\beq\label{e:bij}
	\bigg\{\begin{array}{c}
	\textup{locally max-satisfiable}\\
	\textup{clusters $\omega\subseteq\LOC
		\subseteq\set{\zro,\one}^N$}
	\end{array}
	\bigg\}
	\leftrightarrow
	\bigg\{
	\begin{array}{c}
	\textup{near-frozen warning}\\
	\textup{configurations
	$\uw\in\set{\zro,\one,\free}^{2E}$}
	\end{array}
	\bigg\}
	\eeq
between clusters $\omega$ and near-frozen warning configurations $\uw$. A loose characterization of the correspondence is that
$\UETA\equiv\UETA(\uw)$ encodes the smallest subcube of $\set{\zro,\one}^N$ containing $\omega$: $\ETA_v\in\set{\zro,\one}$ if and only if $x_v=\ETA_v$ for all $\ux\in\omega$, and $\ETA_v=\free$ 
if and only if $x_v$ takes both values $\set{\zro,\one}$. A more precise interpretation is that
	{\setlength{\jot}{0pt}\begin{align*}
	\dw_e
	&= \textup{\bemph{variable-to-clause
	warning} along $e$}\\
	&=\textup{locally optimal choice within $\omega$ of $x_{v(e)}$ in absence of edge $e$,}\\
	\hw_e
	&= \textup{\bemph{clause-to-variable warning} along $e$}\\
	&=\textup{locally optimal choice within $\omega$ of  $x_{v(e)}$ ``in absence of'' edges $\dDEL(e)$,}
	\end{align*}}%
where $\free$ means that both spins $\set{\zro,\one}$ are locally optimal. Under this interpretation, the $\dw,\hw$ must then satisfy local consistency relations, which are the so-called \bemph{warning propagation} (\textsc{wp}) equations \eqref{e:wp.clause} and \eqref{e:wp.var}.
The near-frozen restriction rules out configurations such as $\uw=\free^{2E}$ (all messages $\free$) which we do not expect to correspond to any actual cluster.

\subsection{Tree formula for the max-satisfiable value}
\label{ss:intro.tree.formula}

To give an explicit calculation, let $\tree=(V',F',E')$ be a finite bipartite tree (representing an $O(1)$-sized subgraph of $\GG_N$) with variables at its leaves. Say $\tree$ has a \bemph{frozen boundary}, in the sense  that $\dw_e\in\set{\zro,\one}$ is fixed at every leaf edge $e$. By applying the maps $\hWP,\dWP$ recursively inwards from the leaves, we see that there is exactly one valid warning configuration $\uw$ on $\tree$ that is consistent with the boundary condition.  Let $\EMIN(\tree)$ be the minimum number of clauses violated by any configuration $\ux\in\set{\zro,\one}^{V'}$ with $x_{v(e)}=\dw_e$ at the leaves. We next explain that $\EMIN(\tree)$ can be computed by a simple  dynamic-programming-type method.

Let $E''$ be the set of non-leaf edges of $\tree$.
For any $e\in E''$ we let
$\hat{\tree}_e$ be the component containing $a(e)$ in $\tree\setminus\dDEL(e)$, and let $\dot{\tree}_e$ be the component containing $v(e)$ in $\tree\setminus e$. Let $\hat{\EE}_e=\EMIN(\hat{\tree}_e)$ and $\dot{\EE}_e=\EMIN(\dot{\tree}_e)$. If $V''$ denotes the non-leaf variables of $\tau$,
around any $v\in V''$ we have
	\beq\label{e:sp.var.penalty}
	\EMIN(\tau)
	= \dph(\uhw_{\delta v})
	+\sum_{e\in\delta v}\hat{\EE}_e\,,\quad
	\dph(\uhw_{\delta v})
	\equiv
	\min\bigg\{
	\ell_\zro(\uhw_{\delta v}),
	\ell_\one(\uhw_{\delta v})
	\bigg\}\,.\eeq
Similarly, around any clause $a\in F'$, we have
	\beq\label{e:sp.clause.penalty}
	\EMIN(\tau)
	=\hph(\vec{\lit\oplus\dw}_{\delta a})
	+
	\sum_{e\in\delta a}\dot{\EE}_e\,,\quad
	\hph(\vec{\lit\oplus\dw}_{\delta a})
	\equiv
	\I\bigg\{
	\hWP_k(\vec{\lit\oplus\dw}_{\delta a})\ne\free
	\bigg\}\,.\eeq
We sometimes write 
 $\dph\equiv\dph_d$
 and $\hph\equiv\hph_k$ to emphasize the number of arguments.
Finally, for any $e\in E''$ we have
	\beq\label{e:sp.edge.penalty}
	\EMIN(\tau)
	=\eph(\ww_e)+ \dot{E}_e+\hat{E}_e\,,\quad
	\eph(\ww_e)\equiv\I\bigg\{
		\dw_e\oplus\hw_e=\one\bigg\}
	\,.\eeq
By summing over the internal vertices and subtracting over the internal edges, we arrive at
	\beq\label{e:tree.ground.state.formula}
	\sum_{v\in V''}\dph(\uhw_{\delta v})
	+\sum_{a\in F'}\hph
	(\vec{\lit\oplus\dw}_{\delta a})
	-\sum_{e\in E''} \eph(\ww_e)
	= \bigg(\#V''+\#F'-\#E''\bigg)\EMIN(\tau)
	=\EMIN(\tau)\,,
	\eeq
where the last equality uses that $\tau$ is a tree. Thus the max-satisfiable value of a tree with frozen boundary is a sum of local functionals $\dph$, $\hph$, $\eph$ of the warning configuration.

The $\onersb$ heuristic further assumes that for near-frozen warning configurations, the entire graph $\GG_N=(V,F,E,\ulit)$ can essentially be carved into trees with frozen boundaries. (In reality, even in the regime where free variables do not percolate,
a typical warning configuration may contain a bounded number of small cycles of free warnings which do not admit a tree decomposition. However these few cycles should only affect the number of violated clauses by $O(1)$, so can be ignored in the heuristic analysis.)
Then, by summing \eqref{e:tree.ground.state.formula} over the components of the tree decomposition, we conclude that  $\uw$ corresponds to a cluster $\omega\subseteq\set{\zro,\one}^N$ at energy level
	\beq\label{e:ground.state.energy.bethe}
	\EMIN(\omega;\GG_N)
	\equiv \min\bigg\{\HH_N(\ux)
		: \ux\in\omega\bigg\}
	=
	\vph(\uw)
	\equiv\sum_{v\in V}\dph(\uhw_{\delta v})
	+\sum_{a\in F}\hph
	(\vec{\lit\oplus\dw}_{\delta a})
	-\sum_{e\in E} \eph(\ww_e)
	\,.
	\eeq
This is the main advantage of the $\uw$ encoding; it allows us to read off $\EMIN(\omega;\GG_N)$ as a sum of local terms.

\subsection{Explicit 1RSB prediction}\label{ss:intro.formulas}

Going back to the bijection \eqref{e:bij}, we can take a parameter $y\ge0$ and consider
	\beq\label{e:ytilted}
	\mathfrak{Z}(y)
	=\sum_\omega
		\exp\bigg\{ -y \EMIN(\omega;\GG_N)\bigg\}
	=\sum_{\uw}
	\exp\bigg\{-y\vph(\uw)\bigg\}
	\eeq
where the first sum goes over clusters,
while the second sum goes over near-frozen warning configurations. The corresponding probability measure on warning configurations is given by
	\beq\label{e:wp.model}
	\mu_y(\uw)
	= \f{\exp\{-y\vph(\uw)\}}{
	\mathfrak{Z}(y)}\,.
	\eeq
This is sometimes called the \bemph{survey propagation} or $\SP_y$ model, and can be viewed as a refinement of the reweighting $\eta^{\HH_N(\ux)}$ discussed in \S\ref{ss:intro.anp}. The ``lifting'' from $\eta^{\HH_N(\ux)}$ to $e^{-y\vph(\uw)}$ represents one level of \bemph{replica symmetry breaking}. The 
\bemph{$\onersb$ solution} to the original model is given by the \bemph{replica symmetric} solution to the ``lifted'' model \eqref{e:wp.model}. This sometimes goes by the name of \bemph{survey propagation} (\textsc{sp}). In particular, the  \bemph{$\onersb$ (\textsc{sp}) equations} are simply the \bemph{replica symmetric} or \bemph{belief propagation} (\textsc{bp}) equations for the lifted model. They can be defined as a pair of mappings on the space
	\beq\label{e:MM.symmetric}
	\MM
	\equiv\bigg\{
	\textup{probability measures
	$q$ on $\set{\zro,\one,\free}$
	satisfying $q(\zro)=q(\one)$}
	\bigg\}\,.
	\eeq
The clause survey propagation takes $\dq\in\MM$ and outputs
	\beq\label{e:sp.clause}
	[\hSP_y(\dq)](\hw)
	= 
	\sum_{\udw}
	\I\bigg\{
	\hw=\hWP_{k-1}(\udw)
	\bigg\}
	\prod_{i=1}^{k-1} \dq(\dw_i)\,,
	\eeq
where the sum goes over $\udw\in\set{\zro,\one,\free}^{k-1}$, and $\hSP_y(\dq)$ is a probability measure on $\set{\zro,\one,\free}$, and in fact
$\hSP_y(\dq)\in\MM$.
The variable survey propagation takes $\hq\in\MM$ and outputs
	\beq\label{e:sp.var}
	[\dSP_y(\hq)](\dw)
	=\f1{\dz}
	\sum_{\uhw}
	\I\bigg\{\dw=\dWP_{d-1}(\uhw)\bigg\}
	\exp\bigg\{-y\dph_{d-1}(\uhw)
	\bigg\}
	\prod_{i=1}^{d-1} \hq(\hw_i)\,,
	\eeq
where the sum is over $\uhw\in\set{\zro,\one,\free}^{d-1}$, and $\dz$ is the normalization such that
$\dSP_y(\hq)\in\MM$. Let $\SP_y\equiv\dSP_y\circ\hSP_y$. Now, recalling \eqref{e:cc.parametrization}, we hereafter restrict consideration to parameters $y\ge0$ satisfying
	\beq\label{e:y.restriction}
	\gamma \equiv 2\CC \bigg( 1 - \f1{e^{y/2}}\bigg)^2
	\asymp 1\,.
	\eeq
Note $\gamma\asymp \CC\min\set{1,y^2}$.
If $\CC$ is large then \eqref{e:y.restriction} forces $y\asymp 1/\CC^{1/2}$. If $\CC\asymp1$ then it only forces that $y\ge\Omega(1)$. Define
	\begin{align} \label{e:MM.bullet}
	\MMnf
	&\equiv
	\bigg\{
	q\in\MM : q(\free) \le \f1{k^2}
	\bigg\}\,,\\ \label{e:MM.gamma}
	\MM^\gamma
	&\equiv\bigg\{
	q\in\MMnf : 
	q(\free) 
	\asymp \f{2^{-k\gamma/2}}{
	(\max\set{\CC k e^{-y/2},1})^{1/2}}
	\bigg\}\,.
	\end{align}
We prove the following result on fixed points of the $\SP_y$ recursion:

\begin{ppn}[proved in Section~\ref{s:recursions}]\label{p:sp} Suppose $\alpha=\CC 2^{k-1}\log2$ with $\asat\le\alpha \le 4^k/k$, and suppose $y\ge0$ satisfies \eqref{e:y.restriction}. Then in the set $\MMnf$ there is a unique $\dq_y$ satisfying the fixed-point equation $\dq_y=\SP_y(\dq_y)$. It must further lie in the smaller domain $\MM^\gamma$.
\end{ppn}

\noindent Let $\dq=\dq_y$ be as given by Proposition~\ref{p:sp},
and denote $\hq\equiv\hq_y\equiv\hSP_y(\dq_y)$.
Recall the local functionals 
$\dph$, $\hph$, $\eph$ from
\eqref{e:sp.var.penalty},
\eqref{e:sp.clause.penalty},
\eqref{e:sp.edge.penalty}.
We can define
three probability measures ---
$\dot{\nu}_y$ over $\uhw\in\set{\zro,\one,\free}^d$, $\hat{\nu}_y$ over $\udw\in\set{\zro,\one,\free}^k$, and lastly $\bar{\nu}_y$ over $\ww=(\dw,\hw)\in\set{\zro,\one,\free}^2$ --- as follows:
	\begin{align}
	\label{e:sp.var.marginal}
	\dot{\nu}_y(\uhw)
	&=\f1{\dfz_y(\hq)}
	\exp\bigg\{
	-y\dph_d(\uhw)
	\bigg\}
	\prod_{i=1}^d\hq(\hw_i)\,,\\
	\label{e:sp.clause.marginal}
	\hat{\nu}_y(\udw)
	&= \f1{\hfz_y(\dq)}
	\exp\bigg\{-y\hph_k(\dw)\bigg\}
	\prod_{i=1}^k\dq(\dw_i)\,,\\
	\label{e:sp.edge.marginal}
	\bar{\nu}_y(\ww)
	&= 
	\f1{\efz_y(\dq,\hq)}
	\exp\bigg\{
	-\eph(\dw,\hw)
	\bigg\} \dq(\dw)\hq(\hw)\,.
	\end{align}
Under the \textsc{sp} heuristic, the local marginals of the measure \eqref{e:wp.model} are approximately given by the $\nu$: for instance,
	\[
	\mu_y\bigg(\uw\in\set{\zro,\one,\free}^{2E}:
		\uhw_{\delta v}=\uhw\bigg)
	\approx \dot{\nu_y}(\uhw)\,.
	\]
The corresponding energy level can be obtained by averaging \eqref{e:ground.state.energy.bethe} with respect to the $\nu$: this gives
	\beq\label{e:onersb.ee.of.y}
	\ee(y)
	= \sum_{\uhw} \dph_d(\uhw)
		\dot{\nu}_y(\uhw)
	+ \alpha
	\sum_{\udw}
	\hph_k(\udw)
		\hat{\nu}_y(\udw)
	- d
	\sum_{\ww} \eph(\ww)
		\bar{\nu}_y(\ww)\,.
	\eeq
The \textsc{sp} heuristic further predicts that
$N^{-1}\log\mathfrak{Z}(y)$ converges (for a suitable range of $y$) to the replica symmetric formula,
	\beq\label{e:FF.free.energy}
	\FF(y)
	=\log\dfz_y(\hq)
	+\alpha\bigg\{\log\hfz_y(\dq)
	-k\log\efz_y(\dq,\hq)
	\bigg\}\,,
	\eeq
where $\dfz_y$, $\hfz_y$, and $\efz_y$ are the normalizing constants from \eqref{e:sp.var.marginal}, \eqref{e:sp.clause.marginal},  and \eqref{e:sp.edge.marginal}.
Now, returning to \eqref{e:ytilted}, suppose that we had an ``energetic complexity function'' 
function $\Sigma$ such that
	\[
	\mathfrak{Y}_\ee
	\equiv\#\bigg\{\uw 
	: \vph(\uw) \approx N\ee \bigg\}
	\approx\E\mathfrak{Y}_\ee
	\approx\exp\bigg\{N\Sigma(\ee)\bigg\}
	\]
where the interpretation for $\Sigma(\ee)<0$ 
is that $\E\mathfrak{Y}_\ee$ is exponentially small with respect to $N$ so $\mathfrak{Y}_\ee=0$ whp.
Then we would expect
	\[
	\exp\bigg\{
	N\FF(y)
	\bigg\}\approx
	\mathfrak{Z}(y)\approx
	\exp\bigg\{
	N\max_\ee\Big\{
	\Sigma(\ee)-y\ee
	: \Sigma(\ee)\ge0
	\Big\}
	\bigg\}\,,
	\]
that is to say, given $\Sigma$ we can obtain $\FF$ by taking the Legendre dual. Of course, we are in the opposite situation: we already obtained explicit expressions \eqref{e:onersb.ee.of.y}~and~\eqref{e:FF.free.energy} for $\ee(y)$ and $\FF(y)$, but we do not know $\Sigma$. We therefore formally define the \bemph{energetic complexity function} as
$\SIGMA(y)= \FF(y) + y \ee(y)$. (While the informal complexity $\Sigma$ is a function of $\ee$, the formal complexity $\mathfrak{S}$
is a function of $y$.) Recall \eqref{e:y.restriction} and let
	\beq\label{e:Gamma.function}
	\Gamma(y)
	\equiv
	\CC\bigg( 1-\f{1+y}{e^y} \bigg)\,.
	\eeq
It is straightforward to verify that $\gamma(y)/2 \le \Gamma(y)\le\gamma(y)$ for all $y\ge0$; see Figure~\ref{f:gammas}. For small $y$
(corresponding, via \eqref{e:y.restriction}, to large $\CC$) we have 
	\[\Gamma(y)=\gamma(y)
	\bigg\{ 1-O(y)\bigg\}\,.\]
For $y\ge\Omega(1)$ (corresponding, via \eqref{e:y.restriction}, to $\CC\asymp1$) we have instead
	\[2\Gamma(y)=\gamma(y)
	\bigg\{ 1+O(e^{-y/2})\bigg\}\,.\]
The following proposition formally defines the 
$\onersb$ formula.

\begin{ppn}[proved in Section~\ref{s:recursions}]
\label{p:eone.defn} Suppose $k\ge k_0$ and 
$\asat\le\alpha\le 4^k/k$; and denote $\CC = \alpha/(2^{k-1}\log2)$. Then, on the range of $y$ satisfying \eqref{e:y.restriction}, the function $\SIGMA(y)= \FF(y) + y \ee(y)$ is smooth and strictly decreasing, with a unique root $y_\star=y_\star(\alpha)$. With $\Gamma(y)$ as defined by \eqref{e:Gamma.function}, this root satisfies the estimate
	\beq\label{e:Gamma.star}
	\Gamma(y_\star) = 1 + O\bigg(
		\f1{e^{\Omega(k)}}
		\bigg)\,.\eeq
The \bemph{$\onersb$ ground state energy} 
can be defined as
	\beq\label{e:def.eone}
	\eone(\alpha)\equiv\ee(y_\star(\alpha))
	=- \inf \bigg\{ \f{\FF(y)}{y}
	:y\ge0\,,
	y \textup{ satisfies }\eqref{e:y.restriction}
	\bigg\}
	\eeq
(the equivalence of the last two quantities will be proved).
\end{ppn}

\begin{figure}[h]
\centering
\begin{subfigure}[b]{0.45\textwidth}
\centering
\includegraphics[height=1.8in]{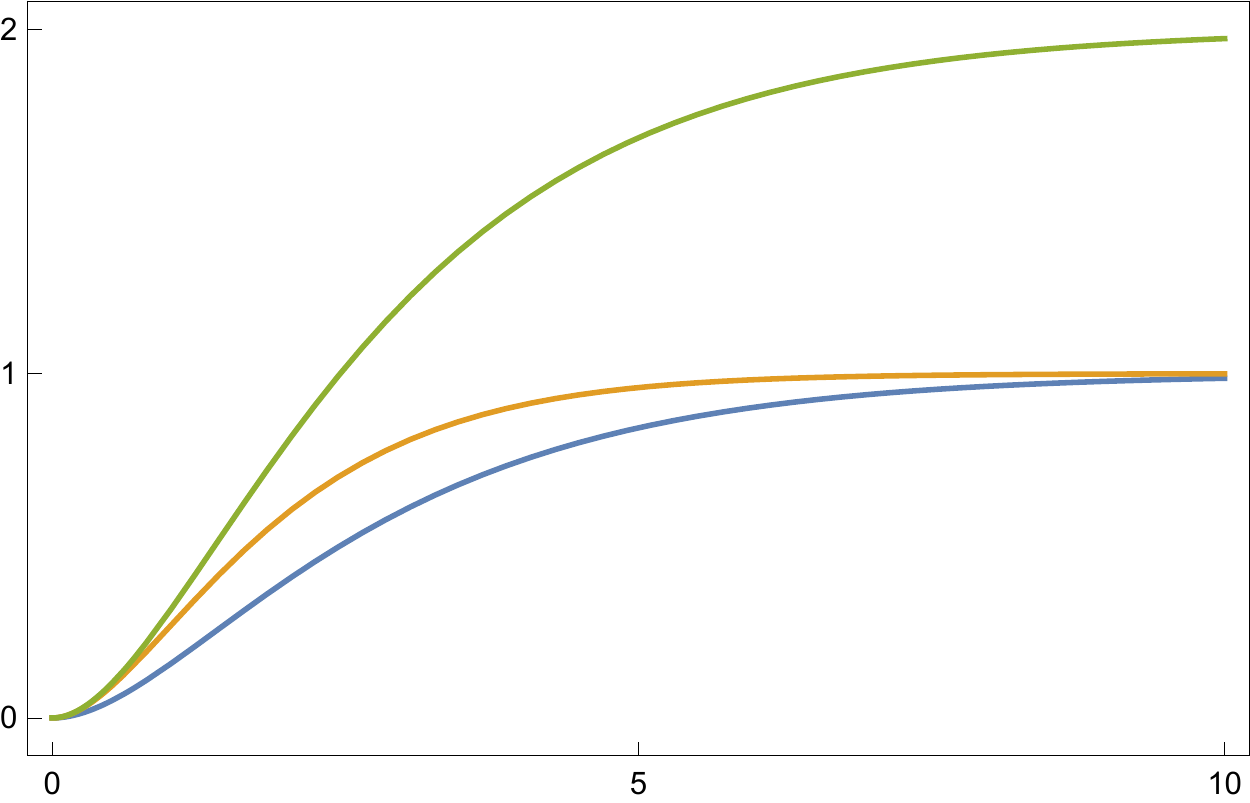}
\caption{$\gamma(y)/2\le\Gamma(y)\le\gamma(y)$
for all $y\ge0$.\\
$\Gamma(y_\star)\doteq1$
for both \textsc{rs} and $\onersb$ solutions
\eqref{e:Gamma.star}.}
\label{f:gammas}
\end{subfigure}
\quad
\begin{subfigure}[b]{0.45\textwidth}
\centering
\includegraphics[height=1.8in]{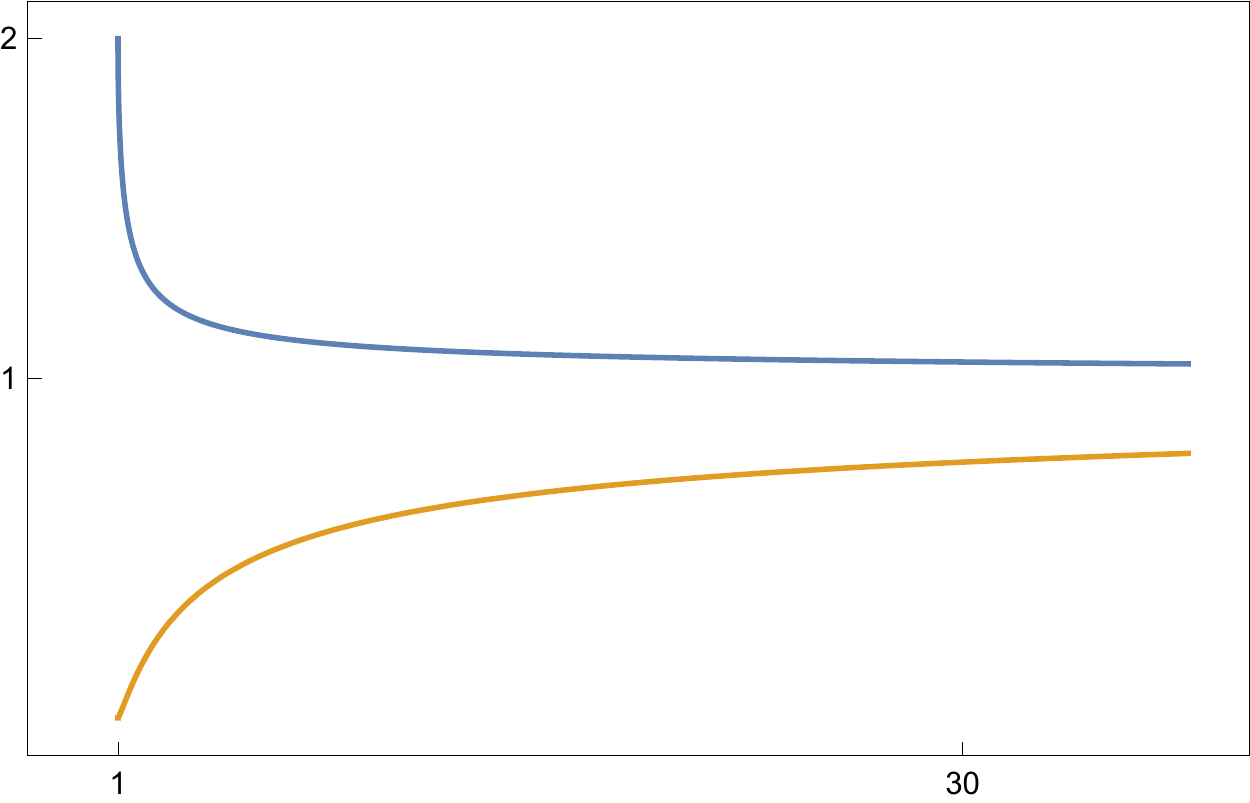}
\caption{Upper curve: $\gamma(y_\star)$ 
as a function of $\CC\ge1$.\\
Lower curve: $\exp(-y_\star)$
as a function of $\CC\ge1$.}
\label{f:gamma.of.c}
\end{subfigure}
\caption{Approximate parameters of the $\onersb$ solution. At clause density $\alpha=\CC 2^{k-1}\log2$,
the max-satisfiable value is $\ee=\alpha(1-\pgs)/2^{k-1}$ where $1-\pgs\doteq\eta\doteq e^{-y}$ is given approximately by the lower curve in panel \textsc{(b)}. At this precision it is consistent with the
replica symmetric (\textsc{rs}) solution (cf.\ the estimate of  \cite{MR2295994}). 
A more precise comparison between \textsc{rs} and $\onersb$ is given by Corollary~\ref{c:alpha.onersb.correction}.}
\end{figure}

\noindent To show that $\SIGMA$ is decreasing in $y$, we will in fact show that $\SIGMA'(y)=y\ee'(y)=-y\FF''(y)<0$. In \S\ref{ss:variational.discussion} (Remark~\ref{r:stationarity}) we review the physical interpretation of $\FF''(y)$.

We remark that the estimate \eqref{e:Gamma.star} is a rather lossy approximation of $\eone(\alpha)$. In fact, on its own it does not carry more information than the first moment \cite{MR2295994} bound: observe from \eqref{e:anp.eta.star}  that $\eta=\eta(\pgs)=(1-\pgs)[1+O(\pgs/2^k)]$. Substituting into \eqref{e:rs.alpha.ubd} gives
	\beq\label{e:log.eta.satisfies.conds}
	\f1{\CC(\pgs)}
	=\bigg( 1-\eta+\eta\log\eta\bigg)
		\bigg\{1+ O\bigg(\f1{e^{\Omega(k)}}\bigg)\bigg\}
	=\bigg( 1 - \f{1+y}{e^y}\bigg)
	\bigg\{1+ O\bigg(\f1{e^{\Omega(k)}}\bigg)\bigg\}\,,
	\eeq
simply by taking $y=-\log\eta$. Thus more care is needed to obtain a comparison such as \eqref{e:alpha.onersb.correction} with the first moment. Towards this end, let us comment briefly on what Proposition~\ref{p:eone.defn} implies for $\dq(\free)$.
Recall from Proposition~\ref{p:sp} that $\dq=\dq_y\in\MM^\gamma$, meaning (see \eqref{e:MM.gamma}) that
	\[\dq(\free)\asymp \f{2^{-k\gamma/2}}{
	(\max\set{\CC k e^{-y/2},1})^{1/2}}
	\asymp
	\f1{2^{k\gamma/2}}
	\bigg(\min\bigg\{
	\f{e^{y/2} \min\set{1,y^2}}{k}
	,1\bigg\}\bigg)^{1/2}
	\le O\bigg(
	\f1{2^{k\gamma/2}}
	\bigg)
	\]
for $\gamma=\gamma(y)$. It follows from \eqref{e:y.restriction} that $\CC k e^{-y/2}= O(1)$ if and only if $y\ge 2\log k - O(1)$. For such $y_\star\ge2\log k-O(1)$, the estimate \eqref{e:Gamma.star} implies
	\[
	\CC=
	\Gamma(y_\star)
	\bigg(1-\f{1+y_\star}{e^{y_\star}}\bigg)^{-1}
	=	1 - \f{O(\log k)}{k^2}\,,
	\]
meaning $\alpha$ is only slightly above the satisfiability threshold. In this regime
	\[
	\f{k\gamma(y_\star)}{2}
	=\f{k\gamma(y_\star)}{2\Gamma(y_\star)}
	\Gamma(y_\star)
	=k\bigg\{1+O\bigg(\f1{e^{y_\star/2}}\bigg)\bigg\}
	\bigg\{1+O\bigg(
		\f1{e^{\Omega(k)}}\bigg)\bigg\}
	=k + O(1)\,,
	\]
so $\dq(\free) \asymp 2^{-k}$. This is consistent with estimates slightly below the satisfiability threshold obtained by \cite{MR3440193}.

We now discuss $y\le 2\log k + O(1)$. In general, for any fixed $\CC$ the value $\Gamma(y)$ is strictly increasing in $y$, therefore $y_\star$ must be roughly decreasing with $\CC$ (modulo the error in the estimate \eqref{e:Gamma.star}). The ratio $\gamma(y)/\Gamma(y)$ is a function of $y$ alone, and is increasing in $y$.
Therefore, as $\CC$ increases, $\gamma(y_\star)\doteq\gamma(y_\star)/\Gamma(y_\star)$ decreases smoothly, from $\gamma\doteq 2$ to $\gamma\doteq 1$ (Figure~\ref{f:gamma.of.c}). For $\Omega(1/k^2)\le y\le 2\log k + O(1)$
we have $\dq(\free) = k^{O(1)}/2^{k\gamma/2}$,
which is roughly increasing as $\CC$ decreases if we ignore the $k^{O(1)}$ factor. Finally, if $y = O(1/k^2)$ then
	\[
	\f{k\gamma(y_\star)}{2}
	=\f{k\gamma(y_\star)}{2\Gamma(y_\star)}
		\Gamma(y_\star)
	=\f{k}{2} \bigg\{ 1 + O(y) \bigg\}
	\bigg\{1+O\bigg(
		\f1{e^{\Omega(k)}}\bigg)\bigg\}
	= \f{k}{2} + O(1)\,,
	\]
so in this regime we have
	\beq\label{e:sqrt.d}
	\dq(\free)
	\asymp \f{y}{(2^k k)^{1/2}}
	\asymp \f1{(2^k \CC k)^{1/2}}
	\asymp \f1{d^{1/2}}\,,
	\eeq
which matches with \eqref{e:alpha.onersb.correction}.

\subsection{Explicit Gardner threshold}
\label{ss:intro.gardner}

We now describe the exact predicted threshold $\aG$ for the stability of the $\onersb$ solution. Recall the loose calculation \eqref{e:percrate} of the branching rate of linked frees. One can refine this by considering the rate of ``bug proliferation'' \cite{MRT,KPW} in the warning model: if a warning incoming to a vertex is changed, it may change an outgoing warning, and one can calculate the branching rate of this process. Explicitly, let
	\[
	\bigg(
	\dw_{ai}:	2\le a\le d,
				2\le i \le k
	\bigg) 
	\equiv (\dw_j)_{1\le j\le \branch}
	\equiv \dw_{1:\branch}
	\in \set{\zro,\one,\free}^{\branch}
	\]
where we have abused notation and made the identification $\dw_{ai}\equiv \dw_{(a-2)(k-1)+(i-1)}$. Recall the mappings $\dWP$ and $\hWP$ defined in \eqref{e:wp.var} and \eqref{e:wp.clause}. 
Define $\hw_a\equiv\hWP(\dw_{a,2},\ldots,\dw_{a,k})$ for each $2\le a\le d$, and then let
	{\setlength{\jot}{0pt}\begin{align*}
	\WP(\dw_{1:\branch})
	&\equiv\dWP(\hw_2,\ldots,\hw_d)\\
	\vph(\dw_{1:\branch})
	&\equiv\dph(\hw_2,\ldots,\hw_d)
	\end{align*}}%
Let $\dq_y$ be as given by Proposition~\ref{p:sp}. Then, for $\dv,\dr,\dw,\ds\in\set{\zro,\one,\free}$, let
	\beq\label{e:full.stability.matrix}
	\bB_{\dv\dr,\dw\ds}
	\equiv
	\f{\displaystyle\sum_{\udw_\circ}
	\I\bigg\{
	\begin{array}{c}
	\dv=\WP(\dw,\dw_{2:\branch})\\
	\dr=\WP(\ds,\dw_{2:\branch})
	\end{array}
	\bigg\}
	\exp\bigg\{
	-y \vph(\ds,\dw_{2:\branch})
	\bigg\}
	\dq_y(\dw)
	\prod_{i=2}^{\branch}
	\dq_y(\dw_i)
	}
	{\displaystyle\sum_{\dw_{1:\branch}}
	\I\bigg\{
	\dv=\WP(\dw_{1:\branch})
	\bigg\}
	\exp\bigg\{
	-y \vph(\dw_{1:\branch})
	\bigg\}
	\prod_{i=1}^{\branch}
	\dq_y(\dw_i)
	}\,.
	\eeq
This defines a $9\times 9$ matrix $\bB$, which is the \bemph{stability matrix} for our model. We let $\bB_{\neq}$ be the $6\times 6$ submatrix with row and column indices in $\set{(\dw,\ds): \dw\ne\ds}$, and let $\lambda\equiv\lambda_y(\alpha)$ be the largest eigenvalue of $\bB_{\neq}$. The physics literature \cite{MRT,KPW} proposes that the $\onersb$ solution is correct as long as $\branch\lambda_y(\alpha)$ (a refinement of \eqref{e:percrate}) is less than one at $y=y_\star(\alpha)$. We extract its large-$k$ behavior in the following:

\begin{ppn}[proved in Section~\ref{s:gardner}]
\label{p:gardner.defn}
The \bemph{Gardner threshold} $\aG$ can be formally defined as 
	\[
	\aG 
	\equiv \sup\bigg\{
	\alpha\le \f{4^k}{k}:
	\branch\lambda_{y_\star(\alpha)}(\alpha)
	\le 1
	\bigg\}\,,
	\]
where $y_\star(\alpha)$ is the root given by Proposition~\ref{p:sp}.
The large-$k$ behavior is given by $\aG \asymp 4^k/k^3$.
\end{ppn}

\subsection{Interpolation bound}\label{ss:intro.interp}

As mentioned before, our proof of Theorem~\ref{t:main} is based on a general interpolation upper bound, in the spirit of \cite{MR1930572, MR1957729, MR2095932, MR3161470, MR3256814}. The precise bound that we use, as we now describe, is a generalization of a similar result in \cite{ssz}. Let $\Omega$ be the space of probability measures on $\set{\zro,\one,\free}$. We write $\rho$ for elements of $\Omega$, and $Q$ for probability measures over $\Omega$. Similarly as above, we will abuse notation and write
	\[\bigg
	(\dw_{ai} :	1\le a\le d,
					2\le i\le k
	\bigg)
	\equiv
	(\dw_j)_{1\le j\le D}
	\equiv
	\dw_{1:D}
	\]
where $D\equiv d(k-1)$. Let $\hw_a\equiv\hWP(\dw_{a,2},\ldots,\dw_{a,k})$ and $\vph(\dw_{1:D})\equiv \dph_d(\hw_1,\ldots,\hw_d)$. Define
	\begin{align}
	\label{e:tworsb.clause}
	\FG(y_1,y_2,Q)
	&=\int
	\bigg\{
	\sum_{\dw_{1:k}}
	\exp(-y_2\hph(\dw_{1:k}))
	\prod_{j=1}^k \rho_j(\dw_j)
	\bigg\}^{y_1/y_2}\prod_{j=1}^k dQ(\rho_j)\,,\\
	\label{e:tworsb.var}
	\WW(y_1,y_2,Q)
	&=\int
	\bigg\{
	\sum_{\dw_{1:D}}\exp(-y_2 \vph(\dw_{1:D}))
	\prod_{i=1}^D
	\rho_i(\dw_i)
	\bigg\}^{y_1/y_2}
	\prod_{i=1}^D dQ(\rho_i)\,.
	\end{align}
For $0\le y_1\le y_2$, the \bemph{zero-temperature 2RSB functional} is defined by
	\beq\label{e:Phi.tworsb}
	\PhiTwo(y_1,y_2,Q)
	\equiv
	\f1{y_1}\log \WW(y_1,y_2,Q)
	-\f{\alpha(k-1)}{y_1}\log \FG(y_1,y_2,Q)\,.
	\eeq
A heuristic derivation of $\PhiTwo$ is presented in Section~\ref{s:interpolation}, but we briefly describe it here. For simplicity assume $\alpha\equiv d/k$ is an integer, and let $\GG_N$ be an instance of random $d$-regular $k$-\textsc{nae-sat} on $N$ variables. Remove $\alpha(k-1)$ clauses and their incident edges at random, and call the resulting graph $\GG_{N+1/2}$: it is still a $k$-\textsc{nae-sat} instance on $N$ variables, but is no longer $d$-regular since some variables have open ``slots'' (missing edges). Then introduce a new variable $v\equiv v_{N+1}$, together with $d$ new clauses. For each new clause, add one new edge connecting the clause to $v$, and $k-1$ new edges connecting the clause to the open ``slots'' in $\GG_{N+1/2}$. Then the resulting graph $\GG_{N+1}$ is an instance of random $d$-regular $k$-\textsc{nae-sat} on $N+1$ variables. For $\beta\ge0$ we can consider
	\beq\label{e:gibbs.intermediate}
	\mu_\beta(\ux)
	\equiv
	\f{ \exp\{ -\beta \HH_{N+1/2}(\ux) \}}
		{Z_{N+1/2}(\beta)}
	\eeq
where $Z_{N+1/2}(\beta)$ is the normalizing constant that makes $\mu_\beta$ a probability measure over $\ux\in\set{\zro,\one}^N$. The structure of $\mu_\beta$ is not known. However, by analogy with other models \cite{derrida1981random,derrida1985generalization,MR875300,parisi1979infinite,MR2999044}, a natural simplifying assumption is that it has a 
\bemph{hierarchichal
(ultrametric) structure} with Poisson--Dirichlet weights on each level of the hierarchy. This means that the $\ell$-point marginals of $\mu_\beta$, for bounded $\ell$, converge in the large-$N$ limit to an explicit form: for a two-level hierarchy,
	\beq\label{e:rpc.cavity}
	\mu_\beta(x_1,\ldots,x_\ell)
	\approx \sum_{s,t\ge1} \nu_{st}
		\prod_{i=1}^\ell w_{st,i}(x_i)\,,
	\eeq
where the $w_{st,i}$ are sampled recursively as follows. Let $\PROB_0\equiv\PROB$ be the space of probability measures over $\set{\zro,\one}$, and for $r\ge1$ let $\PROB_r$ be the space of probability measures over $\PROB_{r-1}$. Let $Q_\beta\in \PROB_2$. Let $(r_{s,i})_{s,i}$ be i.i.d.\ samples from law $Q_\beta$. For each $i$ and each $s$, let $(w_{st,i})_{t\ge1}$ be a sequence of i.i.d.\ samples from $r_{s,i}$. Note $r_{s,i}\in\PROB_1$ so $w_{st,i}\in\PROB$.  Independently, $(\nu_{st})_{s,t\ge1}$ are random weights sampled from the law of a \bemph{Ruelle probability cascade} (RPC) with parameters $0<m_1<m_2<1$ --- a two-level version of the standard Poisson--Dirichlet process (see \cite[Ch.~2]{MR3052333} and Section~\ref{s:interpolation}). Under assumption
\eqref{e:rpc.cavity}, and taking $\beta\to\infty$ with
$m_i\beta\to y_i$, one has
	\begin{align*}
	\lim_{\beta\to\infty}
	\f1\beta
	\log\f{Z_N(\beta)}{Z_{N+1/2}(\beta)}
	&\approx
	\f{\alpha(k-1)}{y_1}
	\log\FG(y_1,y_2,Q)\,,\\
	\lim_{\beta\to\infty}
	\f1\beta
	\log\f{Z_{N+1}(\beta)}{Z_{N+1/2}(\beta)}
	&\approx
	\f1{y_1}\log\WW(y_1,y_2,Q)\,,
	\end{align*}
where $Q$ is a probability measure over $\Omega$, obtained as a projection of $Q_\beta$. The basic idea is as follows:
\begin{enumerate}[a.]
\item Project $w\in\PROB$ to $\ww\in\set{\zro,\one,\free}$
where $\set{w\textup{ near }\I_\zro}$ maps to $\ww=\zro$,
$\set{w\textup{ near }\I_\one}$ maps to $\ww=\one$,
and the remaining $w\in\PROB$ map to $\ww=\free$. Denote this mapping $\pi : \PROB\to\set{\zro,\one,\free}$.
\item Project $r\in\PROB_1$ to $\rho\in\Omega$ via the pushforward, $\rho(\ww)
= (\pi_\sharp r)(\ww)
=r ( \pi^{-1}(\ww))$.

\item Project $Q_\beta\in\PROB_2$ to a probability measure $Q$ over $\Omega$ via another pushforward,
$Q=(\pi_\sharp)_\sharp Q_\beta
= Q_\beta \circ (\pi_\sharp)^{-1}$.
\end{enumerate}
The details are given in Section~\ref{s:interpolation}.  Combining the above relations gives the heuristic approximation
	\[ -\ee_\star = \lim_{\beta\to\infty}
	\f1{N\beta}\log Z_N(\beta)
	=\lim_{\beta\to\infty}
	\f1\beta\log\f{Z_{N+1}(\beta)}{Z_N(\beta)}
	=\PhiTwo(y_1,y_2,Q)\,.
	\]
The following proposition shows that one side of the approximation can be made rigorous: 

\begin{ppn}[proved in Section~\ref{s:interpolation}]\label{p:interpolation}
For any parameters $0\le y_1\le y_2$ and any probability measure $Q$ over $\Omega$, we have a corresponding zero-temperature 2RSB bound $-\einf\le \PhiTwo(y_1,y_2,Q)$.
\end{ppn}

\noindent
The detailed heuristic derivation of 
$\PhiTwo$, as well as the proof of Proposition~\ref{p:interpolation}, are given in Section~\ref{s:interpolation}. There are two simple ways in which $\PhiTwo$ can degenerate:
\begin{enumerate}[I.]
\item The probability measure $Q$ is fully supported on a single element $\rho\in\Omega$. In this case $\PhiTwo(y_1,y_2,Q)$ depends only on $y_2$ and $\rho$, so we can define $
\PhiTwo(y_1,y_2,Q)\equiv \PhiOne(y_2,\rho)$.
\item \label{II}
The probability measure $Q$ decomposes as
$Q=\rho_\zro Q_\zro + \rho_\one Q_\one + \rho_\free Q_\free$ where each $Q_{\dw}$ is fully supported on the single element $\I_{\dw}\in\Omega$. In this case we have $\PhiTwo(y_1,y_2,Q)=\PhiOne(y_1,\rho)$.
\end{enumerate}
Let $\dq_y$ be the solution of Proposition~\ref{p:sp}, and define
	\beq\label{e:Q.II}
	Q_\II
	\equiv Q_{\II,y}
	\equiv 
	\sum_{\dw\in\set{\zro,\one,\free}}
	\dq_y(\dw)
	Q_{\dw}
	\eeq
where each $Q_{\dw}$ is the unit mass on $\I_{\dw}\in\Omega$. One can verify by straightforward algebraic manipulations that
	\[\PhiOne(y,\dq_y)=\f{\FF(y)}{y}\,.\]
Thus an immediate consequence of Proposition
~\ref{p:interpolation} is that $-\einf$ is upper bounded by
	\beq\label{e:inf.Fy.over.y}
	\PhiTwo(y,y,Q_{\II,y})
	=\PhiOne(y,\dq_y)
	=\f{\FF(y)}{y}\,,
	\eeq
throughout the range of $y$ where $\dq_y$ is defined. It has been observed in the physics literature \cite{MRT,KPW} that linearizing the stationarity equations for the functional $\PhiTwo$ (equivalently, the \bemph{2RSB cavity equations})
	\[
	\f{\partial\PhiTwo(y_1,y_2,Q)}{\partial Q}=0
	\]
around $Q=Q_\II$ gives rise to the stability matrix $\bB_{\neq}$ introduced in \S\ref{ss:intro.gardner}. To prove Theorem~\ref{t:main}, we show that an explicit perturbation of $(y,y,Q_\II)$ decreases the value of $\PhiTwo$ as soon as the top eigenvalue of $\bB_{\neq}$ exceeds $1/\branch$. While the physics literature certainly hints that this would be the case, to our knowledge this rigorous connection between the Gardner eigenvalue and the stability of the 2RSB functional has not been previously established.

\subsection*{Organization of paper} 
In Section~\ref{s:recursions} we prove Propositions~\ref{p:sp}~and~\ref{p:eone.defn}, as well as the general version of \eqref{e:alpha.onersb.correction}. In Section~\ref{s:gardner} we prove Proposition~\ref{p:gardner.defn} and Theorem~\ref{t:main}. Finally, in Section~\ref{s:interpolation} we review the 2RSB heuristic and prove Proposition~\ref{p:interpolation}. 

\subsection*{Acknowledgements} We thank Florent Krzakala, Andrea Montanari, Guilhem Semerjian, and Allan Sly for many helpful conversations. N.S.\ is supported in part by NSF CAREER grant DMS-1752728.

\section{Tree recursions and 1RSB value}
\label{s:recursions}

In this section we analyze the survey propagation (\textsc{sp}) recursions introduced in
\S\ref{ss:intro.formulas}. First, in \S\ref{ss:simplification.sp} we write the \textsc{sp} recursions in a simplified form. The proof of Proposition~\ref{p:sp} is then given in \S\ref{ss:contraction}. We give some discussion on the \textsc{sp} formula as a variational problem in \S\ref{ss:variational.discussion}.
The comparison with the first moment bound
(extending \eqref{e:alpha.onersb.correction}) appears in \S\ref{ss:comparison.first.mmt},
and the proof of Proposition~\ref{p:eone.defn} 
appears in \S\ref{ss:ground.state}. Some technical lemmas are deferred to \S\ref{ss:binomial}.
 Recalling \eqref{e:cc.parametrization}, we assume throughout the paper that $\alpha=d/k =\CC 2^{k-1}\log2$ with $\asat\le\alpha \le 4^k/k$. We always restrict our consideration to parameters $y\ge0$ 
satisfying \eqref{e:y.restriction}.
From now on we will often suppress $y$ from the notation, e.g.\ we write $\SP\equiv\dSP\circ\hSP$
rather than $\SP_y\equiv\dSP_y\circ\hSP_y$. 

\subsection{Simplification of SP equations}
\label{ss:simplification.sp} Recall \eqref{e:MM.symmetric} that $\MM$ denotes the space of all probability measures $q$ on $\set{\zro,\one,\free}$ with the symmetry $q(\zro)=q(\one)$. We first take advantage of this symmetry to simplify the SP recursions: the mapping from  $\dq\in\MM$ to $\hq=\hSP(\dq)$ to $\tq=\dSP(\hq)$ can be naturally expressed as \emph{univariate} mappings from $x$ to $w$ to $\tilde{x}$ where $x\equiv \dq(\free)$, $w\equiv 1-\hq(\free)$, and $\tilde{x} \equiv \tq(\free)$. First, the clause recursion \eqref{e:sp.clause} can be rewritten as
	\beq\label{e:clause.x.to.w}
	w = w(x)
	= \f{2(1-x)^{k-1}}{2^{k-1}}
	\le \f{4}{2^k}\,.
	\eeq
To simplify the variable recursion \eqref{e:sp.var}, first write
	\[
	\AM \equiv \f{1+e^{-y}}{2}\,,\quad
	\GM\equiv \f1{\exp\{y/2\}}\,,
	\]
and use these to define (for $0\le\ell\le d-1$) the binomial weights
	\begin{align*}
	\AAA_\ell&\equiv \AAA_{d-1,\ell}(w)
	\equiv
		\binom{d-1}{\ell}
		(w\cdot \AM)^\ell(1-w)^{d-1-\ell}\,,
		\\
	\GGG_\ell
	&\equiv \GGG_{d-1,\ell}(w)
	\equiv \binom{d-1}{\ell}
		(w\cdot \GM)^\ell(1-w)^{d-1-\ell}\,.
	\end{align*}
Note that $1/2\le\AM\le1$ always, but $\GM$ can be arbitrarily small. For each $\ell$ we also define
	\begin{align}
	\PPP_\ell
	&\equiv \P\bigg(
		\Bin\bigg(
		\ell,\f{1}{1+e^y}
		\bigg) < \f{\ell}{2}
		\bigg)\label{e:arith.prob}\,,\\
	\QQQ_\ell
	&\equiv \P\bigg(
		\Bin\bigg(
		\ell,\f{1}{1+e^y}
		\bigg) = \f{\ell}{2}
		\bigg)\label{e:geom.prob.alt}\,,\\
	\SSS_\ell 
	&\equiv \P\bigg(
		\Bin\bigg(\ell,\f12
		\bigg) =\f{\ell}{2}
		\bigg)
		\label{e:geom.prob}\,,
	\end{align}
Note that $\PPP_\ell=0$ for $\ell=0$, $\SSS_\ell=0$ for all $\ell$ odd, and $\AAA_\ell\QQQ_\ell=\GGG_\ell\SSS_\ell$. Let
	\begin{align}
	\label{e:def.dz.zro}
	\dz_\zro \equiv \dz_\zro(w)
	&\equiv
	\sum_{\ell=0}^{d-1}
		\AAA_\ell\PPP_\ell
	\equiv \sum_{\ell=0}^{d-1}
		\AAA_{d-1,\ell}(w)\PPP_\ell\,,\\
	\dz_\free \equiv \dz_\free(w)
	&=\sum_{\ell=0}^{d-1}
		\AAA_\ell
		\QQQ_\ell
	= \sum_{\ell=0}^{d-1}
		\GGG_\ell
		\SSS_\ell\,.
	\label{e:def.dz.free}
	\end{align}
Let $\dz(w)\equiv 2\dz_\zro(w)+\dz_\free(w)$. The variable recursion \eqref{e:sp.var} can be rewritten as
	\beq\label{e:variable.w.to.x}
	\tilde{x}(w) 
	=
	\f{\dz_\free(w)}
	{2\dz_\zro(w)+\dz_\free(w)}
	= \f{\dz_\free(w)}
	{\dz(w)}
	\,.
	\eeq
To analyze the recursion, we let $\LAM$ and $\LGM$ be random variables with distribution
	\begin{align}
	\label{e:arith.Z}
	\P(\LAM=\ell) 
	&= 
	\f{\AAA_\ell\PPP_\ell}{\dz_\zro}
	=\f1{\dz_\zro}\binom{d-1}{\ell}
	(w\cdot\AM)^\ell (1-w)^{d-1-\ell}
	\P\bigg(\Bin\bigg(\ell,\f1{1+e^y}\bigg)<\f12\bigg)
	\,,\\
	\label{e:geom.Z}
	\P(\LGM=\ell) 
	&=\f{\GGG_\ell\SSS_\ell}{\dz_\free}
	=\f1{\dz_\free}
	\binom{d-1}{\ell}
	(w\cdot\GM)^\ell (1-w)^{d-1-\ell}
	\P\bigg(\Bin\bigg(\ell,\f12\bigg)
		=\f{\ell}{2}\bigg)
	\,,
	\end{align}
for $0\le \ell\le d-1$.  For comparison, let $\binLAM$ and $\binLGM$ be random variables with distributions
	\begin{align}\label{e:binLAM}
	\P(\binLAM=\ell) 
	&= \f{\AAA_\ell}{(1-w+w\cdot\AM)^{d-1}}
	\quad\textup{for }0\le\ell\le d-1\,,\\
	\label{e:binLGM}
	\P(\binLGM=\ell) 
	&= \f{\GGG_\ell}{(1-w+w\cdot\GM)^{d-1}}
	\quad\textup{for }0\le\ell\le d-1\,.
	\end{align}
Let $\lam\equiv\E \binLAM$ and $\lgm\equiv\E \binLGM$.

\subsection{Contraction of SP}
\label{ss:contraction}

Recall \eqref{e:MM.symmetric} that $\MM$ denotes the space of all probability measures $q$ on $\set{\zro,\one,\free}$ with the symmetry $q(\zro)=q(\one)$; and recall \eqref{e:MM.bullet} and \eqref{e:MM.gamma} the definitions of $\MM^\gamma\subseteq\MMnf\subseteq\MM$. Let
	\[
	\MMstar
	\equiv\bigg\{
	q\in\MMnf : 
	q(\free) 
	\asymp \f{e^{-\Theta(k)}}{\CC^{1/2}}
	\bigg\}\,\]
Recall \eqref{e:MM.gamma} and note that $\MM^\gamma\subseteq\MMstar$: if $\CC$ is large then \eqref{e:y.restriction} forces $y\asymp1/\CC^{1/2}$, so
	\[
	\f{2^{-k\gamma/2}}
	{(\max\set{\CC k e^{-y/2},1})^{1/2}}
	\asymp
	\f{2^{-k\gamma/2}}{(\CC k)^{1/2}}
	\asymp \f{e^{-\Theta(k)}}
		{\CC^{1/2}}\,.
	\]
If $\CC\asymp1$ then it forces $y\ge\Omega(1)$, so
	\[ \f{e^{-\Theta(k)}}{\CC^{1/2}}
	\asymp
	\f{2^{-k\gamma/2} }{(\CC k)^{1/2}}
	\le\f{2^{-k\gamma/2}}{
		(\max\set{\CC k e^{-y/2},1})^{1/2}}
	\le 2^{-k\gamma/2}
	\asymp 
	\f{e^{-\Theta(k)}}{\CC^{1/2}}\,.
	\]
The next two propositions, which will be proved in this subsection, summarize the contractive behavior of the \textsc{sp} recursion in our regime of interest.

\begin{ppn}\label{p:initial.contraction}
Suppose $\alpha=\CC 2^{k-1}\log2$ with $\asat\le\alpha \le 4^k/k$, and suppose $y\ge0$ satisfies \eqref{e:y.restriction}.
Then for any $\dq\in\MMnf$ we have $\SP_y(\dq)\in\MM^\gamma\subseteq \MMstar$.
\end{ppn}

\begin{ppn}
\label{p:derivative.MMstar}
Suppose $\alpha=\CC 2^{k-1}\log2$ with $\asat\le\alpha \le 4^k/k$, and suppose $y\ge0$ satisfies \eqref{e:y.restriction}. Then the derivative of the survey propagation map satisfies the estimate
	\[
	\Omega\bigg(
		\f{e^{-\Theta(k)}}{\CC^{1/2}}
		\bigg)
	\le
	\f{d [\SP_y(\dq)](\free)}
		{d\dq(\free)}\le O\bigg(
		\f1
			{e^{\Omega(k)}}
		\bigg)
	\]
uniformly over all
 $\dq\in\MMstar$.
\end{ppn}

\begin{proof}[Proof of Proposition~\ref{p:sp}]
It follows from Proposition~\ref{p:initial.contraction} that if $\dq(\free)=1/k^2$ then $[\SP(\dq)](\free)<\dq(\free)$. On the other hand, an easy calculation gives that if $\dq(\free)=0$ then
$[\SP(\dq)](\free)>\dq(\free)$. It follows by the intermediate value theorem that a fixed point $[\SP(\dq)](\free)=\dq(\free)$ must exist with $\dq(\free)\le1/k^2$. By Proposition~\ref{p:initial.contraction}, any such fixed point must in fact lie in $\MM^\gamma\subseteq\MMstar$. It then follows from Proposition~\ref{p:derivative.MMstar} that the fixed point is unique.
\end{proof}

\noindent The following lemma is straightforward, and its proof is deferred:

\begin{lem}[proved in \S\ref{ss:binomial}]
\label{l:binom.Z.estimates}
If $\alpha=\CC 2^{k-1}\log2$ with $\asat\le\alpha \le 4^k/k$, and $w\asymp 1/2^k$, then we have
	\begin{align}
	\label{e:bounds.Z.zro}
	1-O\bigg(\f1{\exp\{\Omega(k)\}}\bigg)
	&\le
	\f{\dz_\zro(w)}{(1-w+w\cdot\AM)^{d-1}}
	\le1\,,\\
	\label{e:bounds.Z.free}
	\f{\dz_\free}{(1-w+w\cdot\GM)^{d-1}}
	&\asymp
	\f1{(\max\set{dw\cdot\GM,1})^{1/2}}\,,
	\end{align}
uniformly over $y\ge0$ satisfying  \eqref{e:y.restriction}.
\end{lem}

\begin{proof}[Proof of Proposition~\ref{p:initial.contraction}]
We start from $\dq\in\MMnf$, which means $x\equiv \dq(\free) \le 1/k^2$. This maps to $\hq=\hSP(\dq)$, where we can see readily from \eqref{e:clause.x.to.w} that
	\[
	1-\hq(\free) = w(x) 
	= \f{4}{2^k}\bigg\{1 + O\bigg(\f1k\bigg)\bigg\}\,.
	\]
This then maps to $\tq=\dSP(\hq)$. Substituting the bounds of Lemma~\ref{l:binom.Z.estimates} into 
\eqref{e:variable.w.to.x} gives
	\begin{align*}
	\tq(\free)
	\equiv\tilde{x}(w)
	&\asymp \f1{(\max\set{dw\cdot\GM,1})^{1/2}}
	\bigg(\f{1-w+w\cdot\GM}
	{1-w+w\cdot\AM}\bigg)^{d-1}\\
	&\asymp
	\f{\exp\{-d w(\AM-\GM)\}}
		{(\max\set{dw\cdot\GM,1})^{1/2}}
	\asymp 
	\f{2^{-k\gamma/2} }{(\max\set{\CC k e^{-y/2},1})^{1/2}}
	\,,
	\end{align*}
which implies $\tq\in\MM^\gamma\subseteq\MMstar$.\end{proof}

\noindent We next provide a simple (and rather crude) estimate on $\E\LAM$ and $\E\LGM$.

\begin{lem}[proved in \S\ref{ss:binomial}]
\label{l:simplified.binom.EL.estimates}
If $\alpha=\CC 2^{k-1}\log2$ with $\asat\le\alpha \le 4^k/k$, and $w\asymp 1/2^k$, then we have
	{\setlength{\jot}{0pt}\begin{align}
	\label{e:simplified.ELAM.estimate}
	|\E\LAM-\lam|&\le O((dw)^{1/2})
	=O((\CC k)^{1/2})
	\,,\\
	\label{e:simplified.ELGM.estimate}
	|\E\LGM-\lgm|
	&\le
	\min\set{dw\cdot\GM,(dw\cdot\GM)^{1/2}}
	\le O((\CC k)^{1/2})\,.
	\end{align}}%
uniformly over $y\ge0$ satisfying  \eqref{e:y.restriction}.
\end{lem}

\noindent 
Lemma~\ref{l:simplified.binom.EL.estimates}
immediately gives
$\E(\LAM-\LGM) \le  O((\CC k)^{1/2})$. We do not expect this bound to be tight, but it will suffice for our purposes. We also obtain the following lower bound:

\begin{lem}[proved in \S\ref{ss:binomial}]
\label{l:one.sided}
In the setting of Lemma~\ref{l:simplified.binom.EL.estimates} we also have $\E(\LAM-\LGM)\ge \Omega(k)$.
\end{lem}

\begin{proof}[Proof of Proposition~\ref{p:derivative.MMstar}] 
From the clause \textsc{sp} recursion \eqref{e:clause.x.to.w} we calculate
	\[
	-w'(x)
	=\f{2(k-1)(1-x)^{k-2}}{2^{k-1}}
	=kw\bigg\{1 +O\bigg(\f1k\bigg)\bigg\}\,.
	\]
From the variable \textsc{sp} recursion \eqref{e:variable.w.to.x} we calculate
	\[
	-\tilde{x}'(w)
	=\f{2\dz_\zro \dz_\free}{(2\dz_\zro+\dz_\free)^2}
	\bigg\{\f1{\dz_\zro}\f{d\dz_\zro}{dw}
	-\f1{\dz_\free}\f{d\dz_\free}{dw}
	\bigg\}
	=\f{\tilde{x}(1-\tilde{x})
	\cdot \E(\LAM-\LGM)}{w(1-w)}\,.
	\]
Combining these gives
	\[
	\f{d[\SP(\dq)(\free)]}{d\dq(\free)}
	= \tilde{x}'(w(x)) w'(x)
	= \bigg\{1 +O\bigg(\f1k\bigg)\bigg\}
	\f{k\tilde{x}(1-\tilde{x})
	\cdot \E(\LAM-\LGM)}{(1-w)}\,.
	\]
It then follows from Lemmas~\ref{l:simplified.binom.EL.estimates}~and~\ref{l:one.sided} that
	\[
	\Omega(k^2\tilde{x})
	\le 
	\f{d[\SP(\dq)(\free)]}{d\dq(\free)}
	\le O \bigg(
	k \tilde{x} (\CC k)^{1/2}
	\bigg)\,.
	\]
Since $\dq\in\MMstar\subseteq\MMnf$, it follows from Proposition~\ref{p:initial.contraction} that $\tq\in\MMstar$ as well. The claimed bound then follows from the definition of $\MMstar$.\end{proof}

\subsection{Replica symmetric formulas for SP model}
\label{ss:variational.discussion}

We next evaluate the formula \eqref{e:FF.free.energy} for $\FF(y)$. Recall that $\dq_y\in\MMnf$ is the solution given by Proposition~\ref{p:sp}, and $\hq_y\equiv \hSP_y(\dq)$. We will now express all formulas in terms of $x_y \equiv\dq(\free)$ and $w_y \equiv1-\hq(\free)$. The normalizing constants of \eqref{e:sp.clause.marginal} and \eqref{e:sp.edge.marginal} are equal:
	\beq\label{e:clause.edge.eq}
	\hfz_y(\dq_y)
	= 1 - \f{2(1-x_y)^k(1-e^{-y})}{2^k}
	= 1 - \f{(1-x_y) w_y(1-e^{-y})}{2}
	= \efz_y(\dq_y,\hq_y)\,.
	\eeq
Meanwhile the normalizing constant of \eqref{e:sp.var.marginal} is 
	\beq\label{e:sp.var.marginal.Zdw}
	\dfz_y(\hq)
	=2\bigg( 
	\dfz_\zro(w) + 
	\f{\dfz_\free(w)}{2}\bigg)
	\equiv \dfz(w)
	\eeq
where  $\dfz_\zro(w)$ and $\dfz_\free(w)$ are defined similarly to \eqref{e:def.dz.zro} and \eqref{e:def.dz.free}, but with $d-1$ in place of $d$:
	\[
	\dfz_\zro(w)
	=\sum_{\ell=0}^d
	\AAA_{d,\ell}(w)\PPP_\ell\,,\quad
	\dfz_\free(w)
	=\sum_{\ell=0}^d
	\GGG_{d,\ell}(w)\SSS_\ell\,.
	\]
Then we have $\FF(y) = F(x_y,w_y,y)$ for
	\beq\label{e:Fxwy}
	F(x,w,y)
	\equiv
	\log \dfz(w)
	+ \alpha\log
	\bigg\{1 - \f{2(1-x)^k(1-e^{-y})}{2^k}\bigg\}
	-\alpha k\log
	\bigg\{ 1- \f{(1-x)w(1-e^{-y})}{2}\bigg\}
	\,.\eeq
Since  $\hfz_y(\dq_y)$ and $\efz_y(\dq_y,\hq_y)$ are equal as we noted above, there is more than one way to define $F(x,w,y)$ such that
$\FF(y) = F(x_y,w_y,y)$. We have chosen the representation \eqref{e:Fxwy} because it satisfies the following:

\begin{lem}\label{l:stationarity}
For any given $y>0$,
if the pair $(x,w)$ satisfies the equations $w=w(x)$
and $x=\tilde{x}(w)$
as in \eqref{e:clause.x.to.w} and \eqref{e:variable.w.to.x}
(where the second relation \eqref{e:variable.w.to.x}
depends also on $y$), then $(x,w)$ is a stationary point of 
$(x,w)\mapsto F(x,w,y)$.
\end{lem}

\noindent The proof of Lemma~\ref{l:stationarity} is deferred to \S\ref{ss:stationarity}, but we now point out its main consequences: first, it gives
	\[
	\FF'(y)
	=\f{\partial F}{\partial x}(x_y,w_y,y)\f{dx_y}{dy}
	+\f{\partial F}{\partial w}(x_y,w_y,y)\f{dw_y}{dy}
	+\f{\partial F}{\partial y}(x_y,w_y,y)
	=\f{\partial F}{\partial y}(x_y,w_y,y)\,.
	\]
The right-hand side above is equal to $-\ee(y)$,
as defined by \eqref{e:onersb.ee.of.y}. Recalling \eqref{e:inf.Fy.over.y}, it follows that
	\beq\label{e:deriv.Fy.over.y}
	-\f{d}{dy}\bigg( \f{\FF(y)}{y}\bigg)
	= \f{\FF(y) - y\FF'(y)}{y^2}
	= \f{\FF(y) + y\ee(y)}{y^2}
	= \f{\SIGMA(y)}{y^2}\,,
	\eeq
which is to say that a stationary point of $\FF(y)/y$ corresponds precisely to a root of $\SIGMA(y)$. Moreover
	\beq\label{e:SIGMA.prime}
	\SIGMA'(y)
	= \f{d}{dy}\bigg(
		\FF(y)+y\ee(y)
		\bigg)
	= \FF'(y) + \ee(y) + y\ee'(y)
	= y\ee'(y)
	= -y\FF''(y)\,.
	\eeq
To determine the sign, we will prove the following, which will be used in the proof of Proposition~\ref{p:eone.defn}:

\begin{ppn}[proved in \S\ref{ss:convexity}]\label{p:second.derivative}
In the setting of Proposition~\ref{p:sp} we have $\FF''(y)\asymp \CC/e^y$: in particular, $\FF$ is convex in the range of $y$ satisfying \eqref{e:y.restriction}.
\end{ppn}

\noindent
The proof of Proposition~\ref{p:second.derivative} is rather long, and is deferred to the very end of this paper. However, we take a moment to review the basic physical interpretations which lead us to expect the result of Proposition~\ref{p:second.derivative}.

\begin{rmk}\label{r:stationarity} As we already commented before, a salient open question is to prove the matching lower bound of Theorem~\ref{t:main}, i.e., to prove that  $\lim_N\E[\emin(\GG_N)]=\eone$ for all $\asat\le\alpha\le\aG$. A natural route is to validate the \textsc{sp} heuristic by proving that $N^{-1}\log\mathfrak{Z}(y)$ indeed converges --- for a suitable range of $y$, as we discuss in Remark~\ref{r:range.of.y} below --- to the predicted value $\FF(y)$ defined by \eqref{e:FF.free.energy}. This is believed to hold, although we do not have a proof in this paper.  However we now review how this prediction relates to other basic properties of $\FF(y)$.  In fact, we expect that in the limit $N\to\infty$ we have
	\beq\label{e:rs.first.mmt}
	\E\mathfrak{Z}(y) 
	\asymp \exp\{N\FF(y)\}
	= \exp\bigg\{
	 N\bigg[\SIGMA(y)  - y\ee(y)\bigg]\bigg\}
	 \eeq
for all $y$ satisfying \eqref{e:y.restriction}. On the other hand, we can express
	\beq\label{e:empir.msr.decomp}
	\E\mathfrak{Z}(y)
	=\sum_\nu
	\f{\E\mathfrak{Z}_\nu}{\exp\{ Ny(\nu,\vph)\}}
	\eeq
where $\nu\equiv (\dot{\nu},\hat{\nu},\bar{\nu})$ is a triple of probability measures (on $\set{\zro,\one,\free}^d$, $\set{\zro,\one,\free}^k$, $\set{\zro,\one,\free}^2$ respectively); $\mathfrak{Z}_\nu$ is the total count of warning configurations $\uw$ with empirical measure $\nu$; and for any such configuration we can express $\vph(\uw)$ as a linear functional of $\nu$, which we write as $N(\nu,\vph)$. From the representation \eqref{e:empir.msr.decomp} we see that on the range of $y$ where \eqref{e:rs.first.mmt} holds, the function $\FF$ should correspond to the $N\to\infty$ limit of cumulant-generating functions. In particular
(by interchanging the order between limit and differentiation) we expect $N\FF'(y)$ to correspond to the mean of $-(\nu,\vph)$. This is consistent with the above observation that $\FF'(y)=-\ee(y)$, since $\ee(y)$ is precisely the replica symmetric prediction for $\lim_N \langle\vph\rangle_N/N$. We further expect that $N\FF''(y)$ corresponds to the variance of $(\nu,\vph)$, and it is natural to guess that this scales as a positive constant times $N$. This explains why we expect $\FF'(y)>0$. 
\end{rmk}

\begin{rmk}\label{r:range.of.y} 
We show in Proposition~\ref{p:second.derivative} that $\ee(y)$ is strictly decreasing, so that $\Sigma(\ee(y))=\SIGMA(y)$ is well-defined.
Continuing from the previous remark, we expect that the typical value of the random variable $\mathfrak{Z}(y)$ is given by 
(compare with \eqref{e:rs.first.mmt})
	\[\lim_{N\to\infty}
	\f{\log\mathfrak{Z}(y)}{N}
	=\max_\ee\bigg\{ \Sigma(\ee)-y\ee
		:\Sigma(\ee)\ge0
		\bigg\}
	< \max_e\bigg\{ \Sigma(\ee)-y\ee
		\bigg\}
	=\lim_{N\to\infty}
	\f{\log \E\mathfrak{Z}(y)}{N}
	=\FF(y)\,,
	\]
where the two sides are equal for $y\le y_\star$. This makes precise the range of $y$ where we expect the \textsc{sp} heuristic to hold. We further expect that the second moment is given by
	\[\E[\mathfrak{Z}(y)^2]
	\asymp 
	\exp\bigg\{ N \max\bigg\{ 2\FF(y),
		 \SIGMA(y)-2y\ee(y) \bigg\}
		 \bigg\}\,,\]
--- in the maximum, the first term is the contribution from pairs of solutions in different clusters, while the second 
term is the contribution from pairs of solutions in the same cluster. Therefore $\E[\mathfrak{Z}(y)^2]/\E[\mathfrak{Z}(y)]^2$ is $O(1)$ for $y\le y_\star$, and is exponentially large (with respect to $N$) for $y>y_\star$. This is to say that we expect the second moment method to work on $\mathfrak{Z}(y)$ precisely up to the critical value $y_\star$.\end{rmk}

\noindent We emphasize that Remarks~\ref{r:stationarity}~and~\ref{r:range.of.y} are speculative (the limits stated there are not proved). We have included it only for the purpose of elaborating on the physical content of basic properties of the functions $\FF(y)$, $\ee(y)$, and $\SIGMA(y)$. We now proceed to formally prove our claims on the $\onersb$ ground state energy formula. 

\subsection{Comparison with known upper bound}
\label{ss:comparison.first.mmt}

We first establish the comparison \eqref{e:alpha.onersb.correction} between the $\onersb$ bound and the first moment upper bound. As above,
let $\dq_y$ be as given by Proposition~\ref{p:sp}, $\hq_y\equiv\hSP_y(\dq_y)$.
Denote $x\equiv x_y \equiv \dq_y(\free)$
and $w\equiv w_y\equiv 1-\hq_y(\free)$. Let
	\[\RSFF(y)
	=\log2 + 
	\alpha\log\bigg(1-\f{2(1-e^{-y})}{2^k}\bigg)\,,
	\]
and note from \eqref{e:ANP.first.moment.f.eta} that we can express
	\[\textsf{f}_\eta(\alpha,\ee)
	= \RSFF\bigg(\log\f1\eta\bigg) 
	+\ee \log\f1\eta\,.\]
Towards the proof of \eqref{e:alpha.onersb.correction}, we first establish the following comparison between
$\FF(y)$ and $\RSFF(y)$:

\begin{lem}\label{l:alpha.onersb.correction}
Under the conditions of Proposition~\ref{p:sp}, we have $\FF(y) \le \RSFF(y) - \Omega(x_y)$
where $x_y\equiv\dq_y(\free)$.

\begin{proof}
With some simple algebraic manipulations we can express
	\begin{align}\nonumber
	\FF(y) - \RSFF(y)
	&=- \log \bigg(1+\f{\dfz_\free(w)}
		{2\dfz_\zro(w)+\dfz_\free(w)}\bigg)+
	\log \f{\dfz_\zro(w)+\dfz_\free(w)}
	{(1-w(1-\AM))^d}\\
	\label{e:FF.RSFF.diff}
	&\qquad+ \alpha\bigg\{
	\log\f{1-w(1-\AM)}{1-(4/2^k)(1-\AM)}
	-(k-1)\log
	\f{1-w(1-\AM)}{1-(1-x)w(1-\AM)}\bigg\}\,.
	\end{align}
We will argue that the dominant contribution comes from the first term. To this end, recall from Proposition~\ref{p:sp} that 
	\[x\asymp \f{2^{-k\gamma/2}}{
	(\max\set{\CC k e^{-y/2},1})^{1/2}}
	\,.\]
From the equation $w=w(x)$ (see \eqref{e:clause.x.to.w}) we have
$w = 4[1-(k-1)x + O(k^2 x^2)]/2^k$. It follows that
	\begin{align*}
	\log\f{1-w(1-\AM)}{1-(4/2^k)(1-\AM)}
	&=\f{4}{2^k}(1-\AM) (k-1)x\bigg\{
	1 + O(kx)
	\bigg\}\,,\\
	(k-1)\log
	\f{1-w(1-\AM)}{1-(1-x)w(1-\AM)}\bigg\}
	&= \f{4}{2^k}(1-\AM) (k-1)x\bigg\{
	1 + O(kx)
	\bigg\}\,.\end{align*}
By combining these we see that the second line of \eqref{e:FF.RSFF.diff} is bounded by 
	\[
	O\bigg(
	\alpha \cdot 
	\f{4}{2^k}(1-\AM) (k-1)x \cdot
	kx
	\bigg)
	= O\bigg( \CC(1-\AM)  k^2 x^2
	\bigg)\le \f{O(x)}{\exp(\Omega(k))}\,,\]
where the last bound above uses 
the bound on $x$ together with the observation
that
$1-\AM\asymp\min\set{1,y}\asymp1/\CC^{1/2}$
by \eqref{e:y.restriction}. Next we estimate $\dfz_\zro(w)+\dfz_\free(w)$. Explicitly,
	\begin{align}
	\label{e:Z.d.zero.as.sum}
	\dfz_\zro(w)
	&= \sum_{\ell=0}^d\binom{d}{\ell}
	(w\cdot\AM)^\ell(1-w)^{d-\ell}\PPP_\ell\,,\\
	\label{e:Z.d.free.as.sum}
	\dfz_\free(w)
	&=
	\sum_{\ell=0}^d\binom{d}{\ell}
	(w\cdot\AM)^\ell(1-w)^{d-\ell}\QQQ_\ell
	\end{align}
with $\PPP_\ell$ as in \eqref{e:arith.prob}
and $\QQQ_\ell$ as in \eqref{e:geom.prob.alt}. We have
	\[
	\RRR_\ell\equiv
	\PPP_\ell+\QQQ_\ell
	=\P\bigg(
	\Bin\bigg(\ell,\f1{e^y+1}\bigg)\le\f{\ell}{2}
	\bigg)\le1\,,
	\]
so clearly $\dfz_\zro(w)+\dfz_\free(w)
< (1-w(1-\AM))^d$, which means that the second term in the first line of \eqref{e:FF.RSFF.diff} is negative. Finally, the estimates of Lemma~\ref{l:binom.Z.estimates} and Proposition~\ref{p:initial.contraction}
remain valid with
$d$ in place of $d-1$, so 
the first term in the first line of \eqref{e:FF.RSFF.diff} is
	\[
	\le -\Omega\bigg(
	\f{\dfz_\free(w)}
		{2\dfz_\zro(w)+\dfz_\free(w)}\bigg)
	\le-\Omega(x)\,.
	\]
The claim follows. 
\end{proof}
\end{lem}

\noindent Recalling
\eqref{e:cc.parametrization},
\eqref{e:anp.eta.star},
\eqref{e:rs.alpha.ubd},
\eqref{e:y.restriction}, 
and \eqref{e:MM.gamma}, we define
	\beq\label{e:x} x(\pgs)\equiv
	\bigg(
	\f{
	\exp\{-(k\log2)
	\cdot
	2\CC(\pgs)(1-\eta(\pgs)^{1/2})^2
	\}
	}{
	\max\set{\CC(\pgs) k 
	\eta(\pgs)^{1/2},1}}
	\bigg)^{1/2}\,.\eeq
We extend \eqref{e:alpha.onersb.correction} to the following bound:

\begin{cor}\label{c:alpha.onersb.correction}
For $\Omega(k/2^k) \le \pgs^2\le1$ 
(corresponding to $1\le\CC(\pgs)\le 2^k/k$)
we have
	\[
	\amax(\pgs)
	\le\bigg\{
	1-\Omega(x(\pgs))\bigg\}
	\aubd(\pgs)
	\]
for $\pgs$ defined by \eqref{e:cc.parametrization}
and $x(\pgs)$ defined by \eqref{e:x}. 

\begin{proof} 
Parametrize $\ee=\alpha(1-\pgs)/2^{k-1}$ 
as in \eqref{e:cc.parametrization}, and let
$y\equiv y_\eta = -\log\eta$ where $\eta= \eta(\pgs)$ as in \eqref{e:anp.eta.star}. Recall that for this particular choice $y=y_\eta$ we have $\textsf{f}_\eta(\alpha,\ee)= \RSFF(y) +y \ee$. Then
Lemma~\ref{l:alpha.onersb.correction} gives 
	\[
	\FF(y_\eta) + y_\eta\ee
	\le \textsf{f}_\eta(\alpha,\ee)
	- \Omega(x(\pgs))
	= \log2
	-\Omega(x(\pgs))
	-\alpha
	\bigg\{
	\f{(1-\pgs)\log\eta}{2^{k-1}}
	-\log\bigg(1-\f{1-\eta}{2^{k-1}}\bigg)\bigg\}
	\,.
	\]
By essentially the same calculation as \eqref{e:rs.alpha.ubd}, the last expression above will be negative for all $\alpha$ larger than
	\[
	\bar{\alpha}(\pgs)
	\equiv \f{2^{k-1}(\log2-\Omega(x(\pgs)))}
	{(2^{k-1}-(1-\pgs))
	\log \f{2^{k-1}-(1-\pgs)}{2^{k-1}-1}
	+(1-\pgs)\log(1-\pgs)}
	\le 
	\bigg\{
	1-\Omega(x(\pgs))\bigg\}
	\aubd(\pgs)\,.
	\]
From our interpolation bound \eqref{e:inf.Fy.over.y}, if $\FF(y')+y'\ee$ is negative for any $y'\ge0$, then
	\[ -\einf(\alpha) 
	\le\inf_{y\ge0} \f{\FF(y)}{y}
	\le \f{\FF(y')}{y'} < -\ee\,.
	\]
Since this applies for all $\alpha>\bar{\alpha}(\pgs)$, we conclude $\amax(\pgs)\le\bar{\alpha}(\pgs)$, proving the claim.
\end{proof}
\end{cor}

\noindent For $k^2 \ll_k \CC \ll_k 4^k/k$ it is straightforward to verify that $x(\pgs)\asymp 1/d^{1/2}$, so Corollary~\ref{c:alpha.onersb.correction} subsumes \eqref{e:alpha.onersb.correction}. One can also verify that for all $\pgs$ in the stated range we have
	\[1\le 
	2\CC(\pgs)(1-\eta(\pgs)^{1/2})^2
	\le \f{2(1-(1-\pgs)^{1/2})^2}
	{\pgs+(1-\pgs)\log(1-\pgs)}
	\le2\,,
	\]
and substituting into \eqref{e:x} gives $x(\pgs) \le 1/2^{k/2}$. This confirms that the improved upper bound of Corollary~\ref{c:alpha.onersb.correction} does not contradict the lower bound \eqref{e:anp.result} (which, as we remarked before, is the analogue of the \cite{MR2295994} lower bound for this model).

\subsection{Ground state energy}
\label{ss:ground.state}

In \S\ref{ss:comparison.first.mmt} we effectively considered $\alpha$ in terms of $\eone$, and set the parameter $y$ exactly to match $\log\eta$ from the first moment calculation; this gives a relatively easy way to obtain the comparison \eqref{e:alpha.onersb.correction}. We now proceed to the proof of Proposition~\ref{p:eone.defn} where the main difficulty is to solve for $y$ for which $\SIGMA(y)=0$.

\begin{lem}\label{l:FF.y.estimate}
In the setting of Proposition~\ref{p:sp},
	\[
	\FF(y)
	=\log2\bigg\{1-\CC(1-e^{-y})\bigg\}
	+O\bigg(\f1{e^{\Omega(k)}}\bigg)\,.
	\]
\end{lem}

\begin{lem}\label{l:energy.function.y.estimate}
In the setting of Proposition~\ref{p:sp},
	\[y\ee(y)= \f{\CC y \log2}{e^y}
		+O\bigg(\f1{e^{\Omega(k)}}\bigg)\,.\]
\end{lem}

\noindent The two lemmas immediately imply 
Proposition~\ref{p:eone.defn}:

\begin{proof}[Proof of Proposition~\ref{p:eone.defn}]
It follows from Lemma~\ref{l:FF.y.estimate}~and~\ref{l:energy.function.y.estimate} that
the energetic complexity function satisfies
	\[\f{\SIGMA(y)}{\log2}
	=\f{\FF(y)+y\ee(y)}{\log2}
	=1-\CC\bigg(1-\f{1+y}{e^y}\bigg)
	+O\bigg(\f1{e^{\Omega(k)}}\bigg)
	=1-\Gamma(y)+O\bigg(\f1{e^{\Omega(k)}}\bigg)\,.
	\]
This implies that $\SIGMA(y)$ must have at least one root $y_\star$ that satisfies the estimate \eqref{e:Gamma.star}. Now recall further from   
\eqref{e:SIGMA.prime} that $\SIGMA'(y)=-y\FF''(y)$, 
which by Proposition~\ref{p:second.derivative} is negative. It follows that $y_\star$ is unique. Finally, recall \eqref{e:deriv.Fy.over.y} that
	\[
	-\f{d}{dy}\bigg( \f{\FF(y)}{y}\bigg)
	= \f{\SIGMA(y)}{y^2}\,.
	\]
This means that in the range of $y$ satisfying \eqref{e:y.restriction}, the function $\FF(y)/y$ 
is decreasing in $y$ up to $y=y_\star$, and is increasing thereafter. This implies that the two characterizations of \eqref{e:def.eone} are equivalent, and concludes the proof.
\end{proof}

\noindent In the remainder of this subsection we prove the preceding two lemmas. As before, we let $\dq\equiv\dq_y$ be the solution of Proposition~\ref{p:sp}, and $\hq\equiv\hq_y\equiv\hSP(\dq_y)$. Let $x\equiv x_y\equiv \dq_y(\free)$ and $w\equiv w_y\equiv 1-\hq_y(\free)$.

\begin{proof}[Proof of Lemma~\ref{l:FF.y.estimate}] Recalling \eqref{e:clause.edge.eq} and \eqref{e:Fxwy}, we have (with some rearranging)
	\begin{align*}
	\FF(y)&=\log 2+\log\bigg( 
	\dfz_\zro(w) + \f{\dfz_\free(w)}{2}\bigg)
	-\alpha(k-1)\log\bigg\{ 1- (1-x)w(1-\AM)\bigg\}\\
	&=\log2+\log\f{\dfz_\zro(w)}{(1- (1-x)w(1-\AM))^d}
	+O(x_d(w))+\alpha\log
	\bigg\{ 1- (1-x)w(1-\AM)\bigg\}\,,
	\end{align*}
where we defined $x_d(w)$ similarly as \eqref{e:variable.w.to.x} but with $d$ in place of $d-1$:
	\beq\label{e:variable.w.to.x.d.version}
	x_d(w)\equiv
	\f{\dfz_\free(w)}
		{2\dfz_\zro(w)+\dfz_\free(w)}\,.
	\eeq
The estimates of Lemma~\ref{l:binom.Z.estimates} apply equally well with $d$ in place of $d-1$, so $x_d(w)\asymp x$, and
	\begin{align*}
	&\log\f{\dfz_\zro(w)}{(1- (1-x)w(1-\AM))^d}
	=-d\log\f{1- (1-x)w(1-\AM)}{1-w(1-\AM)}
	+O\bigg(\f1{e^{\Omega(k)}}\bigg)\\
	&\qquad=O\bigg(
	dw(1-\AM)x+
	\f1{e^{\Omega(k)}}\bigg)
	=O\bigg(
	\CC^{1/2} k x+
	\f1{e^{\Omega(k)}}\bigg)
	=O\bigg(\f1{e^{\Omega(k)}}\bigg)\,,
	\end{align*}
where we used \eqref{e:y.restriction}
to estimate $1-\AM\asymp \min\set{1,y}\asymp1/\CC^{1/2}$, and then obtained the final bound on $x$ using the result from Proposition~\ref{p:sp} that $\dq_y\in\MMstar$. Substituting into the expression for $\FF(y)$ and simplifying further gives
	\[
	\FF(y)
	= \log2
	+\alpha
	\log
	\bigg\{ 1- (1-x)w(1-\AM)\bigg\}
	+O\bigg(\f1{e^{\Omega(k)}}\bigg)
	=\log2
	-\alpha w(1-\AM)
	+O\bigg(\f1{e^{\Omega(k)}}\bigg)\,.
	\]
Recalling \eqref{e:clause.x.to.w}, we  have
$w=w(x)=[1-O(kx)]4/2^k$, while $\alpha=\CC 2^{k-1}\log2$. The result follows.\end{proof}

\begin{proof}[Proof of Lemma~\ref{l:energy.function.y.estimate}] Recall the definition \eqref{e:onersb.ee.of.y} of $\ee(y)$:
it consists of a variable term,
minus a clause term, minus an edge term. We estimate these separately:
\smallskip

\noindent\emph{Clause and edge terms.}
From the \textsc{sp} equations, the clause and edge terms of \eqref{e:onersb.ee.of.y} agree, and can be simplified as
	\[
	\sum_{\udw}\hph_k(\udw)\hat{\nu}_y(\udw)
	=\sum_{\ww} \eph(\ww)\bar{\nu}_y(\ww)
	= \f{2\dq_y(\zro)\hq_y(\one) e^{-y}}
	{1-2\dq_y(\zro)\hq_y(\one)(1-e^{-y})}
	=\f{ w(\AM-\tfrac12)}{(1-x)^{-1}-w(1-\AM)}
	\equiv
	\bar{\mathfrak{e}}(x,w)\,.
	\]
Recall from Proposition~\ref{p:sp} that $\dq_y\in\MM^\gamma\subseteq\MMstar$.
Expanding with respect to $x$ and $w$ gives
	\begin{align*}
	\alpha(k-1)\bar{\mathfrak{e}}(x,w)
	&=\f{\alpha(k-1) w(\AM-\tfrac12)}{1-w(1-\AM)}
	\bigg\{1+O(x)
	\bigg\}
	= \f{\alpha(k-1) w /(2e^y)}{1-w(1-\AM)}
	+O\bigg(\f{dwx}{e^y}\bigg)\\
	&= \f{\alpha(k-1) w}{2e^y}
	+O\bigg(
	\f{dw}{e^y}
	\Big(w (1-\AM)+x\Big)\bigg)
	= \f{\alpha(k-1) w}{2e^y}
	+O\bigg(\f{dw}{e^y \CC^{1/2}e^{\Omega(k)}}\bigg)\,,
	\end{align*}
where (similarly as in the proof of Lemma~\ref{l:FF.y.estimate})
we used \eqref{e:y.restriction} to estimate
$\AM-1\asymp\min\set{1,y} \asymp 1/\CC^{1/2}$, and used
$\dq_y\in\MMstar$ to bound $x$. Multiplying by $y$ and simplifying gives
	\beq\label{e:ebar.final.estimate}
	y\bigg\{\alpha(k-1)\bar{\mathfrak{e}}(x,w)
	-\f{\alpha(k-1) w }{2e^y}
	\bigg\}
	=O\bigg(\f{dwy}{e^y \CC^{1/2}e^{\Omega(k)}}\bigg)
	=
	O\bigg(\f{\CC^{1/2} k y}{e^y e^{\Omega(k)}}\bigg)
	=O\bigg(
	\f1{e^{\Omega(k)}}
	\bigg)\,,
	\eeq
where we again made use of \eqref{e:y.restriction}.\smallskip

\noindent\emph{Variable term.}
Define $\LdAM$, $\LdGM$, $\binLdAM$, $\binLdGM$
 similarly to \eqref{e:arith.Z}, \eqref{e:geom.Z},
 \eqref{e:binLAM}, \eqref{e:binLGM},
  but with $d$ in place of $d-1$: 
	\begin{alignat}{2}
	\label{e:LdAM}
	\P(\LdAM=\ell) 
	&= 
	\f{\AAA_{d,\ell}(w)\PPP_\ell}
		{\dfz_\zro(w)}
	\,,\quad
	& \P(\binLdAM=\ell) 
	&=\f{\AAA_{d,\ell}(w)}{(1-w+w\cdot\AM)^d}\,,
	\\
	\P(\LdGM=\ell) 
	&=\f{\GGG_{d,\ell}(w)\SSS_\ell}
		{\dfz_\free(w)}\,,
	& \P(\binLdGM=\ell) 
	&= \f{\GGG_{d,\ell}(w)}{(1-w+w\cdot\GM)^d}\,,
	\nonumber
	\end{alignat}
for all $0\le\ell\le d$, and with $\dfz_\zro(w)$ and $\dfz_\free(w)$ as defined by \eqref{e:Z.d.zero.as.sum} and \eqref{e:Z.d.free.as.sum}. Let
$x_d(w)$ be as defined by \eqref{e:variable.w.to.x.d.version}.
Conditional on $\LdAM=\ell$, let $X\sim\Bin(\ell,p=1/(1+e^y) )$. Then the variable term
in \eqref{e:onersb.ee.of.y} can be simplified as
	\begin{align}\nonumber
	\sum_{\uhw} \dph_d(\uhw)
		\dot{\nu}_y(\uhw)
	&= \f1{\dfz_y(\hq)}
	\sum_{\ell_\zro,\ell_\one}
	\min\set{\ell_\zro,\ell_\one}
	\binom{d}{\ell_\zro,\ell_\one}
	\f{\hq(\zro)^{\ell_\zro}\hq(\one)^{\ell_\zro}
	\hq(\free)^{d-\ell_\zro-\ell_\one}}
		{\exp\{ y\min\set{\ell_\zro,\ell_\one} \}}
	\\
	&= (1-x_d(w))\E\bigg\{
	\E\bigg(X\,\bigg|\, X< \f{\LdAM}{2}\bigg)
	\bigg\}
	+ x_d(w)
	\f{\E\LdGM}{2}
	\equiv \dot{\mathfrak{e}}(w)\,.
	\label{e:dot.e}
	\end{align}
Let $\ldam\equiv\E\binLdAM$ and $\ldgm\equiv\E\binLdGM$. We will show in \S\ref{ss:binomial}  (Lemma~\ref{l:Ld.estimates}) that
	\beq\label{e:energy.vertex.term.AM}
	\E\bigg\{
	\E\bigg(X\,\bigg|\, X< \f{\LdAM}{2}\bigg)\bigg\}
	= \f{dw}{2e^y}
	\bigg\{1- 
		O\bigg(\f1{\CC^{1/2}e^{\Omega(k)}}\bigg)
	\bigg\}\,.
	\eeq
Next, note that changing $d-1$ to $d$ does not affect the analysis of Lemma~\ref{l:simplified.binom.EL.estimates}, so we have $\E\LdGM= O(\ldgm)$ in general, and 
$\E\LdGM = \ldgm + O((\ldgm)^{1/2})$
if $\ldgm$ is large. We also note that
	\[
	x_d(w)
	\bigg\{\f{\ldgm}{2}
	-\ldam p
	\bigg\}
	\asymp
	\f{x_d(w) dw \min\set{1,y}}{e^{y/2}}
	\asymp 
	\f{dwx}{\CC^{1/2}e^{y/2}}
	\asymp 
	\f{\ldgm x}{\CC^{1/2}}\,,
	\]
where the estimate of $\min\set{1,y}$ comes from \eqref{e:y.restriction}. Consequently, if $\ldgm$ is large, we have
	\beq\label{e:energy.vertex.term.GM.large}
	x_d(w)
	\bigg\{
	\f{\E\LdGM}{2}
	-\E\bigg\{
	\E\bigg(X\,\bigg|\, X< \f{\LdAM}{2}\bigg)\bigg\}
	\bigg\}
	=O\bigg(x
	\bigg\{
	\f{\ldgm}{\CC^{1/2}} 
	+ (\ldgm)^{1/2}
	\bigg\}\bigg)
	=O\bigg(\f{\CC^{1/2}kx}{e^{y/2}}\bigg)
	\eeq
On the other hand, since $\lgm\asymp \CC k/e^{y/2}$, if $\lgm=O(1)$ then we must have $e^{y/2}\ge\Omega(\CC k)\ge \Omega(k)$, in which case \eqref{e:y.restriction} forces $\CC\asymp1$. In this case we can simply bound
	\beq\label{e:energy.vertex.term.GM.small}
	x_d(w)
	\bigg\{
	\f{\E\LdGM}{2}
	-\E\bigg\{
	\E\bigg(X\,\bigg|\, X< \f{\LdAM}{2}\bigg)\bigg\}
	\bigg\}
	=O\bigg(
	\f{dw x}{e^y}
	+
	\ldgm x
	\bigg)
	=O\bigg(\f{\ldgm x}{\CC^{1/2}}\bigg)
	=O\bigg(\f{\CC^{1/2}kx}{e^{y/2}}\bigg)\,.\eeq
By substituting
\eqref{e:energy.vertex.term.AM}, \eqref{e:energy.vertex.term.GM.large}, and \eqref{e:energy.vertex.term.GM.small}
into the formula
\eqref{e:dot.e} for 
$\dot{\mathfrak{e}}(w)$, we obtain 
	\beq\label{e:edot.final.estimate}
	y\bigg\{
	\dot{\mathfrak{e}}(w)
	-\f{dw}{2e^y}
	\bigg\}
	=O\bigg(
	\f{dwy}{\CC^{1/2} e^y e^{\Omega(k)}}
	+ \f{\CC^{1/2} k x y}{e^{y/2}}
	\bigg)
	= O\bigg(
	\f{\CC^{1/2}ky}{e^y e^{\Omega(k)}}
	+ \f{\CC^{1/2}k xy }{e^{y/2}}
	\bigg)
	= O\bigg(
	\f1{e^{\Omega(k)}}
	\bigg)\,,\eeq
where the last estimate
uses \eqref{e:y.restriction}
and $\dq\in\MMstar$.\smallskip

\noindent\emph{Combined.}
It follows from \eqref{e:ebar.final.estimate}~and~\eqref{e:edot.final.estimate} together that
	\[
	y\ee(y)
	=y\bigg\{\dot{\mathfrak{e}}(w)
		-\alpha(k-1)\bar{\mathfrak{e}}(x,w)
	\bigg\}
	= \f{y\alpha w}{2e^y}
	+O\bigg(\f1{e^{\Omega(k)}}\bigg)
	= \f{\CC y\log2}{e^y}
	+O\bigg(\f1{e^{\Omega(k)}}\bigg)\,,
	\]
where the last step uses that $w=[1-O(kx)]4/2^k$
and $\alpha=\CC 2^{k-1}\log2$.
\end{proof}

\subsection{Binomial estimates}
\label{ss:binomial}

We now prove the technical estimates used earlier in this section. Recall the classical binomial Chernoff bounds: if $X\sim\Bin(r,p)$ then it holds for any $t\ge0$ that
	\beq\label{e:chernoff.upper}
	\P\bigg(X\ge r p(1+t)\bigg)
	\le \exp\bigg\{ -r 
	\relent\bigg (p(1+t) \,
		\bigg|\, p\bigg)
	\bigg\}
	\le \exp\bigg\{- \f{r p t^2}{2+t}\bigg\}
	\le \exp\bigg\{- \f{r p 
		\min\set{t,t^2}}{3}\bigg\}\,.
	\eeq
In the lower tail a simpler bound holds: for all $0\le t\le 1$, we have
	\beq\label{e:chernoff.lower}
	\P\bigg(X\le r p(1-t)\bigg)
	\le \exp\bigg\{ -r 
	\relent\bigg (p(1-t) \,
		\bigg|\, p\bigg)
	\bigg\}
	\le\exp\bigg\{ - \f{rpt^2}{2} \bigg\}\,.
	\eeq
Throughout, $\relent(x|p)\equiv x\log (x/p)+(1-x)\log[(1-x)/(1-p)]$, the binary relative entropy function. We will make frequent use of \eqref{e:chernoff.upper} and \eqref{e:chernoff.lower} in the remainder of this section.

\begin{lem}\label{l:binom.estimates}
Let $\PPP_\ell$, $\SSS_\ell$ be as defined by \eqref{e:arith.prob}~and~\eqref{e:geom.prob}. 
For all $0\le\ell\le d-1$ we have
	\[
	\SSS_\ell = 
	\Ind{\ell\textup{ even}}
	\bigg(\f2{\pi\ell}\bigg)^{1/2}
	\bigg[1+O\bigg(\f1{1+\ell}\bigg)\bigg]\,.
	\]
It holds uniformly over all $y\ge0$ that
$1-\PPP_\ell\le \exp\{-\Omega(\ell\min\set{y,y^2})\}$.
We have
	\[1-\PPP_\ell\le 
	\exp\bigg\{-\f{\ell y^2}{8}
	\bigg[1 + O(y^2)\bigg]
	\bigg\}
	\]
for $y\ge0$ small enough.

\begin{proof}The estimate on $\SSS_\ell$ follows immediately from Stirling's approximation. The binomial Chernoff bound gives
	\[
	1-\PPP_\ell
	\le 
	\exp\bigg\{-\ell\relent
		\bigg(\f12
		\,\bigg|\, \f1{1+e^y}
		\bigg)\bigg\}
	=\exp\bigg\{-
	\Theta\bigg(\ell \min\set{y,y^2}\bigg)
	\bigg\}
	\]
uniformly over all $y\ge0$. The estimate for small $y$ follows by Taylor expanding the relative entropy function.\end{proof}
\end{lem}

\noindent We now turn to the proofs of Lemmas~\ref{l:binom.Z.estimates}~and~\ref{l:simplified.binom.EL.estimates} which were introduced in \S\ref{ss:contraction}.

\begin{proof}[Proof of Lemma~\ref{l:binom.Z.estimates}]
As before we abbreviate $\dz_\zro\equiv \dz_\zro(w)$ and $\dz_\free\equiv \dz_\free(w)$.\smallskip

\noindent\emph{Bounds on $\dz_\zro$.} From the definition 
\eqref{e:def.dz.zro}, together with the trivial bound $\PPP_\ell\le1$, we immediately conclude
	\[\f{\dz_\zro}{(1-w+w\cdot\AM)^{d-1}}\le1\,.\]
By the lower bound on $\PPP_\ell$ from Lemma~\ref{l:binom.estimates}, we also have 
	\begin{align*}
	&1-\f{\dz_\zro}{(1-w+w\cdot\AM)^{d-1}}
	\le  \bigg(
	\f{1-w+\exp\{-\Omega(\min\set{y,y^2})\} w\cdot\AM }{1-w+w\cdot\AM}
	\bigg)^{d-1}\\
	&\qquad
	= \bigg(1-
	\f{w\cdot\AM[1-
		\exp\{-\Omega(\min\set{y,y^2})\}]}
		{1-w+w\cdot\AM}
	\bigg)^{d-1}
	\le \f1{\exp\{\Omega(dw\min\set{y^2,1})\}}
	\le \f1{\exp\{\Omega(k)\}}\,,
	\end{align*}
since $dw\min\set{y^2,1}\asymp \CC k\min\set{y^2,1} \asymp k$ by \eqref{e:y.restriction}. This proves \eqref{e:bounds.Z.zro}.
\smallskip

\noindent\emph{Bounds on $\dz_\free$.} 
From the definition \eqref{e:def.dz.free}, we have trivially
	\beq\label{e:lgm.small.sandwiching}
	\f1{\exp\{O(dw\cdot\GM) \}}
	\le
	\bigg( \f{1-w}{1-w+w\cdot\GM}\bigg)^{d-1}
	\le\f{\dz_\free}{(1-w+w\cdot\GM)^{d-1}} \le 1\,,
	\eeq
where the left-hand side is $\Omega(1)$ as long as $dw\cdot\GM=O(1)$. It remains to consider what happens when $dw\cdot\GM$ is large. From the estimate of Lemma~\ref{l:binom.estimates} we have
	\[\f{\dz_\free}{(1-w+w\cdot\GM)^{d-1}}
	=
	\E\bigg[	\Ind{\binLGM\textup{ even}}
	\bigg(\f{2}{\pi \binLGM}\bigg)^{1/2}
	\bigg\{1 + O\bigg(\f1{1+\binLGM}\bigg)\bigg\}
	\bigg]\,.
	\]
Since $\lgm\asymp dw \cdot\GM$, it follows from the Chernoff bound 
\eqref{e:chernoff.upper}
that for a large enough absolute constant $C$,
	 \[\P\bigg(|\binLGM - \lgm|
	 \ge  C (\lgm \log \lgm)^{1/2}\bigg) 
	 \le \f1{(\lgm)^{10}}\]
(where the power $10$ is somewhat arbitrary, but large enough for our purposes). This implies
	\[\f{\dz_\free}{(1-w+w\cdot\GM)^{d-1}}
	=
	\bigg(\f{2}{\pi \lgm}\bigg)^{1/2}
	\bigg\{1 + O\bigg(
		\f{(\lgm\log\lgm)^{1/2}}{\lgm}
		\bigg)\bigg\}
		\P(\binLGM\textup{ even})
	+O\bigg( \f1{(\lgm)^{3/2}}\bigg)\,.
	\]
We then note that
	\beq\label{e:prob.even.GM}
	\P(\binLGM\textup{ even})
	=  \f{1+ \E[(-1)^{\binLGM}]}{2}
	= \f12\bigg[1 + 
		\bigg(\f{1-w-w\cdot\GM}{1-w+w\cdot\GM}
		\bigg)^{d-1}\bigg]
	=\f12 + O\bigg(\f{1}{\exp\{\Omega(\lgm)\}}\bigg)\,.
	\eeq
Combining with the previous estimate gives 
	\[
	\f{\dz_\free}{(1-w+w\cdot\GM)^{d-1}}
	=\bigg(\f{1}{2\pi \lgm}\bigg)^{1/2}
	\bigg\{1+ O\bigg(
	\f{(\log\lgm)^{1/2}}{(\lgm)^{1/2}}
	\bigg)\bigg\}
	\asymp \f1{(\lgm)^{1/2}}\,,
	\]
from which \eqref{e:bounds.Z.free} follows.
\end{proof}

\begin{proof}[Proof of Lemma~\ref{l:simplified.binom.EL.estimates}] We abbreviate $p\equiv 1/(e^y+1)$ as well as
	\[
	\pam \equiv \f{w\cdot \AM}{1-w+w\cdot\AM}\,,\quad
	\pgm \equiv \f{w\cdot \GM}{1-w+w\cdot\GM}\,.
	\]
Recall the definitions \eqref{e:arith.Z}
and \eqref{e:geom.Z} of $\LAM$ and $\LGM$.
\smallskip

\noindent\emph{Bounds on $\E\LAM$.} By Jensen's inequality we have
	\beq\label{e:jensen.AM}
	\bigg(|\E\LAM-\lam|\bigg)^2
	\le
	\E\bigg[(\LAM-\lam)^2\bigg]
	=
	\sum_{\ell=0}^{d-1}
	\f{\AAA_\ell\PPP_\ell(\ell-\lam)^2}
		{\dz_\zro}
	\asymp
	\sum_{\ell=0}^{d-1}
	\f{\AAA_\ell\PPP_\ell(\ell-\lam)^2}
	{(1-w+w\cdot\AM)^{d-1}}\eeq
where the last estimate is by 
\eqref{e:bounds.Z.zro} from Lemma~\ref{l:binom.Z.estimates}. Since $\PPP_\ell\le1$, the last expression above is upper bounded by
	\[
	\sum_{\ell=0}^{d-1}
	\f{\AAA_\ell(\ell-\lam)^2}
	{(1-w+w\cdot\AM)^{d-1}}
	=\Var\binLAM
	=(d-1) \pam(1-\pam)
	\asymp (d-1)\pam
	\asymp dw\,,
	\]
which proves the first claim \eqref{e:simplified.ELAM.estimate}.\smallskip

\noindent\emph{Bounds on $\E\LGM$
when $\lgm$ is large.} Abbreviate $J$ for the set of even integers between $\lgm/2$ and $2\lgm$. Recalling the estimate \eqref{e:bounds.Z.free} from Lemma~\ref{l:binom.Z.estimates}, we have
	\beq\label{e:LGM.decompose}
	|\E\LGM-\lgm|
	\le \E|\LGM-\lgm|
	\le 
	\sum_{\ell\in J}
	\f{|\ell-\lgm|
	\GGG_\ell\SSS_\ell}{\dz_\free}
	+\sum_{\ell\notin J}
	\f{(\lgm)^{1/2}
	|\ell-\lgm|\GGG_\ell}{(1-w+w\cdot\GM)^{d-1}}
	\,.\eeq
The lower tail Chernoff bound
\eqref{e:chernoff.lower}
gives
	\[
	\sum_{\ell \le \lgm/2}
	\f{(\lgm)^{1/2}
	|\ell-\lgm|\GGG_\ell}{(1-w+w\cdot\GM)^{d-1}}
	\le (\lgm)^{3/2}
	\P\bigg(
	\Bin(d-1,\pgm) \le \f{\lgm}{2}\bigg)
	\le
	\f{(\lgm)^{3/2}}{\exp\{\Omega(\lgm)\}}\,,
	\]
The upper tail Chernoff bound \eqref{e:chernoff.upper} gives
	\begin{align}
	\sum_{\ell \ge 2\lgm}
	\f{(\lgm)^{1/2}
	|\ell-\lgm|\GGG_\ell}{(1-w+w\cdot\GM)^{d-1}}
	&\le
	\sum_{j\ge0}
	(\lgm)^{1/2}(\lgm+j)
	\P\bigg(
	\Bin(d-1,\pgm) \ge 2\lgm +j \bigg)
	\label{e:decompose.upper.tail}
	\\
	&\le
	\f{(\lgm)^{3/2}}{\exp\{\Omega(\lgm)\}}
	+ \sum_{j\ge0} 
	\f{(\lgm)^{1/2} j}{\exp\{\Omega(\lgm+j)\}}
	\le
	\f{(\lgm)^{3/2}}{\exp\{\Omega(\lgm)\}}\,.
	\nonumber
	\end{align}
By Lemma~\ref{l:binom.estimates} we have $\SSS_\ell \asymp 1/(\lgm)^{1/2}$ uniformly over $\ell\in J$.
We then have (similarly to \eqref{e:jensen.AM}, and using \eqref{e:bounds.Z.free} again)
	\begin{align*}
	\bigg(\f1{\dz_\free}
	\sum_{\ell\in J}
	|\ell-\lgm|
	\GGG_\ell\SSS_\ell\bigg)^2
	&\le
	\f1{\dz_\free}
	\sum_{\ell\in J}
	(\ell-\lgm)^2
	\GGG_\ell\SSS_\ell
	\asymp \sum_{\ell\in J}
	\f{(\ell-\lgm)^2\GGG_\ell}
		{(1-w+w\cdot\GM)^{d-1}}\,,\\
	&\le \sum_{\ell=0}^{d-1}
	\f{(\ell-\lgm)^2\GGG_\ell}
		{(1-w+w\cdot\GM)^{d-1}}
	=\Var\binLGM
	=(d-1)\pgm(1-\pgm) \asymp dw\,,
	\end{align*}
Substituting these estimates back into \eqref{e:LGM.decompose} gives the claim
\eqref{e:simplified.ELGM.estimate} in the case that $\lgm$ is large.\smallskip

\noindent\emph{Bounds on $\E\LGM$ when $\lgm=O(1)$.} Write $\ddot{\P}$ for the law of $\binLGM$ conditioned to be even. It is straightforward to check that $\SSS(\ell)\equiv\SSS_\ell$ is nonincreasing with respect to $\ell$ even. It follows that $\binLGM$ and $\SSS(\binLGM)$ have nonpositive covariance under $\ddot{\P}$: to see this, let $L,L'$ be independent samples from the law $\ddot{\P}$, and note
	\beq\label{e:standard.covar.bound}	
	\Cov_{\ddot{\P}}\bigg( \SSS(\binLGM),
		\binLGM\bigg)
	=\f12 \ddot{\E}\bigg\{
	\Big(\SSS(L)-\SSS(L')\Big)
	\Big( L-L'\Big)\bigg\}
	\le0\,,
	\eeq
where the last inequality holds since the random variable inside the expectation is nonpositive almost surely. Thus
	\begin{align}
	\label{e:E.LGM.covar.bound}	
	\E \LGM
	&= \f{\ddot{\E}[\SSS(\binLGM)\binLGM]}
		{\ddot{\E}[ \SSS({\binLGM})]}
	\le\ddot{\E} \binLGM
	\le \f{\E(\binLGM)}{\P(\binLGM\textup{ even})}\\
	\nonumber
	&\le \lgm
	\bigg(\f{1-w+w\cdot\GM}{1-w}\bigg)^{d-1}
	\le \lgm
	\exp\{O(\lgm)\}
	\le O(\lgm)\,,
	\end{align}
by the assumption that $\lgm \asymp dw\cdot\GM =O(1)$. This proves \eqref{e:simplified.ELGM.estimate} in the case $\lgm=O(1)$.\end{proof}

\begin{proof}[Proof of Lemma~\ref{l:one.sided}]
We first use simple correlation inequalities (such as \eqref{e:E.LGM.covar.bound}) to obtain one-sided improvements on the bounds of Lemma~\ref{l:simplified.binom.EL.estimates}.\smallskip

\noindent\emph{Upper bound on $\E\LGM$.} Recall from \eqref{e:E.LGM.covar.bound} that 
$\E \LGM\le\ddot{\E}\binLGM$, which is $O(\lgm)$ in the case $\lgm=O(1)$. If $\lgm$ is large, then
we can use the binomial moment-generating function to estimate
	\begin{align*}
	\E\Big( (-1)^{\binLGM}\Big)
	&= (1-2\pgm)^{d-1}
	= \bigg(\f{1-w-w\cdot\GM}{1-w+w\cdot\GM}\bigg)^{d-1}
	= O\bigg(\f1{\exp\{\Omega(\lgm)\}}\bigg)\,,\\
	\E\Big(\binLGM (-1)^{\binLGM}\Big)
	&=\f{d}{d\theta}
	(1-\pgm+\pgm e^\theta)^{d-1}\bigg|_{\theta=i\pi}
	=O\bigg( \f{\lgm}
		{\exp\{\Omega(\lgm)\}}\bigg)
	=  O\bigg(\f1{\exp\{\Omega(\lgm)\}}\bigg)\,.
	\end{align*}
Rearranging these bounds gives
	\begin{align}\nonumber
	\P(\binLGM\textup{ even})
	&= \f{1+ \E[(-1)^{\binLGM}]}{2}
	= \f12 
	+ O\bigg(\f1{\exp\{\Omega(\lgm)\}}\bigg)\,,\\
	\ddot{\E}(\binLGM)
	&=\f{ \E\binLGM + \E \binLGM (-1)^{\binLGM}}
	{2\P(\binLGM\textup{ even})}
	= \lgm + O\bigg(\f1{\exp\{\Omega(\lgm)\}}\bigg)\,.
	\label{e:binomial.mgf.deriv}
	\end{align}
We therefore conclude that if $\lgm$ is large then
	\beq\label{e:LGM.ubd}
	\E\LGM
	\le \lgm + 
	O\bigg( \f1{\exp\{\Omega(\lgm)\}}\bigg)\,,
	\eeq
which is an improvement on the upper bound on $\E\LGM$ obtained in Lemma~\ref{l:simplified.binom.EL.estimates}.
\smallskip

\noindent\emph{Lower bound on $\E\LAM$.} As in the proof of Lemma~\ref{l:stationarity},  for integers $i\ge1$ let $I_i$ be i.i.d.\ Bernoulli random variables with $\E I_i =  p \equiv 1/(e^y+1)$. For any finite subset $S$ of positive integers let $Y(S)$ be the sum of $I_i$ over $i\in S$. Abbreviating $[\ell]\equiv \set{1,\ldots,\ell}$, we have $\PPP_\ell \equiv \P(E_\ell)$ where $E_\ell\equiv \set{Y([\ell]) < \ell/2}$. We then note that
	\begin{align*}
	E_{\ell+2}\setminus E_\ell
	&=\bigg\{ \f{\ell}{2}\le Y([\ell])
	\le Y([\ell+2])< \f{\ell+2}{2}\bigg\}
	= \bigg\{
	Y([\ell])= \bigg\lceil
	\f{\ell}{2}\bigg\rceil,
	Y(\set{\ell+1,\ell_2})=0
	\bigg\}\,,\\
	E_\ell \setminus E_{\ell+2}
	&= \bigg\{
	\f{\ell}{2}-1 \le
	Y([\ell+2])-2 \le Y([\ell]) < \f{\ell}{2}
	\bigg\}
	=\bigg\{
	Y([\ell]) = \bigg\lceil\f{\ell}{2}-1\bigg\rceil,
	Y(\set{\ell+1,\ell_2})=2
	\bigg\}\,.
	\end{align*}
Writing $p_{\ell}(k)\equiv \P(\Bin(\ell,p)=k)$, the above implies
	\beq\label{e:P.ell.skip.two}
	\PPP_{\ell+2}-\PPP_\ell
	= p_\ell
	\bigg(\bigg \lceil \f{\ell}{2} \bigg\rceil \bigg)
	 (1-p)^2
	-p_\ell
	\bigg(\bigg \lceil \f{\ell}{2} \bigg\rceil -1
	\bigg) p^2
	= p_\ell
	\bigg(\bigg \lceil \f{\ell}{2} \bigg\rceil \bigg)
	 (1-p)^2
	 \bigg\{1 - \f{e^{-y} \lceil\ell/2\rceil}
	 {\ell+1- \lceil\ell/2\rceil}
	  \bigg\}\ge0\,,
	\eeq
since $2 \lceil\ell/2\rceil \le \ell+1$ and $y\ge0$.
Now let $\dot{\P}$ and $\ddot{\P}$ denote the laws of 
the binomial random variable $\binLAM$ conditioned to be odd and even, respectively. Write $\PPP(\ell)\equiv\PPP_\ell$. The bound \eqref{e:P.ell.skip.two}
 implies (by the same argument as in \eqref{e:standard.covar.bound})
that $\binLAM$ and $\PPP(\binLAM)$ have nonnegative covariance under $\dot{\P}$ and under $\ddot{\P}$. 
Rearranging gives
	{\setlength{\jot}{0pt}\begin{align}
	 \nonumber
	\E( \LAM \,|\, \LAM\textup{ even})
	&\ge \ddot{\E} \binLAM\,,\\
	\E( \LAM \,|\, \LAM\textup{ odd})
	&\ge \ddot{\E} \binLAM\,.
	\label{e:L.AM.stoch.dom}
	\end{align}}%
By the argument of \eqref{e:binomial.mgf.deriv} we have
	\begin{align*}
	\P(\binLAM\textup{ even})
	&=
	\f12 + O\bigg(\f1{\exp\{\Omega(\CC k)\}}\bigg)\,,\\
	\ddot{\E}(\binLAM)
	&= \lam + O\bigg(\f1{\exp\{\Omega(\CC k)\}}\bigg)
	= \dot{\E}(\binLAM)\,,
	\end{align*}
having used that $\lam\asymp dw\asymp \CC k$.
This gives
	\beq\label{e:LAM.lbd}
	\E\LAM\ge\lam 
	- O\bigg( \f1{\exp\{\Omega(\CC k)\}}\bigg)
	\eeq
which is an improvement on the lower bound on $\E\LAM$ obtained in Lemma~\ref{l:simplified.binom.EL.estimates}.
\smallskip

\noindent\emph{Conclusion.} First of all we note that
\eqref{e:y.restriction} implies
	\[
	\lam-\lgm
	= 
	\f{(d-1)w(1-w)(\AM-\GM)}
	{(1-w+w\cdot\AM)(1-w+w\cdot\GM)}
	\asymp dw(\AM-\GM) 
	\ge \Omega(k)\,.
	\]
If $\lgm$ is large, then combining with \eqref{e:LGM.ubd} and \eqref{e:LAM.lbd} gives
	\[
	\E(\LAM-\LGM)
	\ge \lam-\lgm
	-O\bigg(\f1{\exp\{\Omega(\lgm)\}}\bigg)
	\ge \Omega( k)\,.
	\]
If $\lgm=O(1)$ then 
$\E(\LAM-\LGM)\ge \lam-O(\lgm)
\ge \Omega(\CC k)-O(1) \ge \Omega(k)$. The result follows.\end{proof}

\noindent
The next lemma was used  in the proof of Lemma~\ref{l:energy.function.y.estimate}
 (in \S\ref{ss:ground.state}).

\begin{lem}\label{l:Ld.estimates}
In the setting of Proposition~\ref{p:eone.defn},
let $\LdAM$ and $\binLdAM$ be the random variables with laws given by \eqref{e:LdAM}. Conditional on $\LdAM=\ell$ let $X\sim\Bin(\ell,1/(e^y+1))$. Then
	\[
	\E\bigg\{\E\bigg(
		X \,\bigg|\,X<\f{\LdAM}{2}\bigg)\bigg\}
	=\f{dw}{2e^y}
	\bigg\{
	1
	- O\bigg( 
	\f1{\CC^{1/2}e^{\Omega(k)}}
	\bigg)\bigg\}\,.
	\]
(On the left-hand side, the inner expectation is over $X$ conditional on $\LdAM$, and the outer expection is over $\LdAM$.)

\begin{proof}
Let $p\equiv1/(1+e^y)$, and
$X\sim\Bin(\ell,p)$. We also define $\pam$ as before, so $\ldam \equiv \E\binLdAM= d\pam$, and
	\beq\label{e:simplified.EXL}
	(\E\binLdAM)p
	= \f{dw e^{-y}/2}{1-w(1-\AM)}
	= \f{dw}{2e^y}
	\bigg\{1 + O(w(1-\AM))\bigg\}
	= \f{dw}{2e^y}
	\bigg\{1 + O\bigg(
	\f1{2^k \CC^{1/2}}
	\bigg)\bigg\}\,,
	\eeq
where the last step used \eqref{e:y.restriction}
to estimate $1-\AM\asymp\min\set{1,y}\asymp1/\CC^{1/2
}$. We shall now consider the expectation of $X$ conditioned on (typical) $\LdAM$, then average over the law of $\LdAM$.\smallskip

\noindent
\emph{Conditional on $\LdAM=\ell$.}  Abbreviate
$E(\ell)\equiv \E(X\,|\,X<\ell/2)$. Then
	\beq\label{e:X.reflect}
	E(\ell)-\ell p
	=\f{ \E(X - \E X ; X<\ell/2)}{\P(X<\ell/2)}
	=\f{-\E(X-\E X;X\ge \ell/2)
	}{\P(X<\ell/2)}
	\le -\f{
	\P(X\ge\ell/2) \ell(1/2-p)
	}{\P(X<\ell/2)} \le0\,.\eeq
Now suppose for the moment that
$\ell\asymp \ldam$.
By Lemma~\ref{l:binom.estimates} and \eqref{e:y.restriction}, for 
such $\ell$ we have
	\[\P\bigg(X\ge \f{\ell}{2}\bigg)
	=1-\PPP_\ell
	\le\exp\bigg\{-\Omega
	\bigg(\ell\min\set{y,y^2}\bigg)\bigg\}
	\le\exp\bigg\{-\Omega
	\bigg(\CC k\min\set{1,y^2}\bigg)\bigg\}
	\le \f1{e^{\Omega(k)}}\,.
	\]
If $y$ is small then $\ell(1/2-p)\asymp \ell y$, so for a large enough constant $C$ we can bound
	\[
	\E\bigg(X-\E X; X\ge \f{\ell}{2}\bigg)
	\le 2C\ell y
	\P\bigg(X\ge \f{\ell}{2}\bigg)
	+ \ell\P\bigg(X\ge \f{\ell}{2} + C\ell y\bigg)
	\le O\bigg(
	\f{\ell y}{\exp\{\Omega( \ell y^2)\}}\bigg)
	\le O\bigg( \f{\ell p}
		{\CC^{1/2} e^{\Omega(k)}}\bigg)\,,
	\]
where the last step uses that for small $y$ we have $p\asymp1$, and $\ell y^2 \asymp \CC k y^2 \asymp k$ by \eqref{e:y.restriction}. On the other hand, if $y\ge\Omega(1)$ then $\ell(1/2-p)\asymp\ell$, and we can instead bound
	\[\E\bigg(X-\E X; X\ge \f{\ell}{2}\bigg)
	\le
	\ell \P\bigg(X\ge \f{\ell}{2}\bigg)
	\le O\bigg(\f{\ell}{e^{\Omega(ky)}}\bigg)
	\le O\bigg( \f{\ell p}{\CC^{1/2}
	e^{\Omega(ky)}}\bigg)\,,
	\]
where the last step uses that
for large $y$ we have
$p\asymp e^{-y}$, and
 $\CC\asymp1$ by \eqref{e:y.restriction}. Substituting into \eqref{e:X.reflect} gives
	\beq\label{e:E.ell.bound}
	E(\ell) = \E\bigg(
	X \,\bigg|\,X<\f{\ell}{2}\bigg)
	\ge 
	\ell p
	\bigg\{1
	-  O\bigg( \f1{\CC^{1/2}
	e^{\Omega(k\max\set{1,y})}}\bigg)
	\bigg\}
	\eeq
for $\ell\asymp\ldam$,
and for all $y\ge0$ satisfying \eqref{e:y.restriction}.\smallskip

\noindent
\emph{Averaging over $\LdAM$.}
Combining \eqref{e:X.reflect} with the lower tail Chernoff bound
\eqref{e:chernoff.lower} gives
	\[
	\E\bigg( \f{E(\binLdAM)}{p}	 ; \binLdAM
		\le \f{\ldam}{2}\bigg)
	\le \E\bigg( \binLdAM ; \binLdAM \le 
	\f{\ldam}{2}\bigg)
	\le O\bigg( \f{\ldam}
		{\exp\{\Omega(\ldam)\}} \bigg)\,.
	\]
Combining \eqref{e:X.reflect} with the upper tail Chernoff bound
\eqref{e:chernoff.lower} gives
	\[
	\E\bigg( \f{E(\binLdAM)}{p}	; \binLdAM
		\ge 2\ldam\bigg)
	\le\E\bigg( \binLdAM ; \binLdAM \ge 
	2\ldam\bigg)
	\le 
	\sum_{j\ge0}
	O\bigg( \f{2\ldam+j}
		{\exp\{\Omega(\ldam+j)\}} \bigg)\,.
	\]
Now abbreviate $J$ for the integers between $\ldam/2$ and $2\ldam$. Combining \eqref{e:X.reflect} with the last two bounds gives
	\[
	0\le \E\bigg(\binLdAM p- E(\binLdAM)\bigg)
	\le \E \bigg(\binLdAM p- E(\binLdAM) ;
		\binLdAM\in J\bigg)
	+ O\bigg( \f{\ldam p}
		{\exp\{\Omega(\ldam)\}} \bigg)\,.
	\]
Recall that $\ldam\asymp \CC k$.
Combining with \eqref{e:E.ell.bound} gives
that the right-hand side above is
	\[
	\le
	O\bigg(
	\f{\ldam p}{\CC^{1/2}e^{\Omega(k\max\set{1,y})}}
	+\f{\ldam p}{e^{\Omega(\CC k)}}
	\bigg)
	\le
	O\bigg(
	\f{\ldam p}{\CC^{1/2} e^{\Omega(k)}}
	\bigg)\,.
	\]
Rearranging this bound gives
	\[
	(\E\binLdAM)p\bigg\{
	1 - O\bigg(\f1{\CC^{1/2} e^{\Omega(k)}}
	\bigg)\bigg\}
	\le \E[E(\binLdAM)]
	\le (\E\binLdAM)p\,,
	\]
and the claim follows by recalling \eqref{e:simplified.EXL}.
\end{proof}
\end{lem}

\noindent We conclude this section with some estimates which will be used in the next two sections. For integers $i\ge0$ define
	\begin{align}\nonumber
	S_i
	&\equiv
	\sum_{\ell_1,\ell_\zro}
	\Ind{\ell_1=\ell_\zro + i}
	\binom{d-2}{\ell_\zro,\ell_\one}
	\f{\hq(\zro)^{\ell_\zro}
	\hq(\one)^{\ell_\one}
	\hq(\free)^{d-2-\ell_\zro-\ell_\one}}
	{\exp\{ y\ell_\zro \}}\\
	&= \f1{\GM^i}
	\sum_{\ell=i}^{d-2}
	\binom{d-2}{\ell}
	(w\cdot\GM)^\ell (1-w)^{d-2-\ell}
	\binom{\ell}{(\ell-i)/2}
	\f1{2^\ell}\,,
	\label{e:def.S.i}
	\end{align}
where $\ell$ contributes to the sum only if it has the same parity as $i$. Let
	\beq\label{e:def.S.geq.i}
	S_{\ge i} \equiv \sum_{t\ge0} S_{i+t}\,.
	\eeq
The next lemma is used in the proofs of Lemmas~\ref{l:determinant}~and~\ref{l:Gyy} (in \S\ref{ss:convexity}).

\begin{lem}\label{l:S.estimates} In the setting of Proposition~\ref{p:sp}, let
$S_i$ be defined by \eqref{e:def.S.i}. As long as $0\le i\le O(1)$ we have
	\beq\label{e:S.estimate.small.l}
	\f{S_i(\GM)^i}{S_0} \le 
	(\lgm)^i
	\bigg\{1+O(w\cdot\GM)\bigg\}\,.
	\eeq
If $\lgm$ is large we then have a stronger bound:
as long as $0\le i\le O(1)$,
	\beq\label{e:S.estimate.large.l}
	\f{S_i(\GM)^i}{S_0}
	= 1 + O\bigg( \f{(\log \lgm)^{1/2}}{(\lgm)^{1/2}}\bigg)\,.
	\eeq

\begin{proof}
If $\lgm=O(1)$, essentially the same calculation as \eqref{e:lgm.small.sandwiching} gives
	\[\Omega(1)\le
	\f{(1-w)^{d-2}}{(1-w(1-\GM))^{d-2}}
	\le
	\f{S_0}{(1-w(1-\GM))^{d-2}}
	\le1\,.
	\]
Next we rewrite the expression \eqref{e:def.S.i} for $S_{i+1}$ as
	\begin{align*}
	s_{i+1}&\equiv
	\f{S_{i+1}\cdot\GM^{i+1}}{(1-w(1-\GM))^{d-2}}
	= \sum_{\ell=i}^{d-2}
	\binom{d-2}{\ell+1}
	(\pgm)^{\ell+1} (1-\pgm)^{d-2-(\ell+1)}
	\binom{\ell+1}{(\ell-i)/2}
	\f1{2^{\ell+1}}\\
	&=\sum_{\ell=i}^{d-2}
	\binom{d-2}{\ell}
	\f{(d-2-\ell) \pgm}{(\ell+1)(1-\pgm)}
	(\pgm)^\ell (1-\pgm)^{d-2-\ell}
	\binom{\ell}{(\ell-i)/2}\f1{2^\ell}
	\f{\ell+1}{\ell+i+2}\\
	&\le\f{\lgm}{1-\pgm}
	\sum_{\ell=i}^{d-2}
	\binom{d-2}{\ell}
	(\pgm)^\ell (1-\pgm)^{d-2-\ell}
	\binom{\ell}{(\ell-i)/2}\f1{2^\ell}
	= \f{\lgm}{1-\pgm} s_i\,.
	\end{align*}
Since $\pgm=\Theta(w\cdot\GM)$, we conclude that in general we have
	\[
	\f{S_{i+1}\cdot\GM}{S_i}
	=\f{s_{i+1}}{s_i}
	\le \lgm
	\bigg\{1+O(w\cdot\GM)\bigg\}\,,
	\]
which implies \eqref{e:S.estimate.small.l}.
Now assume that $\lgm$ is large, and let $I$ be the integers $\ell$ with $|\ell-(d-2)\pgm| \le C(\lgm\log\lgm)^{1/2}$. It follows from the Chernoff bound that for a large enough constant $C$, the contribution to the above sum from $\ell\notin I$ is $O((\lgm)^{-10})$. This gives a better comparison between the $s_i$:
	\begin{align*}
	s_{i+1}
	&=O\bigg(\f1{(\lgm)^{10}}\bigg)
	+\sum_{\ell\in I}
	\bigg\{
	1 + O\bigg( \f{(\log \lgm)^{1/2}}{(\lgm)^{1/2}}
	\bigg)\bigg\}
	\binom{d-2}{\ell}
	(\pgm)^\ell(1-\pgm)^{d-2-\ell}
	\binom{\ell}{(\ell-i)/2}\f1{2^\ell}\\
	&= O\bigg(\f1{(\lgm)^{10}}\bigg)
	+s_i	\bigg\{
	1 + O\bigg( \f{(\log \lgm)^{1/2}}{(\lgm)^{1/2}}
	\bigg)\bigg\}
	= s_i	\bigg\{
	1 + O\bigg( \f{(\log \lgm)^{1/2}}{(\lgm)^{1/2}}
	\bigg)\bigg\}\,,
	\end{align*}
which implies \eqref{e:S.estimate.large.l}.
\end{proof}
\end{lem}

\section{Gardner threshold and 2RSB perturbation}
\label{s:gardner}

In this section we evaluate the Gardner threshold and prove our main result. The explicit evaluation of the stability matrices is given in \S\ref{ss:eval.matrices}, and the asymptotics of the Gardner eigenvalue are extracted in \S\ref{ss:aux.matrices} to prove Proposition~\ref{p:gardner.defn}. The proof of the main result Theorem~\ref{t:main} is completed in
 \S\ref{ss:pert.ii}--\ref{ss:proof.main}.
 Let $\dq_y$ be as given by Proposition~\ref{p:sp}, and $\hq_y\equiv\hSP_y(\dq_y)$. For the most part we will suppress $y$ from the notation and write simply $\rho_{\dw}\equiv\dq_y(\dw)$ and $\psi_{\hw}\equiv\hq_y(\hw)$. For any integers $a\le b$ we write $x_{a:b}\equiv (x_a,\ldots,x_b)$.
 
\subsection{Evaluation of the stability matrix}
\label{ss:eval.matrices}

We decompose the stability matrix \eqref{e:full.stability.matrix} as a product of two matrices, as follows. Define the \bemph{clause stability matrix} to be the $9\times 9$ matrix $\hB$ with entries
	\[
	\hB_{\hw\hs,\dw\ds}
	\equiv
	\f{\displaystyle\rho_{\dw}
	\sum_{\dw_{3:k}}
	\I\bigg\{
	\begin{array}{c}
	\hw = \hWP(\dw\dw_{3:k})\\
	\hs = \hWP(\ds\dw_{3:k})
	\end{array}
	\bigg\}\prod_{j=3}^k \rho_{\dw_j}}
	{\displaystyle
	\sum_{\dv_{2:k}}
	\Ind{\hw = \hWP(\dv_{2:k})}
	\prod_{j=2}^k \rho_{\dv_j}}
	= \f{\rho_{\dw}}{\psi_{\hw}}
	\sum_{\dw_{3:k}}
	\I\bigg\{
	\begin{array}{c}
	\hw = \hWP(\dw\dw_{3:k})\\
	\hs = \hWP(\ds\dw_{3:k})
	\end{array}
	\bigg\}\prod_{j=3}^k \rho_{\dw_j}
	\equiv
	\f{\rho_{\dw}}{\psi_{\hw}} 
	\hN_{\hw\hs,\dw\ds}
	\,,
	\]
where the last equality is the definition of a $9\times 9$ matrix $\hN$.
Similarly, define the \bemph{variable stability matrix} to be the $9\times9$ matrix $\dB$ with entries
	\[
	\dB_{\dw\ds,\hw\hs}
	=\f{\displaystyle\psi_{\hw}
	\sum_{\uhw_{3:d}}
	\f{\I\bigg\{\begin{array}{c}
	\dw=\dWP(\hw\uhw_{3:d})\\
	\sw=\dWP(\hs\uhw_{3:d})
	\end{array}\bigg\}}
	{\exp\{ y\dph_{d-1}(\hs\uhw_{3:d})  \}}
	\prod_{i=3}^d \psi_{\hw_i}}
	{ \displaystyle
	\sum_{\hv_{2:d}}
	\f{\Ind{\dw=\dWP(\hv_{2:d})}}{
	\exp\{y\dph_{d-1}(\hv_{2:d}) \}
	} \prod_{i=2}^d \psi_{\hv_i}
	}
	=\f{ \psi_{\hw}}{\dz \rho_{\dw} }
	\sum_{\uhw_{3:d}}
	\I\bigg\{\begin{array}{c}
	\dw=\dWP(\hw\uhw_{3:d})\\
	\sw=\dWP(\hs\uhw_{3:d})
	\end{array}\bigg\}
	\prod_{i=3}^d \psi_{\hw_i}
	\equiv
	\f{ \psi_{\hw}}{\dz \rho_{\dw} }
	\dN_{\dw\ds,\hw\hs}\,,
	\]
where the last equality is the definition of a $9\times 9$ matrix $\dN$. The full stability matrix is $\bB=\dB\hB$, and we define
	\[
	\bB_{\dv\dr,\dw\ds}
	\equiv \f{\rho_{\dw} \bN_{\dv\dr,\dw\ds}}
		{\dz \rho_{\dv}}\,.
	\]
The pattern of non-zero entries in $\hB$ and $\bB$ is shown in Figure~\ref{f:B}. For instance, the last two rows of $\hB$ are identically zero because there is no choice of $\dw,\ds,\dw_{3:k}$ such that $\zro=\hWP(\dw\dw_{3:k})$ while $\one=\hWP(\ds\dw_{3:k})$. 

\begin{figure}[h!]
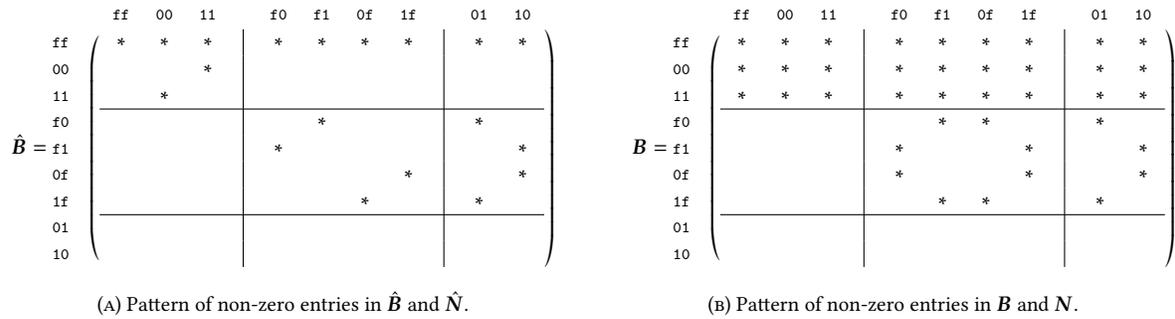

	\centering
    \begin{subfigure}[b]{0.47\textwidth}
    {\footnotesize\[\hB=
	\kbordermatrix{
	& \free\free & \zro\zro & \one\one
		& & \free\zro & \free\one & \zro\free
		& \one\free && \zro\one & \one \zro \\
	\free\free & * & * & * & \vrule & * & * & * & * 
		&\vrule& * & * \\
	\zro\zro   & & & * &\vrule&&&&&
		\vrule& &  \\
	\one\one   & & * &&\vrule&&&&&
		\vrule& &  \\
	\cline{2-12}
	\free\zro &&&&\vrule & & * &&&
		\vrule& * &  \\
	\free\one &&&&\vrule & * & &&&
		\vrule&  & * \\
	\zro\free &&&&\vrule&&&&*&\vrule& & * \\
	\one\free &&&&\vrule&&&*&&\vrule& *&  \\
	\cline{2-12}
	\zro\one  &&&&\vrule&&&&&\vrule \\
	\one\zro  &&&&\vrule&&&&&\vrule
	}\]}%
	\caption{Pattern of non-zero entries in $\hB$
		and $\hN$.}
	\label{f:hB}
    \end{subfigure}\quad
    \begin{subfigure}[b]{0.47\textwidth}
	{\footnotesize\[
	\quad
	\bB=\kbordermatrix{& \free\free & \zro\zro & \one\one
		& & \free\zro & \free\one & \zro\free
		& \one\free && \zro\one & \one \zro \\
\free\free &*&*&*&\vrule&*&*&*&*&\vrule&*&*\\
\zro\zro   &*&*&*&\vrule&*&*&*&*&\vrule&*&*\\
\one\one   &*&*&*&\vrule&*&*&*&*&\vrule&*&*\\
\cline{2-12}
\free\zro &&&&\vrule&&*&*&&\vrule& * &  \\
\free\one &&&&\vrule&*&&&*&\vrule&  & * \\
\zro\free &&&&\vrule&*&&&*&\vrule&  & * \\
\one\free &&&&\vrule&&*&*&&\vrule& * &  \\
\cline{2-12}
	\zro\one  &&&&\vrule&&&&&\vrule \\
	\one\zro  &&&&\vrule&&&&&\vrule 
	}
	\]}%
	\caption{Pattern of non-zero entries in $\bB$
		and $\bN$.}
		\label{f:bB}
	  \end{subfigure}
	\caption{The stability matrices $\hB$ and $\bB$.
	Only the top left $7\times 7$ submatrices will be used.}
	\label{f:B}
\end{figure}

\noindent
We will not use the last two columns of either matrix, so we will only evaluate the entries in the top left $7\times7$ submatrices, which we denote $\hB_{7\times7}$ and $\bB_{7\times7}$. Clearly, it is sufficient to evaluate $\hN_{7\times7}$ and 
$\bN_{7\times7}$. We have
	{\setlength{\jot}{0pt}\begin{align*}
	\hN_{\free\free,\free\free}
	&=1\,,\\
	\hN_{\free\free,\zro\zro}
	=\hN_{\free\free,\free\zro}
	=\hN_{\free\free,\zro\free}
	&=1-(\rho_\zro)^{k-2}\,,\\
	\hN_{\zro\zro,\one\one}
	=\hN_{\free\zro,\free\one}
	=\hN_{\zro\free,\one\free}
	&= (\rho_\zro)^{k-2}
	\end{align*}}%
All other entries of $\hN_{7\times7}$ (hence $\hB_{7\times7}$) are defined by the symmetry between $\zro$ and $\one$. To calculate $\bN_{7\times7}$, recall from \eqref{e:def.S.i} and \eqref{e:def.S.geq.i} the notation $S_i$ and $S_{\ge i}$. We will express the entries of $\bN$ in terms of these quantities. Note that
	\beq\label{e:dz.S}
	\dz
	= S_0 + 2\bigg(
		1 - (1-e^{-y})\psi_\zro
	\bigg)(S_1+S_{\ge2})
	\eeq
In the first three rows of $\bN_{7\times7}$ we have
	{\setlength{\jot}{0pt}\begin{align*}
	\bN_{\free\free,\free\free}
		&= S_0 \,,\\
	\bN_{\zro\zro,\free\free}
		=\bN_{\zro\zro,\free\one}
		=\bN_{\zro\zro,\one\free}
		&= S_1+S_{\ge2}\,,\\
	\bN_{\free\free,\free\zro}
		=\bN_{\free\free,\zro\free}
		&=(1-(\rho_\zro)^{k-2}) S_0\\
	\bN_{\zro\zro,\one\one}
		&= (\rho_\zro)^{k-2}S_0+ S_1+S_{\ge2}\,,\\
	\bN_{\zro\zro,\zro\free}
		&= (1-(\rho_\zro)^{k-2})S_1+S_{\ge2}\\
	\bN_{\free\free,\zro\zro}
		&=(1-(\rho_\zro)^{k-2}) S_0
		+(\rho_\zro)^{k-2} S_1 e^{-y}\,,\\
	\bN_{\zro\zro,\zro\zro}
		=\bN_{\zro\zro,\free\zro}
		&=(1-(\rho_\zro)^{k-2})(S_1+S_{\ge2})
		+ (\rho_\zro)^{k-2}S_{\ge2} e^{-y}\,,\end{align*}
All other entries in the first three rows of $\bN_{7\times7}$ are determined by the symmetry between $\zro$ and $\one$. 

\subsection{Gardner eigenvalue and auxiliary matrices} 
\label{ss:aux.matrices}

Now let $\bB_{4\times4}$ be the $4\times 4$ submatrix of $\bB$ given by row and column indices in $\set{\free\zro,\free\one,\zro\free,\one\free}$ (in the center of Figure~\ref{f:bB}): the corresponding entries of $\bN$ are given by
	\[
	\bN_{4\times4}
	=
	\kbordermatrix{
	&\free\zro&\free\one&\zro\free&\one\free\\
	\free\zro& & S_0 & S_1 & \\
	\free\one& S_0 & & & S_1\\
	\zro\free& \f{S_1}{e^y} & & & S_0\\
	\one\free& & \f{S_1}{e^y} & S_0
	} \cdot (\rho_0)^{k-2}\,.
	\]
From this it is easy to calculate that the largest eigenvalue of $\bB_{4\times4}$
(hence of the $6\times 6$ matrix $\bB_{\neq}$) is
	\beq\label{e:Gardner.lambda.explicit}
	\lambda 
	= \f{(\rho_\zro)^{k-2} 
	(S_0 + S_1 e^{-y/2})}{\dz}\,.
	\eeq
This is precisely the same $\lambda$ that appears in the statement of Proposition~\ref{p:gardner.defn}. Moreover, this $\lambda$ corresponds to a (right) eigenvector $\bxi\in\mathbb{R}^9$ of the full $9\times9$ matrix $\bB$, given explicitly by
	\beq\label{e:eigenvector}
	\bxi^\st =
	\kbordermatrix{
	&\free\free &\zro\zro &\one\one
	&\free\zro&\free\one&\zro\free&\one\free\\
	&-\f{2\rho_\zro}{e^{y/2}}
	&-\rho_\free&-\rho_\free
	& \rho_\zro & \rho_\zro  
	& \f{ \rho_\free}{e^{y/2}}
	& \f{ \rho_\free}{e^{y/2}}
	}\,.
	\eeq

\begin{proof}[Proof of Proposition~\ref{p:gardner.defn}]
Recalling \eqref{e:def.S.i}, we can rewrite $S_0$ 
in terms of $x,w$ as
	\[S_0
	=\sum_{\ell=0}^{d-2}
	\binom{d-2}{\ell}
	(w\cdot\GM)^\ell
	(1-w)^{d-2-\ell}
	\P\bigg(\Bin\bigg(\ell,\f12\bigg) =\f{\ell}{2}\bigg)\,.
	\]
Similarly we can $S_1,S_{\ge2}$ in terms of $x,w$ as
	\begin{align*}
	\f{S_1}{e^{y/2}}
	&=\sum_{\ell=0}^{d-2}
	\binom{d-2}{\ell}
	(w\cdot\GM)^\ell
	(1-w)^{d-2-\ell}
	\P\bigg(\Bin\bigg(\ell,\f12\bigg)
	=\f{\ell-1}{2}\bigg)\,,\\
	S_{\ge2}
	&=\sum_{\ell=0}^{d-2}
	\binom{d-2}{\ell}
	(w\cdot\AM)^\ell
	(1-w)^{d-2-\ell}
	\P\bigg(\Bin(\ell,p) \le \f{\ell}{2}-1
	\bigg)
	\end{align*}
The estimates of Lemma~\ref{l:binom.Z.estimates} apply equally well with $d-2$ in place of $d-1$, so
	\[
	S_0+ \f{S_1}{e^{y/2}}
	\asymp S_0
	\asymp \f{(1-w+w\cdot\GM)^{d-2}}{(\max\set{1,dw\cdot\GM})^{1/2}}\,,
	\]
and from \eqref{e:dz.S} we conclude
	\[
	\dz
	= S_0 + 2\bigg(
		1 - \f{(1-e^{-y})w}{2}
	\bigg)(S_1+S_{\ge2})
	\asymp S_0+S_1+S_{\ge2}
	\asymp(1-w+w\cdot\AM)^{d-2}\,.
	\]
Substituting into \eqref{e:Gardner.lambda.explicit}
gives
	\[
	\lambda\asymp
	\f1{2^k(\max\set{1,dw\cdot\GM})^{1/2}}
	\bigg(
	\f{1-w+w\cdot\GM}{1-w+w\cdot\AM}
	\bigg)^{d-2}
	\asymp \f{x}{2^k}\,,\]
therefore $\branch\lambda\asymp dk\lambda\asymp \CC k^2 x$. From Proposition~\ref{p:sp}~and~\ref{p:eone.defn} we have that $x$ is exponentially small with respect to $k$ for all $\asat\le\alpha\le 4^k$, so $\branch\lambda\ge1$ cannot occur before $\CC\ge e^{\Omega(k)}$. In this regime
the estimate \eqref{e:Gamma.star} 
(and the discussion leading to \eqref{e:sqrt.d})
implies $x\asymp 1/d^{1/2}$, therefore
	\[
	\branch\lambda
	\asymp \CC k^2 x
	\asymp \f{\CC k^2}{d^{1/2}}
	\asymp
	\bigg( \f{\CC k^3}{ 2^k}\bigg)^{1/2}\,.
	\]
This crosses one at $\CC \asymp 2^k/k^3$, corresponding to $\aG\asymp 4^k/k^3$ as claimed.
\end{proof}

\noindent
We now define some auxiliary matrices which will have a role in what follows. Let $\PROJ$ be the $9\times9$ symmetric matrix with entries $\PROJ_{\dv\dr,\dw\ds}=\Ind{\dv=\dw,\dr=\ds}$. Let $\bPi$, $\bXi$, $\bGa$ be the $9\times9$ matrices with entries
	\begin{align*}
	\bPi_{\hw\hs,\dw\ds}
	&= \psi_{\hw} \rho_{\dw}
		\exp\{ -y\eph(\ds\hs) \}\,,\\
	\bXi_{\hw\hs,\dw\ds}
	&= \psi_{\hw} \rho_{\dw}
	\exp\{ y\eph(\dw\hw)
		-y\eph(\dw\hs)-y\eph(\ds\hw)\}\,,\\
	\bGa_{\hw\hs,\dw\ds}
	&= \psi_{\hw} \rho_{\dw}
		\exp\{ -y\eph(\ds\hs) \}
		[\eph(\ds\hs)-\eph(\dw\hw)]\,.
	\end{align*}
Let $\bTe$ be the $9\times9$ matrix with entries
	\[
	\bTe_{\dv\dr,\dw\ds}
	\equiv \Ind{\dv=\dw}\sum_{\hw}
		\f{\psi_{\hw}\rho_{\dw}
		\exp(y\eph(\dw\hw))
		}{
		\exp(y\eph(\dr\hw)+y\eph(\ds\hw))
		}
	\equiv \Ind{\dv=\dw}\sum_{\hw}
	\f{\bPi_{\hw\hw,\dw\dr}\bPi_{\hw\hw,\hw\ds}}
		{\bPi_{\hw\hw,\dw\dw}}\,.
	\]
We remark for later use that $\bPi^\st\hB$, $\bXi^\st\hB$, and $\bGa^\st\hB$ are all symmetric matrices, as are $\bPi\dB$, $\bXi\dB$, and $\bGa\dB$. It also follows 
from the definitions that 
$\PROJ(\bPi-\bXi)$,
$(\bPi-\bXi)\PROJ$,
 and $\PROJ\bGa\PROJ$ are all identically zero. As a result, for any vector $\bde\in\mathbb{R}^9$ satisfying $\bde=\PROJ\bde$, we have $\PROJ\bGa\bde=\PROJ\bGa\PROJ\bde=0$,
and $\bPi\bde=\bXi\bde$. Since
$\bde=\PROJ\bde$ implies 
$\hB\bde=\PROJ\hB\bde$ and
$\bB\bde=\PROJ\bB\bde$,
two further consequences are that
$(\hB\bde,\bGa\bde)
=(\hB\bde)^\st \PROJ \bGa\PROJ\bde=0$
and similarly
$(\hB\bde,\bGa\bB\bde)
=(\hB\bde)^\st\PROJ\bGa\PROJ\bB\bde=0$.
We record also that $\bPi\bxi$ is identically zero,
while
	{\setlength{\jot}{0pt}\begin{align}
	(\hB\bxi,\bGa\bxi)
	&=(1-e^{-y/2})e^{-y/2}(1-x)x^2w\,,\\
	(\hB\bxi,\Xi\bxi)
	&=(1-e^{-y/2})
		(1-e^{-y})(1-x)x^2w
	\label{e:scalar.product.sign}
	\end{align}}%
We will use all these observations in what follows
to complete the proof of our main result Theorem~\ref{t:main}.

\subsection{Perturbation around type II degeneracy}
\label{ss:pert.ii}

Recall the discussion of \S\ref{ss:intro.interp}. As suggested by the physics literature \cite{MRT,KPW}, we evaluate
the zero-temperature 2RSB functional
\eqref{e:Phi.tworsb} on a slight perturbation of case \ref{II}, as follows. Let $y_1=y$ and take $y_2>y_1$ such that $\NU\equiv y_1/y_2$ is close to $1$, or equivalently that $\ZETA\equiv 1-\NU$ is small.
Suppose
	\beq\label{e:Q.delta}
	Q = \sum_{\dw} \rho_{\dw} (1+\delta_{\dw})
		Q_{\dw}\,.\eeq
where each $Q_{\dw}$ is a probability measure on $\Omega$ whose support is contained in a small neighborhood of $\I_{\dw}$. This means that if $\rho$ is sampled from $Q_{\dw}$, then $f\equiv f_\rho\equiv \rho-\I_{\dw}$ is a signed measure all of whose weights are small. Let
	\beq\label{e:def.epsilon}
	\ep_{\dw\ds}
	\equiv\int
	\bigg( \rho(\ds)-\Ind{\dw=\ds}\bigg)
	\,dQ_{\dw}(\rho)
	= \int f_\rho(\ds)\,dQ_{\dw}(\rho)
	\,;\eeq
this quantity captures the ``average mass sent from $\dw$ to $\ds$.'' Finally, for any $\dw,\dr,\ds$ define the scalar product
	\beq\label{e:def.upsilon}
	\Upsilon_{\dw\dr,\dw\ds}
	\equiv
	\int
	\bigg( \rho(\dr)-\Ind{\dw=\ds}\bigg)
	\bigg( \rho(\ds)-\Ind{\dw=\ds}\bigg)
	\,dQ_{\dw}(\rho)
	= \int f_\rho(\dr)f_\rho(\ds)\,dQ_{\dw}(\rho)
	\,.\eeq
In order for $Q,Q_\zro,Q_\one,Q_\free$ to all be valid probability measures on $\Omega$, we must have
	\beq\label{e:lin.constraints}
	\sum_{\dw} \rho_{\dw} \delta_{\dw}=0\,,\quad
	\sum_{\ds} \ep_{\dw\ds} = 0
	\textup{ for all }\dw\,.
	\eeq
We will write $\bep$ for the vector in $\mathbb{R}^9$ with entries $\ep_{\dw\ds}$. It will be convenient also to define vectors $\bde,\bpi\in\mathbb{R}^9$ where
	{\setlength{\jot}{0pt}\begin{align*}
	\bde_{\dw\ds}&=\Ind{\dw=\ds}\delta_{\dw}\,,\\
	\bpi_{\dw\ds}&=\Ind{\dw=\ds}\delta_{\dw}\ep_{\dw\dw}\,.
	\end{align*}}%
Note that $\bde$ can have at most three nonzero entries, and the same holds for $\bpi$. Let $\bup$ be the $9\times9$ matrix with entries
$\bup_{\dv\dr,\dw\ds}=\Ind{\dv=\dw}\Upsilon_{\dw\dr,\dw\ds}$.
For our purposes, the vectors $\bde$, $\bep$, and $\bup$ encode the key summary statistics of $Q$. We will assume that all entries of $\bde$ and $\bep$ are $O(\ZETA^2)$, while all entries of $\bup$ are $O(\ZETA^4)$. Let $Q_\II$ be    as in \eqref{e:Q.II}, corresponding to
$\bde$, $\bep$, $\bup$ all zero.

\begin{ppn}\label{p:two.perturb}
Suppose $y_1=y<y_2$ such that $\NU\equiv y_1/y_2$ is close to $1$. Let $\ZETA\equiv1-\NU$, and take $Q$ as in \eqref{e:Q.delta} such that for all $\dw$ we have
$\|\rho-\I_{\dw}\|_\infty =O(\ZETA^2)$
uniformly over all $\rho\in\supp Q_{\dw}$. If $Q$ has summary statistics $\bde,\bep,\bup$, then
	\begin{align*}
	\PhiTwo(y_1,y_2,Q)
	&=\PhiTwo(y,y,Q_\II)
	+
	\f{d(k-1)}{2}
	\bigg(\hB\bta,
		\bigg(\f{(\bPi-y\ZETA\bGa)}{\nu}
		-\ZETA\bXi\bigg)
		(\branch\bB-\bm{I})\bta
		\bigg)\\
	&\qquad-\f{d(k-1)(dk-d-k)}2
	\bigg(\I_9,\PROJ
		(\bPi-y\NU^{-1}\ZETA\bGa)\bta
	\bigg)^2
	+O(\ZETA^6)
	\end{align*}
where $\bta\equiv\bde+\NU(\bep+\bpi)\in\mathbb{R}^9$.
\end{ppn}

\subsection{Perturbed clause functional} In this subsection we analyze the clause 2RSB functional $\FG(y_1,y_2,Q)$ for $Q$ near $Q_\II$, and show how the clause stability matrix $\hB$ arises.

\begin{lem}\label{l:two.clause.perturb}
In the setting of Proposition~\ref{p:two.perturb},
	\begin{align*}
	\FG(y_1,y_2,Q)
	&=\FG(y,y,Q_\II)+
	k\bigg(\I_9,\PROJ
		(\bPi-y\NU^{-1}\ZETA\bGa)\bta
	\bigg)
	-\f{k\NU\ZETA}{2}\FG_3(\I_9, (\bup\odot\bTe)\I_9)\\
	&\qquad+\binom{k}{2}\bigg\{\f{(\hB\bta,
		(\bPi-y\ZETA\bGa)\bta)}{\nu}
	-\ZETA(\hB\bta,\bXi\bta)
	\bigg\}
	\end{align*}
where $\odot$ denotes the Hadamard (entrywise) matrix product.

\begin{proof}
Recall the definition \eqref{e:tworsb.clause}. We abbreviate $\FG\equiv\FG(y_1,y_2,Q)$, $\FG_\II\equiv\FG(y,y,Q_\II)$, and $\Delta\FG\equiv \FG-\FG_\II$. We also write
$\rho(\dw_{1:k})$ as shorthand for the $k$-fold product $\rho_{\dw_1}\cdots\rho_{\dw_k}$. Expanding according to the definition 
\eqref{e:Q.delta} gives
	\[
	\FG=\sum_{\dw_{1:k}}
	\rho(\dw_{1:k})
	\prod_{j=1}^k
	(1+\delta_{\dw_j})
	\int\bigg\{
	\sum_{\ds_{1:k}}
	\exp(-y_2\hph(\ds_{1:k}))
	\prod_{j=1}^k\rho_j(\ds_j)
	\bigg\}^\NU
	\prod_{j=1}^k dQ_{\dw_j}(\rho_j)\,.\]
For $\rho_j$ sampled from $Q_{\dw_j}$ let $f_j\equiv\rho_j-\I_{\dw_j}$. Notice that the inner sum above, over configurations $\ds_{1:k}$, is dominated by the contribution from the case $\ds_{1:k}=\dw_{1:k}$. We can expand it to second order (with respect to the $f_j$) as a sum of three terms $I_0,I_1,I_2$: the contribution from 
$\ds_{1:k}=\dw_{1:k}$ is
	\begin{align*}
	I_0 &= \f1{\exp(y_2\hph(\dw_{1:k}))}
	\prod_{j=1}^k(1+f_j(\dw_j))\\
	&= \f1{\exp(y_2\hph(\dw_{1:k}))}
	\bigg\{
	1 + \sum_{j=1}^k f_j(\dw_j)
		+\sum_{1\le i<j\le k}
		f_i(\dw_i)f_j(\dw_j)
	\bigg\} + O(\ZETA^6)\,.\end{align*}
The contribution from configurations $\ds_{1:k}$ that differ from $\dw_{1:k}$ in a single coordinate is
	\[I_1
	=\sum_{i=1}^k
	\sum_{\ds_i}
	\f{\Ind{\ds_i\ne\dw_i}}
		{\exp(y_2\hph(\ds_i\dw_{[k]\setminus i}))}
	f_i(\ds_i)
	\bigg\{
	1+\sum_{j\in[k]\setminus i} f_j(\dw_j)\bigg\}
	+O(\ZETA^6)\,,\]
where $\dw_{[k]\setminus i}$ refers to $\dw_{1:k}$ with the $i$-th entry dropped.
The contribution from configurations $\ds_{1:k}$ that differ from $\dw_{1:k}$ in two coordinates is
	\[I_2
	=\sum_{1\le i<j\le k}
	\sum_{\ds_i,\ds_j}
	\f{\Ind{\ds_i\ne\dw_i,\ds_j\ne\dw_j}}
		{\exp(y_2\hph(\ds_i\ds_j
			\dw_{[k]\setminus\set{i,j}}))}
	f_i(\ds_i) f_j(\ds_j) + O(\ZETA^6)\,,\]
where $\dw_{[k]\setminus\set{i,j}}$ refers to $\dw_{1:k}$ with the $i$-th and $j$-th entries dropped.
It is convenient for us to rearrange the terms and express $I_0+I_1+I_2=\exp(-y_2\vph(\dw_{1:k}))(1+J_1+J_2)$ where
	\begin{align*}
	J_1&\equiv
	\sum_{j=1}^k
	\sum_{\ds_j}
	\f{\exp(-y_2\hph(\ds_j\dw_{[k]\setminus j} ))}
		{\exp(-y_2\hph(\dw_{1:k}))}
	f_j(\ds_j) = O(\ZETA^2)\,,\\
	J_2
	&\equiv \sum_{1\le i<j\le k}
	\sum_{\ds_i,\ds_j}
	\f{\exp(-y_2\hph(\ds_i\ds_j
		\dw_{[k]\setminus\set{i,j}}))}
		{\exp(-y_2\hph(\dw_{1:k}))}
	f_i(\ds_i) f_j(\ds_j) = O(\ZETA^4)\,.
	\end{align*}
Since $(1+t)^\NU = 1+\NU t-\NU(1-\NU)t^2/2 + O(t^3)$ for small $t$, we can expand
	\[
	(1+J_1+J_2)^\NU
	= 1 + \NU J_1 + \NU J_2 
		- \f{\NU(1-\NU)}{2} (J_1)^2
		+O(\ZETA^6)\,.
	\]
Since $J_1$ involves a sum over $1\le i\le k$, its square is a double sum, and we can further decompose
$(J_1)^2=J_3+J_4$ where $J_3$ captures the diagonal terms while $J_4$ captures the off-diagonal terms.
Substituting this expansion back into the definition of $\FG$ results in the decomposition
	\beq\label{e:clause.expand.inside}
	\FG
	=\FG_0+ k\NU \FG_1+ \binom{k}{2} \NU\FG_2
		-k\f{\NU(1-\NU)}{2} \FG_3 
		-\binom{k}{2}\NU(1-\NU)  \FG_4
	+O(\ZETA^6)\,.
	\eeq
We now proceed to evaluate the $\FG_i$,
beginning with $\FG_0$ which is the value when $\bep$ and $\bup$ are zero:
	\begin{align*}
	\FG_0
	&=\sum_{\dw_{1:k}}
	\f{\rho(\dw_{1:k})}{\exp(y\hph(\dw_{1:k}))}
	\prod_{j=1}^k(1+\delta_{\dw_j})
	=\sum_{\dw_{1:k}}
	\f{\rho(\dw_{1:k})}{\exp(y\hph(\dw_{1:k}))}
	\bigg\{1 + k \delta_{\dw_1}
		+\binom{k}{2}\delta_{\dw_1}\delta_{\dw_2}
		\bigg\} 
		+O(\ZETA^6)\\
	&=\FG_\II
	+k(\I_9,\PROJ\bPi\bde)
	+\binom{k}{2}(\hB\bde,\bPi\bde)
		+O(\ZETA^6)\,.
	\end{align*}
For future use we denote the two scalar products appearing in the last expression above as
$\FG_{0,1}$ and $\FG_{0,2}$, so that
	\beq\label{e:G.zero.expansion}
	\FG_0
	=\FG_\II+k\FG_{0,1}
	+\binom{k}{2}\FG_{0,2}\,.\eeq
By symmetry among the coordinates $1\le j\le k$, the average of $J_1$ is $k\FG_1$ where
	\begin{align*}\nonumber
	\FG_1
	&=\sum_{\dw_{1:k}}
	\rho(\dw_{1:k})
	\prod_{j=1}^k(1+\delta_{\dw_j})
	\sum_{\ds_1}
	\f{\ep_{\dw_1\ds_1}
	\exp\{-(y_2-y)
		(\hph(\ds_1\dw_{2:k})
		-\hph(\dw_{1:k}))
		\}}
	{\exp\{y\hph(\ds_1\dw_{2:k})\}}
	\\\nonumber
	&=\sum_{\dw_{1:k}}
	\rho(\dw_{1:k})
	\bigg\{1 + \delta_{\dw_1}
		+(k-1)\delta_{\dw_2}\bigg\}
	\sum_{\ds_1}
	\f{\ep_{\dw_1\ds_1}
	\{
	1 - y_2(1-\NU)(\hph(\ds_1\dw_{2:k})
		-\hph(\dw_{1:k}))\}}
	{\exp\{y\hph(\ds_1\dw_{2:k})\}}+O(\ZETA^6)\\
	&=\bigg(\I_9,\PROJ
		(\bPi-y\NU^{-1}\ZETA\bGa)(\bep+\bpi)
	\bigg)
	+(k-1)\bigg(\hB\bde,
		(\bPi-y\ZETA\bGa)\bep\bigg)
	+O(\ZETA^6)\,.
	\end{align*}
For future use we denote the two scalar products appearing in the last expression above
as $\FG_{1,1}$ and $\FG_{1,2}$, so that
	\beq\label{e:G.one.expansion}
	k\FG_1
	=k\FG_{1,1}+2\binom{k}{2}\FG_{1,2}+O(\ZETA^6)\,.
	\eeq
The average of $J_2$ is $\binom{k}{2}\FG_2$ where
	\begin{align*}
	\FG_2
	&=\sum_{\dw_{1:k}}
	\f{\rho(\dw_{1:k})}
		{\exp(y\hph(\ds_1\ds_2\dw_{3:k}))}
	\sum_{\ds_1,\ds_2}
	\f{\ep_{\dw_1\ds_1}\ep_{\dw_2\ds_2}}
		{\exp\{(y_2-y)(\hph(\ds_1\ds_2\dw_{3:k})
		-\hph(\dw_{1:k}))\}}
		+O(\ZETA^6)\\
	&=\bigg(\hB\bep,
		(\bPi-y\ZETA\bGa)\bep\bigg)+O(\ZETA^6)\,.
	\end{align*}
The average of $J_3$ is $k\FG_3$ where 
	\begin{align*}
	\FG_3
	&=\sum_{\dw_{1:k}}
	\sum_{\dr_1,\ds_1}
	\f{\rho(\dw_{1:k})\exp(y\hph(\dw_{1:k}))
		\Upsilon_{\dw_1\dr_1,\dw_1\ds_1}}
		{\exp(y\hph(\dr_1\dw_{2:k})
			+ y\hph(\ds_1\dw_{2:k}))}
	+O(\ZETA^5)\\
	&=\sum_{\dw,\dr,\ds}
	\Upsilon_{\dw\dr,\dw\ds}
	\sum_{\hw}
	\f{\rho_{\dw}\psi_{\hw}
		\exp(y\eph(\dw\hw))
		}
		{\exp(y\eph(\dr\hw)
			+ y\eph(\ds\hw))}+O(\ZETA^5)
	= (\I_9, (\bup\odot\bTe)\I_9)
	+ O(\ZETA^5)
	\end{align*}
The average of $J_4$ is $\binom{k}{2}\FG_4$ where 
	\[\FG_4
	=\sum_{\dw_{1:k}}
	\sum_{\ds_1,\ds_2}
	\f{\rho(\dw_{1:k})\exp(y\hph(\dw_{1:k}))
		\ep_{\dw_1\ds_1}\ep_{\dw_2\ds_2}}
		{\exp(y\hph(\ds_1\dw_{2:k}))
		\exp(y\hph(\dw_1\ds_2\dw_{3:k}))}
		+O(\ZETA^5)
	=(\hB\bep,\bXi\bep)+ O(\ZETA^5)\,.\]
Collecting terms gives
	\beq\label{e:collect.Delta.FG}
	\Delta\FG
	= k\bigg\{
		\FG_{0,1}+\NU\FG_{1,1}-\f{\NU\ZETA}{2}\FG_3
	\bigg\}
	+\binom{k}{2}
	\bigg\{
	\FG_{0,2}+2\NU\FG_{1,2}+\NU\FG_2-\NU\ZETA\FG_4
	\bigg\}
	\equiv k \Delta_1 + \binom{k}{2} \Delta_2
	\,.\eeq
Recall that $\bPi^\st\hB$, $\bXi^\st\hB$, and $\bGa^\st\hB$ are all symmetric matrices. We have
	\[\Delta_1
	=\bigg(\I_9,\PROJ
		(\bPi-y\NU^{-1}\ZETA\bGa)(
		\bde+\NU(
		\bep+\bpi))
	\bigg)
	-\f{\NU\ZETA}{2}\FG_3(\I_9, (\bup\odot\bTe)\I_9)
	+O(\ZETA^6)\,,
	\]
having used that $\PROJ\bGa\bde$ is identically zero. We also have
	\[\Delta_2
	=(\hB\bde,(\bPi-y\ZETA\bGa)\bde)
	+2\NU\bigg(\hB\bde,
		(\bPi-y\ZETA\bGa)\bep\bigg)
	+\NU\bigg(\hB\bep,
		(\bPi-y\ZETA\bGa)\bep\bigg)
		-\ZETA(\hB\bep,\bXi\bep)+O(\ZETA^6)\,.\]
Recall from the discussion at the end of \S\ref{ss:aux.matrices} that $(\hB\bde,\bGa\bde)=0$, so we can freely add any multiple of $(\hB\bde,\bGa\bde)$ to the above. Recall also that $\bPi\bde=\bXi\bde$, so we can also freely interchange $(\bPi\bde,\hB\bep)$ with $(\bXi\bde,\hB\bep)$. Using these identities, and absorbing some errors into the $O(\ZETA^6)$ term, we can ``complete the square'' and obtain
	\beq\label{e:Delta.two.completed.square}
	\Delta_2
	=\f{(\hB(\bde+\bep),
		(\bPi-y\ZETA\bGa)(\bde+\bep))}{\nu}
	-\ZETA(\hB(\bde+\bep),\bXi(\bde+\bep))
	+O(\ZETA^6)\,.
	\eeq
The claimed result follows.
\end{proof}
\end{lem}

\subsection{Perturbed variable functional}
In this subsection we analyze the variable 2RSB functional $\WW(y_1,y_2,Q)$ for $Q$ near $Q_\II$, and show how the full stability matrix $\bB$ arises.

\begin{lem}\label{l:two.var.perturb}
In the setting of Proposition~\ref{p:two.perturb},
	\begin{align*}
	\f{\WW(y_1,y_2,Q)}{\dz}
	&=\f{\WW(y,y,Q_\II)}{\dz}+
	 d(k-1) \bigg(\I_9,\PROJ
		(\bPi-y\NU^{-1}\ZETA\bGa)\bta
	\bigg)
	-\f{d(k-1)
	\NU\ZETA}{2}\FG_3(\I_9, (\bup\odot\bTe)\I_9)\\
	&\qquad+d\binom{k-1}{2}\bigg\{\f{(\hB\bta,
		(\bPi-y\ZETA\bGa)\bta)}{\nu}
	-\ZETA(\hB\bta,\bXi\bta)
	\bigg\}\\
	&\qquad+\binom{d}{2}(k-1)^2
	\bigg\{\f{(\hB\bta,
		(\bPi-y\ZETA\bGa)\bB\bta)}{\nu}
	-\ZETA(\hB\bta,\bXi\bB\bta)
	\bigg\}
	+O(\ZETA^6)\,,
	\end{align*}
where $\dz$ 
is the normalizing constant in the $\dSP_y$ recursion \eqref{e:sp.var}.

\begin{proof}
This is similar to the proof of Lemma~\ref{l:two.clause.perturb}, although slightly more involved because the input warnings are propagated two layers to the root. We recall the definition \eqref{e:tworsb.var}, and continue to write $D\equiv d(k-1)$. We abbreviate $\WW\equiv\WW(y_1,y_2,Q)$, $\WW_\II\equiv\WW(y,y,Q_\II)$,
and $\Delta\WW\equiv\WW-\WW_\II$. Then
	\[
	\WW=\sum_{\dw_{1:D}}\rho(\dw_{1:D})
		\prod_{j=1}^D(1+\delta_{\dw_j})
	\int\bigg\{
	\sum_{\ds_{1:D}}
	\exp(-y_2\vph(\ds_{1:D}))
	\prod_{j=1}^D\rho_j(\ds_j)
	\bigg\}^\NU\prod_{j=1}^D dQ_{\dw_j}(\rho_j)\,.
	\]
As in the proof of Lemma~\ref{l:two.clause.perturb}, the inner sum over $\ds_{1:D}$ is dominated by the
$\ds_{1:D}=\dw_{1:D}$ term. It can be expanded similarly as $\exp(-y_2\vph(\dw_{1:D}))(1+ J_1 + J_2
		+ J_{2,\textup{vx}})+ O(\ZETA^6)$ where the first-order correction is
	\[J_1=
	\sum_{j=1}^D
	\sum_{\ds_j}
	\f{\exp(-y_2\vph(\ds_j
		\dw_{[D]\setminus j}))}
		{\exp(-y_2\vph(\dw_{1:D}))}
	f_j(\ds_j)
	=O(\ZETA^2)\,,\]
and we now split the second-order correction into two components:
	\begin{align*}
	J_2
	&=\sum_{a=1}^d
	\sum_{2\le i<j\le k}
	\f{\exp(-y_2\vph(
			\ds_{ai}\ds_{aj}
			\dw_{[D]\setminus\set{ai,aj}}
			)}
		{\exp(-y_2\vph(\dw_{1:D}))}
	f_{ai}(\ds_{ai})
	f_{aj}(\ds_{aj})= O(\ZETA^4)\,,\\
	J_{2,\textup{vx}}
	&=\sum_{1\le a<b\le d}
	\sum_{i=2}^k\sum_{j=2}^k
	\f{\exp(-y_2\vph(
			\ds_{ai}\ds_{bj}
			\dw_{[D]\setminus\set{ai,bj}}
			)}
		{\exp(-y_2\vph(\dw_{1:D}))}
	f_{ai}(\ds_{ai})
	f_{bj}(\ds_{bj}) = O(\ZETA^4)\,.
	\end{align*}
We further decompose $(J_1)^2= J_3 + J_4 + J_{4,\textup{vx}}$ where $J_3$ is the contribution to the double sum from diagonal terms $(ai,ai)$, $J_4$ is the contribution from off-diagonal terms $(ai,aj)$ with $i\ne j$, and $J_{4,\textup{vx}}$ is the contribution from off-diagonal terms $(ai,bj)$ with $a\ne b$. Altogether it gives
	\[
	(1+ J_1 + J_2
		+ J_{2,\textup{vx}})^\NU
	=1+\NU J_1 + \NU J_2 + \NU J_{2,\textup{vx}}
		-\f{\NU\ZETA}{2}J_3
		-\f{\NU\ZETA}{2}J_4
		-\f{\NU\ZETA}{2}J_{4,\textup{vx}}
		+O(\ZETA^6)\,.
	\]
Then, similarly to \eqref{e:clause.expand.inside}, we have a corresponding expansion
	{\setlength{\jot}{2pt}\begin{align*}
	\WW
	&=\WW_0+d(k-1)\NU\WW_1
		+d\binom{k-1}{2}\NU\WW_2
		+\binom{d}{2}(k-1)^2\NU\WW_{2,\textup{vx}}\\
	&\qquad-d(k-1)\f{\NU\ZETA}{2}\WW_3
		-d\binom{k-1}{2} \NU\ZETA\WW_4
		-\binom{d}{2}(k-1)^2
		\NU\ZETA\WW_{4,\textup{vx}}+ O(\ZETA^6)\,.
	\end{align*}}%
Recalling
\eqref{e:G.zero.expansion} and 
\eqref{e:G.one.expansion}, we have similarly
	\begin{align*}
	\WW_0&=\WW_\II
		+d(k-1)\WW_{0,1}
		+d\binom{k-1}{2}\WW_{0,2}
		+\binom{d}{2}(k-1)^2
		\WW_{0,2,\textup{vx}}\,,\\
	d(k-1)\WW_1
	&=d(k-1)\WW_{1,1}
	+d\binom{k-1}{2}2\WW_{1,2}
	+2\binom{d}{2}(k-1)^2
	\WW_{1,2,\textup{vx}}
	\end{align*}
Recalling that $\dz$ is the normalizing constant in the $\dSP_y$ recursion \eqref{e:sp.var}, we have
	\[
	\dz
	=\f{\WW_{0,1}}{\FG_{0,1}}
	=\f{\WW_{0,2}}{\FG_{0,2}}
	=\f{\WW_{1,1}}{\FG_{1,1}}
	=\f{\WW_{1,2}}{\FG_{1,2}}
	=\f{\WW_2}{\FG_2}
	=\f{\WW_3}{\FG_3}
	=\f{\WW_4}{\FG_4}\,,
	\]
so it remains to calculate 
$\WW_{0,2,\textup{vx}}$, $\WW_{1,2,\textup{vx}}$, 
$\WW_{2,\textup{vx}}$, and
$\WW_{4,\textup{vx}}$. Recalling that $\bB\equiv\dB\hB$, we have
	\begin{align*}
	\WW_{0,2,\textup{vx}}/\dz
		&=(\hB\bde,\bPi\bB\bde)
			+O(\ZETA^6)\,,\\
	\WW_{1,2,\textup{vx}}/\dz
		&=(\hB\bde,(\bPi-y\ZETA\bGa)\bB\bep)
			+O(\ZETA^6)\,,\\
	\WW_{2,\textup{vx}}/\dz
	&=(\hB\bep,(\bPi-y\ZETA\bGa)\bB\bep)
			+O(\ZETA^6)\,,\\
	\WW_{4,\textup{vx}}/\dz
	&= (\hB\bep,\bXi\dB\hB\bep)+O(\ZETA^5)\,.
	\end{align*}
Collecting terms gives
	\begin{align*}
	\f{\Delta\WW}{\dz}
	= d(k-1) \Delta_1
	+d\binom{k-1}{2}\Delta_2
	+\binom{d}{2}(k-1)^2
	\bigg\{
	\WW_{0,2,\textup{vx}}
	+2\NU \WW_{1,2,\textup{vx}}
	+\NU\WW_{2,\textup{vx}}
	-\NU\ZETA\WW_{4,\textup{vx}}
	\bigg\}
	\end{align*}
for $\Delta_1,\Delta_2$ as defined by \eqref{e:collect.Delta.FG}. We denote the last coefficient above (in braces) as $\Delta_{2,\textup{vx}}$. Similarly to \eqref{e:Delta.two.completed.square} we have
	\[
	\Delta_{2,\textup{vx}}
	=\f{(\hB(\bde+\bep),
		(\bPi-y\ZETA\bGa)
		\bB(\bde+\bep))}{\nu}
	-\ZETA(\hB(\bde+\bep),
		\bXi\bB(\bde+\bep))
	+O(\ZETA^6)\,.
	\]
The result follows.
\end{proof}
\end{lem}

\subsection{Proof of main theorem}
\label{ss:proof.main}

We now prove Proposition~\ref{p:two.perturb}
and deduce our main result Theorem~\ref{t:main}.

\begin{proof}[Proof of Proposition~\ref{p:two.perturb}]
We abbreviate $\Delta\Phi\equiv\PhiTwo(y_1,y_2,Q)
	-\PhiTwo(y,y,Q_\II)$. 
By substituting the estimates of Lemmas~\ref{l:two.clause.perturb}~and~\ref{l:two.var.perturb} into \eqref{e:Phi.tworsb}, and recalling that $\log(1+x) = x -x^2/2 + O(x^3)$ for small $x$, we find
	\[
	\Delta\Phi
	=\f{d(k-1)}{2}(\branch
	\Delta_{2,\textup{vx}}-\Delta_2)
	-\f{d(k-1)(dk-d-k)}2(\Delta_1)^2
	+O(\ZETA^6)\,.
	\]
The result follows.
\end{proof}

\begin{proof}[Proof of Theorem~\ref{t:main}]
For $\asat\le\alpha\le 4^k/k$, and $y$ satisfying \eqref{e:y.restriction}, let $\dq_y$ be defined by Proposition~\ref{p:sp}, and let $Q_{\II,y}$ be defined by \eqref{e:Q.II}. Recall from \eqref{e:inf.Fy.over.y} that
	\[\f{\FF(y)}{y}
	=\PhiOne(y,\dq_y)
	=\PhiTwo(y,y,Q_{\II,y})\,.\]
Proposition~\ref{p:eone.defn} gives the well-defined formula
	\[\eone(\alpha)\equiv\ee(y_\star(\alpha))
	=- \inf \bigg\{ \f{\FF(y)}{y}
	:y\ge0\,,
	y \textup{ satisfies }\eqref{e:y.restriction}
	\bigg\}\,.
	\]
It follows from Proposition~\ref{p:interpolation} that
$\einf \ge \eone(\alpha)$ --- i.e., that
\eqref{e:diff} is nonnegative --- for all 
$\asat\le\alpha\le 4^k/k$. It remains to verify that 
\eqref{e:diff} is strictly positive above $\aG$. For this,
 recall from \S\ref{ss:aux.matrices} that the $9\times9$ stability matrix $\bB\equiv\dB\hB$ has eigenvalue $\lambda$ given explicitly by \eqref{e:Gardner.lambda.explicit}, which is the same as the $\lambda$ that appears in Proposition~\ref{p:gardner.defn}. Associated to this $\lambda$ is a right eigenvector $\bxi\in\mathbb{R}^9$ of $\bB$, given explicitly by \eqref{e:eigenvector}. We now note that this vector can be split as
$\bxi=\bvp+\bsi$ where
	\begin{align*}
	\begin{pmatrix}
	\bvp^\st\\
	\bsi^\st
	\end{pmatrix}
	&=\kbordermatrix{
	&\free\free &\zro\zro &\one\one
	&\free\zro&\free\one&\zro\free&\one\free\\
	& 2(1-e^{-y/2})\rho_\zro
	&-(1-e^{-y/2})\rho_\free
	&-(1-e^{-y/2})\rho_\free
	&0&0&0&0\\
	&-2\rho_\zro 
	& -\f{\rho_\free}{e^{y/2}}
	& -\f{\rho_\free}{e^{y/2}}
	&\rho_\zro & \rho_\zro
	& \f{\rho_\free}{e^{y/2}}
	& \f{\rho_\free}{e^{y/2}}
	}\,.
	\end{align*}
These vectors satisfy (cf.\ \eqref{e:lin.constraints}) the constraints
	\[
	\sum_{\dw}\rho_{\dw}
		\varpi_{\dw}=0\,,\quad
	\sum_{\ds}\sigma_{\dw\ds}=0
		\textup{ for all }\dw\,.
	\]
We apply Proposition~\ref{p:two.perturb}
with $\bde=\ZETA^2\bvp$ and $\bep=\ZETA^2\NU^{-1}\bsi$, so that $\bta=\bde+\NU(\bep+\bpi)
=\ZETA^2\bxi + O(\ZETA^4)$: this gives
	\begin{align*}
	\f{\Delta\Phi}{d(k-1)/2}
	&=
	(\branch\lambda-1)\ZETA^4
	\bigg(\hB\bta,
		\bigg(\f{(\bPi-y\ZETA\bGa)}{\nu}
		-\ZETA\bXi\bigg)\bxi
		\bigg)\\
	&\qquad-(dk-d-k)
	\ZETA^4
	\bigg(\I_9,\PROJ
		(\bPi-y\NU^{-1}\ZETA\bGa)\bxi
	\bigg)^2
	+O(\ZETA^6)
	\end{align*}
where, as before, we abbreviate $\Delta\Phi\equiv\PhiTwo(y_1,y_2,Q)-\PhiTwo(y,y,Q_\II)$. It follows from our earlier calculation \eqref{e:scalar.product.sign} that $\Delta\Phi$ is negative whenever $\branch\lambda>1$. The result follows by applying Proposition~\ref{p:interpolation}.
\end{proof}

\section{Interpolation, stationarity, and convexity}\label{s:interpolation}

In this final section we prove a few auxiliary results which were used in the proof of the main theorem. In order to keep our presentation somewhat self-contained, in \S\ref{ss:rpc} we review the Ruelle probability cascade weights; and in \S\ref{ss:as2} we give the heuristic derivation of the (positive-temperature) 2RSB functional in the setting of random regular \textsc{nae-sat}. In \S\ref{ss:interp} we review a general interpolation bound proved in prior work, and use it to deduce Proposition~\ref{p:interpolation}.
In \S\ref{ss:stationarity} we derive stationary equations to prove Lemma~\ref{l:stationarity}. Finally, in \S\ref{ss:convexity} we prove Proposition~\ref{p:second.derivative} on convexity of the function $\FF(y)$.

\subsection{Ruelle probability cascades}
\label{ss:rpc}

For $m\in(0,1)$ we shall write $\Pi\sim \PP(m)$ to mean that $\Pi$ is a Poisson point process on $(0,\infty)$ with intensity measure
	\[
	\f{m\,dx}{x^{m+1}}\,.
	\]
The key property of $\PP(m)$ is the following scaling relation:

\begin{lem}[{\cite[Thm.~2.6]{MR3052333}}]
\label{l:poissondirichlet}
Let $(\Omega,\mathscr{F},\P)$ be a probability space,
and $(X,Y):\Omega\to(0,\infty)\times S$ a pair of random elements on $\Omega$, where $(S,\mathscr{S})$ is a measurable space. Suppose $\E(X^m)<\infty$ and let $\nu_m$ be the measure on $S$ defined by
	\[
	\nu_\zeta(B)
	= \f{ \E(X^m \Ind{Y\in B}) }{\E(X^m)}\,.\]
Suppose $\Pi\sim\PP(m)$, and let $(u_n)_{n\ge1}$ denote the points of $\Pi$ arranged in decreasing order. Let $(X_n,Y_n)_{n\ge1}$ be an i.i.d.\ sequence of copies of $(X,Y)$, independent from $\Pi$. Then
$(u_n X_n, Y_n)_{n\ge1}$ is again a Poisson process,
and has the same intensity measure as
the process
$( \E(X^m)^{1/m} u_n, \bar{Y}_n)_{n\ge1}$
where $(\bar{Y}_n)_{n\ge1}$ is a sequence of i.i.d.\
 samples from $\nu_m$
that are also independent from $\Pi$.
\end{lem}

The following discussion generalizes trivially to any finite number of levels of replica symmetry breaking, but for concreteness we consider only two levels. Fix 2RSB parameters $0<m_1<m_2<1$.
 Let $\Pi\sim\textsf{P}(m_1)$, and let $(u_s)_{s\ge1}$ denote the points of $\Pi$ arranged in decreasing order. 
 For all integers $s\ge1$ let $\Pi_s$ be an independent sample from $\textsf{P}(m_2)$, and let $(u_{st})_{t\ge1}$ denote the points of $\Pi_s$ arranged in decreasing order. Let $w_{st}\equiv u_s u_{st}$, and let
	\beq\label{e:rpc}
	\nu_{st} \equiv \f{w_{st}}
	{\displaystyle\sum_{s',t'\ge1} w_{s',t'}}\,.\eeq
The doubly infinite array $\vec{\nu}\equiv(\nu_{st})_{s,t\ge1}$ gives the weights of a 2-level \bemph{Ruelle probability cascade} with parameters $m_1,m_2$. We hereafter abbreviate this as $\vec{\nu}\sim\RPC(m_1,m_2)$. (For a survey of properties of this process see \cite{MR3052333}. For further motivations see
\cite{derrida1981random,derrida1985generalization,MR875300,MR2070334,MR2070335}.)

\subsection{Heuristic derivation of 2RSB functional}
\label{ss:as2}

The heuristic derivation in this subsection expands on the outline presented in \S\ref{ss:intro.interp}; and is a simple application of the well-known ``cavity method.'' There are too many instances of the method to be adequately cited here, but we point out a few influential works \cite{bethe1935statistical,mezard2001bethe,yedidia2001bethe,aizenman2003extended}. Our discussion is based on \cite{MR3052333}, and we follow similar notation.

For simplicity we continue to assume that $\alpha=d/k$ is an integer. Let $\GG_N$, $\GG_{N+1/2}$, and $\GG_{N+1}$ be as defined in \S\ref{ss:intro.interp}. For $\beta\ge0$ we consider the Gibbs measure $\mu_\beta$ defined by \eqref{e:gibbs.intermediate}, using the Hamiltonian of $\GG_{N+1/2}$. We assume that the finite-dimensional marginals of $\mu_\beta$ are given by \eqref{e:rpc.cavity}, which we repeat here for convenience:
	\[
	\mu_\beta(x_1,\ldots,x_\ell)
	\approx
	\sum_{s,t\ge1}\nu_{st}
	\prod_{i=1}^\ell \nu_{st}
	\prod_{i=1}^\ell w_{st,i}(x_i)\,.
	\]
We sample the weights $\nu_{st}$ from the $\RPC(m_1,m_2)$ law, as defined by \eqref{e:rpc}. We recall that the $w_{st,i}$ are generated recursively, as follows. Let
 $\PROB_0\equiv\PROB$ be the space of probability measures over $\set{\zro,\one}$, and for $r\ge1$ let $\PROB_r$ be the space of probability measures over $\PROB_{r-1}$. Let $\pQ\in \PROB_2$. Let $(r_{s,i})_{s,i}$ be i.i.d.\ samples from law $\pQ$. For each $i$ and each $s$, let $(w_{st,i})_{t\ge1}$ be a sequence of i.i.d.\ samples from $r_{s,i}$. Note $r_{s,i}\in\PROB_1$ so $w_{st,i}\in\PROB$. Recall that $\GG_{N+1/2}$ is formed by deleting from $\GG_N$ a set of $d(1-1/k)$ random clauses, which we denote $F'$. Then
	\begin{align*}
	\log \f{Z_N(\beta)}{Z_{N+1/2}(\beta)}
	&\approx \log
	\sum_{s,t\ge1} \nu_{st}
	\prod_{a\in F'}\bigg\{
	\sum_{\ux_{\partial a}}
	\exp\{-\beta\HH_a(\ux_{\partial a})\}
	\prod_{i\in\partial a}
	w_{st,i}(x_i)\bigg\}
	= \log
	\sum_{s,t\ge1} \nu_{st}
	\prod_{a\in F'} 
	\hat{X}_{a,\beta}( w_{st,\partial a} )\\
	&=\log \sum_{s\ge1} u_s
	\sum_{t\ge1} u_{st} 
	\prod_{a\in F'}
	\hat{X}_{a,\beta}( w_{st,\partial a} )
	- \log \sum_{s\ge1} u_s
	\sum_{t\ge1} u_{st}
	\end{align*}
where $\hat{X}_{a,\beta}( w_{st,\partial a} )$ is a random function (depending on the labels $\lit_e$ of the edges $e\in\delta a$) of the $k$-tuple of measures 
$w_{st,\partial a} \equiv (w_{st,i})_{i\in\partial a}$. Taking expectations and applying  Lemma~\ref{l:poissondirichlet} for the 
$(u_{st})_{t\ge1}$ gives
	\[
	\E\log \f{Z_N(\beta)}{Z_{N+1/2}(\beta)}
	\approx
	\E \log \sum_{s\ge1} u_s
	\prod_{a\in F'}
	\E_s \bigg(\hat{X}_{a,\beta}
		(w_{s1,\partial a})
	^{m_2}\bigg)^{1/m_2}
	\sum_{t\ge1} u_{st} 
	- \E \log \sum_{s\ge1} u_s
	\sum_{t\ge1} u_{st} 
	\]
where $\E_s$ denotes expectation conditional on the $r_{s,i}$ and on the $\set{\zro,\one}$-labels of the edges incident to $F'$. Let $\E'$ denote expectation conditional only on the $\set{\zro,\one}$ edge labels. Applying  Lemma~\ref{l:poissondirichlet} again for $(u_s)_{s\ge1}$ gives
	\begin{align}\nonumber
	\E\log \f{Z_N(\beta)}{Z_{N+1/2}(\beta)}
	&\approx
	\E \log \bigg\{
	\E' \bigg[ 
	\prod_{a\in F'}
	\E_1 \bigg(\hat{X}_{a,\beta}
		(w_{11,\partial a})
	^{m_2}\bigg)^{m_1/m_2}
	\bigg]^{1/m_1}
	\sum_{s\ge1} u_s
	\sum_{t\ge1} u_{st} 
	\bigg\}
	-\E \log \sum_{s\ge1} u_s
	\sum_{t\ge1} u_{st} \\
	&=
	\f{\alpha(k-1)}{m_1}
	\E \log \E' \bigg[ \E_1 \bigg(\hat{X}_{a,\beta}
		(w_{11,\partial a})
	^{m_2}\bigg)^{m_1/m_2}
	\bigg]
	\equiv 
	\f{\alpha(k-1)
	\E \log G_{\beta,m_1,m_2}(\pQ)}{m_1}
	\,,\label{e:cavity.clause}
	\end{align}
where the last equality defines $G_\beta$. 
We can write it more explicitly as
	\begin{align}\label{e:cavity.clause.explicit}
	G_{\beta,m_1,m_2}(\pQ)
	&=  \int
	\bigg(
	\int
	\hat{X}_{a,\beta}(w_{1:k})^{m_2}
	\prod_{i=1}^k
	\,d r_i(w_i)
	\bigg)^{m_1/m_2}
	\,\prod_{i=1}^k
	d\pQ(r_i)\,, \\
	\hat{X}_{a,\beta}(w_{1:k})
	&= \sum_{x_{1:k}}
	\exp\{-\beta\HH_a(x_{1:k})\}
	\prod_{i=1}^k w_i(x_i)\,.\nonumber
	\end{align}
We emphasize that the comparison \eqref{e:cavity.clause}  holds \emph{under the heuristic \eqref{e:rpc.cavity}}.
Under the same assumption we can likewise derive a comparison between $\GG_{N+1/2}$ and $\GG_{N+1}$
--- since this is very similar to the preceding calculation, we omit the details and simply state the result. Write $D\equiv d(k-1)$ as before, and denote
	\[
	\bigg(x_{aj}
	:1\le a\le d,2\le j\le k
	\bigg)
	\equiv
	x_{1:D}\,.
	\]
Let $\HH_1,\ldots,\HH_d$ be the Hamiltonians for $d$ random clauses, and let
	\[
	\dot{\HH}(x_{0:D})
	\equiv \dot{\HH}(x_0, x_{1:D})
	\equiv
	\sum_{a=1}^d
	\HH_a(x_0, x_{a2},\ldots,x_{ak})\,.
	\]
Then, analogously to \eqref{e:cavity.clause}, 
we have the comparison
	\beq\label{e:cavity.var}
	\E\log\f{Z_{N+1}(\beta)}{Z_{N+1/2}(\beta)}
	\approx
	\f{\E\log W_{\beta,m_1,m_2}(\pQ)}{m_1}\,,
	\eeq
where the explicit form of $W_\beta$ is given, analogously to \eqref{e:cavity.clause.explicit}, by
	\begin{align*}
	W_{\beta,m_1,m_2}(\pQ)
	&\equiv
	\int
	\bigg(
	\int
	X_\beta(w_{1:D})^{m_2} \,
	\prod_{i=1}^D dr_i(w_i)
	\bigg)^{m_1/m_2}
	\prod_{i=1}^D
	d\pQ(r_i)\,,\\
	X_\beta(w_{1:D})
	&=\sum_{x_{0:D}}
	\exp\bigg\{
	-\beta \dot{\HH}(x_{0:D})
	\bigg\}
	\prod_{i=1}^D w_i(x_i)\,.
	\end{align*}
Combining \eqref{e:cavity.clause} and \eqref{e:cavity.var} gives, under the heuristic assumption 
\eqref{e:rpc.cavity}, the comparison
	\beq\label{e:positive.temp.2rsb.functional}
	\E\log \f{Z_{N+1}(\beta)}{Z_N(\beta)}
	\approx
	\f1{m_1}\E\bigg\{
	\log W_{\beta,m_1,m_2}(\pQ)
	- \alpha(k-1)\log G_{\beta,m_1,m_2}(\pQ)
	\bigg\}
	\equiv 
	\Phi_{\beta,m_1,m_2}(\pQ)\,,
	\eeq
where the last identity defines the \bemph{(positive-temperature) 2RSB functional} $\Phi_{\beta,m_1,m_2}$.
As we review next, one side of 
\eqref{e:positive.temp.2rsb.functional}
can be made rigorous via an interpolation bound
(Proposition~\ref{p:interp.pos.temp} below).

\subsection{General interpolation bound}\label{ss:interp}

Let $\GG_N$ be an instance of random $d$-regular $k$-\textsc{nae-sat} on $N$ variables, with Hamiltonian $\HH_N$. As before, let
	\beq\label{e:Z.beta}
	Z_N(\beta)
	\equiv \sum_{\ux} \exp\bigg\{-\beta\HH_N(\ux)\bigg\}\eeq
where the sum goes over $\ux\in\set{\zro,\one}^N$.
The following is a direct consequence of prior results:

\begin{ppn}[{proved in \cite[\S E.4]{ssz}}]
\label{p:interp.pos.temp}
Let $\GG_N$ be an instance of random $d$-regular $k$-\textsc{nae-sat} on $N$ variables, and let $Z_N(\beta)$ be as in \eqref{e:Z.beta}. If $\E$ denotes expectation over the law of $\GG_N$, then
	\[
	\f{\E\log Z_N(\beta)}{N\beta}
	\le
	\f{\Phi_{\beta,m_1,m_2}(\pQ)}{\beta}
	+ O_{m_1,m_2,\beta,\pQ}\bigg(\f1{N^{1/3}} \bigg)\,,
	\]
for any $\beta\ge0$, $0<m_1<m_2<1$, and $\pQ\in\PROB_2$.
\end{ppn}

\noindent To conclude, we take $\beta\to\infty$ to deduce the zero-temperature bound Proposition~\ref{p:interpolation}:

\begin{proof}[Proof of Proposition~\ref{p:interpolation}]
Since $Z_N(\beta) \ge \exp\{-N\beta\emin(\GG_N)\}$,  Proposition~\ref{p:interp.pos.temp} implies
	\[
	-\einf
	= \limsup_{N\to\infty}\E[-\emin(\GG_N)]
	\le
	\limsup_{N\to\infty}\f{\E[\log Z_N(\beta)]}{N\beta}
	\le \f{\Phi_{\beta,m_1,m_2}(\pQ)}{\beta}
	\]
for any $\beta\ge0$, $0<m_1<m_2<1$, and $\pQ\in\PROB_2$.
Now let $w^\zro\equiv \I_\zro$, $w^\one\equiv \I_\one$, and $w^\free \equiv (\I_\zro+\I_\one)/2$: in this way, for each $\ww\in\set{\zro,\one,\free}$ we have defined an element  $w^\ww\in\PROB$. Recall that $\Omega$ is the space of probability measures over $\set{\zro,\one,\free}$: for each $\rho\in\Omega$ we define $r^\rho \in\PROB_1$  which is
 supported only on the three points $w^\zro, w^\zro, w^\free$:
	\[r^\rho
	\equiv
	\sum_{\ww\in\set{\zro,\one,\free}}
	\rho(\ww)
	\I_{w^\ww}\,.
	\]
Finally, if $Q$ is a probability measure over $\rho\in \Omega$, we let $\pQ$ be the induced law of $r^\rho$.
(Formally $\pQ=\mathfrak{r}_\sharp Q$ if $\mathfrak{r}$ denotes the mapping $\rho\mapsto r^\rho$.) Then, as $\beta\to\infty$ we have
	{\setlength{\jot}{0pt}\begin{align*}
	\hat{X}_{a,\beta}(w_{1:k})^{y_2/\beta}
	&\to \exp\{ -y_2 \hph(\ww_{1:k})\}\,,\\
	X_\beta(w_{1:D})^{y_2/\beta}
	&\to \exp\{ -y_2 \vph(\ww_{1:D})\}\,.
	\end{align*}}%
It follows from this that as $\beta\to\infty$ we have
	{\setlength{\jot}{0pt}\begin{align*}
	G_{\beta,y_1/\beta,y_2/\beta}(\pQ)
	&\to \FG(y_1,y_w,Q)\,,\\
	W_{\beta,y_1/\beta,y_2/\beta}(\pQ)
	&\to \WW(y_1,y_w,Q)\,.
	\end{align*}}%
Therefore $\beta^{-1}\Phi_{\beta,y_1/\beta,y_2/\beta}(\pQ)
\to \PhiTwo(y_1,y_2,Q)$, and the result follows.
\end{proof}

\subsection{Stationarity equations}
\label{ss:stationarity}

We next verify that fixed points of the SP recursion correspond to stationary points of $F(x,w,y)$.

\begin{proof}[Proof of Lemma~\ref{l:stationarity}] It is straightforward to check with $w(x)$ as defined by \eqref{e:clause.x.to.w} we have
	\[
	\f{\partial F}{\partial x}(x,w,y)
	= \f{\alpha k(1-e^{-y})}{2}
	\bigg\{ \f{w(x)}
	{1-(1-x)w(x)(1-\AM)}
	-\f{w}{1-(1-x)w(1-\AM)}\bigg\}
	\]
which is zero if $w=w(x)$. The partial derivative with respect to $w$ is slightly more involved. We first consider the normalizing constant from
\eqref{e:sp.var.marginal}, which is written above as
\eqref{e:sp.var.marginal.Zdw}, and with more explicit expressions given in 
\eqref{e:Z.d.zero.as.sum} and
\eqref{e:Z.d.free.as.sum}. Denote
	\[
	\dq_y(\ell_\zro,\ell_\one)
	\equiv \f1{\dz(w_y)}
	\binom{d-1}{\ell_\zro,\ell_\one}
	\f{\hq_y(\zro)^{\ell_\zro}
	\hq_y(\one)^{\ell_\one}
	\hq_y(\free)^{d-1-\ell_\zro-\ell_\free}}
	{\exp\{y \min\set{\ell_\zro,\ell_\one}\}}\,.
	\]
We will write, for instance, $\dq_y(\ell_\zro \ge \ell_\one+2)$ for the sum of $\dq_y(\ell_\zro,\ell_\one)$ over all pairs $(\ell_\zro,\ell_\one)$ satisfying $\ell_\zro \ge \ell_\one+2$. 
With this notation, $\dq_y(\zro)=\dq_y(\one)=\dq_y(\ell_\zro\ge\ell_1+1)$
while $\dq_y(\free)=\dq_y(\ell_\zro=\ell_\one)$.
By decomposing
\eqref{e:sp.var.marginal} according to the first warning $\hw_1$, we can compare it with the normalizing constant $\dz(w)$ from the \textsc{sp} recursion \eqref{e:variable.w.to.x}:
	\begin{align*}
	\f{\dfz_\zro(w_y)}{\dz(w_y)}
	&= 
	\hq_y(\zro)\bigg(
		\dq_y(\zro)+\dq_y(\free)\bigg)
	+\hq_y(\free)\dq_y(\zro)
	+\f{\hq_y(\one)\dq_y(\ell_\zro \ge \ell_1+2)}{e^y}
	\,,\\
	\f{\dfz_\free(w_y)}{\dz(w_y)}
	&=\f{\hq_y(\zro)\dq_y(\ell_\one=\ell_\zro+1)}{e^y}
	+\f{\hq_y(\one)\dq_y(\ell_\zro=\ell_\one+1)}{e^y}
	+\hq_y(\free)\dq_y(\free)\,.
	\end{align*}
Combining with the symmetries
$\dq_y(\zro)=\dq_y(\one)$ and
$\hq_y(\zro)=\hq_y(\one)$ gives (after some algebra)
	\[
	\f{\dfz(w_y)}{\dz(w_y)}
	=2\hq_y(\zro)\bigg(\dq_y(\zro)
	(1+e^{-y})
	+\dq_y(\free)
		\bigg)
	+\hq_y(\free)
	=1-w_y(1-x_y)(1-\AM)\,.\]
In fact, by essentially the same derivation it holds for all $w$ that
	\beq\label{e:Zd.Zdminusone.reln}
	\f{\dfz(w_y)}{\dz(w)}
	= 1-w(1-\tilde{x}(w))(1-\AM)
	\eeq
with $\tilde{x}(w)$ as in \eqref{e:variable.w.to.x}.
Next, differentiating the above expressions for $\dz_\zro(w)$ and $\dz_\free(w)$ gives
	\begin{align}\nonumber
	{(\dfz_\zro)'(w)}
	&=\sum_{\ell=0}^d
	\binom{d}{\ell}
	(w\cdot\AM)^\ell(1-w)^{d-\ell}\PPP_\ell
	\f{\ell-dw}{w(1-w)}\\
	&= -\f{d {\dfz_\zro(w)}}{1-w}
	+ d w \cdot\AM \sum_{\ell=1}^d
	\binom{d-1}{\ell-1} (w\cdot\AM)^{\ell-1}
	(1-w)^{d-\ell}\PPP_\ell
	\label{e:Zd.prime.zro}\\
	{(\dfz_\free)'(w)}
	&=-\f{d {\dfz_\free(w)}}{1-w}
	+dw\cdot\AM
	\sum_{\ell=1}^d
	\binom{d-1}{\ell-1} (w\cdot\AM)^{\ell-1}
	(1-w)^{d-\ell}\QQQ_\ell
	\label{e:Zd.prime.free}\,.
	\end{align}
For integers $i\ge1$ let $I_i$ be i.i.d.\ Bernoulli random variables with $\E I_i =  p \equiv 1/(e^y+1)$. For any finite subset $S$ of positive integers let $Y(S)$ be the sum of $I_i$ over $i\in S$. Abbreviating $[\ell]\equiv \set{1,\ldots,\ell}$, we have $\PPP_\ell \equiv \P(E_\ell)$ where $E_\ell\equiv \set{Y([\ell]) < \ell/2}$, and $\QQQ_\ell=\P(F_\ell)$
where $F_\ell\equiv\set{Y([\ell])=\ell/2}$. We then consider two cases:
\begin{enumerate}[a.]
\item If $\ell$ is odd, then $E_{\ell-1}=\set{Y([\ell-1])<(\ell-1)/2}=\set{Y([\ell-1])\le (\ell-1)/2-1}$, and it implies
	\[
	Y([\ell])
	\le Y([\ell-1])+1
	\le \f{\ell-1}{2} < \f{\ell}{2}\,,
	\]
so $E_{\ell-1}\subseteq E_\ell$. On the event 
$E_\ell\setminus E_{\ell-1}$ we must have
	\[
	\f{\ell-1}{2} \le Y([\ell-1])\le Y([\ell]) < \f{\ell}{2}\,,
	\]
which means $
Y([\ell-1])= Y([\ell])=\ell/2$. Therefore
$E_\ell\setminus E_{\ell-1}
=F_{\ell-1}\cap \set{I_\ell=0}$, and so
	\[
	\PPP_\ell+\f{\QQQ_\ell}{2}
	=\PPP_\ell
	=\PPP_{\ell-1} + (1-p)\QQQ_{\ell-1}\,,
	\]
where the first equality holds simply because $\QQQ_\ell=0$ for $\ell$ odd.

\item If $\ell$ is even, then $E_{\ell-1}=\set{Y([\ell-1])<(\ell-1)/2}=\set{Y([\ell-1])
\le \ell/2-1}$, and it implies
	\[
	Y([\ell])
	\le Y([\ell-1])+1
	\le \f{\ell}{2}\,,
	\]
so $E_{\ell-1}\subseteq E_\ell \cup F_\ell$. On the event $(E_\ell \cup F_\ell)\setminus E_{\ell-1}$ we must have
	\[ \f{\ell-1}{2} \le Y([\ell-1]) \le
	Y([\ell]) \le \f{\ell}{2}\,,
	\]
which means $Y([\ell-1]) =Y([\ell]) = \ell/2$. Therefore 
$(E_\ell \cup F_\ell)\setminus E_{\ell-1}
=\set{Y([\ell-1])=\ell/2} \cap\set{I_\ell=0}$, and so
	\[
	\PPP_\ell + \QQQ_\ell
	=\PPP_{\ell-1}
	+(1-p) \P\bigg(
	\Bin(\ell-1,p) =\f{\ell}{2}
	\bigg)\,.
	\]
Recalling the definition of $\QQQ_\ell$ then gives
	\[
	\PPP_\ell + \f{\QQQ_\ell}{2}
	= \PPP_{\ell-1}
	+(1-p) \P\bigg(
	\Bin(\ell-1,p) =\f{\ell}{2}
	\bigg)
	-\f12\P\bigg(
	\Bin(\ell,p) =\f{\ell}{2}
	\bigg)
	=\PPP_{\ell-1}\,,
	\]
where the last step is by a simple algebraic manipulation of the binomial coefficients.
\end{enumerate}
Combining
\eqref{e:Zd.prime.zro} and
\eqref{e:Zd.prime.free} gives
	\[{\dfz'(w)}
	=-\f{d\dfz(w)}{1-w}
	+\f{2d\cdot\AM}{1-w}
	\sum_{\ell=1}^d
	\binom{d-1}{\ell-1} (w\cdot\AM)^{\ell-1}
	(1-w)^{d-\ell}
	\bigg(
	\PPP_\ell+\f{\QQQ_\ell}{2}\bigg)\,.\]
Substituting the above expressions for 
$\PPP_\ell+\QQQ_\ell/2$, then re-indexing $\ell-1$ as $\ell$, gives
	\begin{align*}
	{\dfz'(w)}
	&=-\f{d\dfz(w)}{1-w}
	+\f{2d\cdot\AM}{1-w}
	\sum_{\ell=0}^{d-1}
	\binom{d-1}{\ell} (w\cdot\AM)^\ell
	(1-w)^{d-\ell}\bigg( \PPP_\ell
	+(1-p)\QQQ_\ell\bigg)\\
	&=-\f{d\dfz(w)}{1-w}
	+\f{2d\cdot\AM}{1-w}
	\bigg(\dz_\zro(w)
	+(1-p)\dz_\free(w)\bigg)\\
	&=-\f{d\dfz(w)}{1-w}
	+\f{d\cdot\AM\cdot \dz(w)}{1-w}
	\bigg(
	1
	+
	\f{(1-\AM)\tilde{x}(w)}{\AM}\bigg) \,.
	\end{align*}
Finally, combining with \eqref{e:Zd.Zdminusone.reln} gives
	\[\f{\dfz'(w)}{\dfz(w)}
	=\f{-d[1-w(1-\tilde{x}(w))(1-\AM)]
	+ d[\AM
	+(1-\AM)\tilde{x}(w)]}{(1-w)
	[1-w(1-\tilde{x}(w))(1-\AM)]
	}
	=-\f{d(1-\AM)(1-\tilde{x}(w))}
	{1-w(1-\tilde{x}(w))(1-\AM)}\,,\]
where again the last equality is by some simple algebra. Altogether we obtain
	\[
	\f{\partial F}{\partial w}(x,w,y)
	=d(1-\AM)\bigg\{
	\f{1-x}{ 1- (1-x)w(1-\AM)}
	-\f{1-\tilde{x}(w)}
	{ 1- (1-\tilde{x}(w))w(1-\AM)}
	\bigg\}\,,
	\]
which is zero if $\tilde{x}(w)=x$.\end{proof}

\subsection{Convexity}
\label{ss:convexity}
We finally prove Proposition~\ref{p:second.derivative}. For this purpose it is useful to re-express \eqref{e:Fxwy} as
	\begin{align*}
	F(x,w,y)
	&\equiv 
	\log\dZ-d\bigg\{
	\log\eZ-\f1k\log\hZ\bigg\}\\
	&=
	\log \f{\dfz(w)}{(1-w)^d}
	-d\bigg\{
	\log \f{\hfz(x)}{(1-x)^k}
	-\f1k
	\log \f{\efz(x)}{(1-x)(1-w)}
	\bigg\}\,.
	\end{align*}
It is convenient to reparametrize $(x,w)$ as $(X,W)$
	\[
	\bigg(\f{1-x}{x},\f{1-w}{w}\bigg)
	\equiv (e^X,e^W)\,,
	\]
and then write $F(x,w,y)\equiv G(X,W,y)$.
In the above we have $\dZ\equiv\dZ(W,y)$,
$\hZ\equiv\hZ(X,y)$, and $\eZ\equiv \eZ(W,X,y)$.
We abbreviate $G_X$, $G_W$, $G_{XW}$, and so on for partial derivatives of $G$ evaluated at the point $(X_y,W_y,y)$.
 We divide the proof of Proposition~\ref{p:second.derivative} into a series of lemmas, which occupy the remainder of this section.

\begin{lem}\label{l:second.derivative}
In the setting of Proposition~\ref{p:sp} we have
	\[
	\FF''(y)
	=G_{yy}
	+\f{G_{XX}(G_{Wy})^2 
		-2G_{XW}G_{Xy}G_{Wy}
		+G_{WW}(G_{Xy})^2
	}{(G_{XW})^2-G_{XX}G_{WW}}\,.
	\]
as long as $(G_{XW})^2-G_{XX}G_{WW}\ne0$.

\begin{proof}
Since the pair $(x_y,w_y)$ is a stationary point of $F(\cdot,\cdot,y)$ by Lemma~\ref{l:stationarity}, the corresponding pair $(X_y,W_y)$ is a stationary point of $G(\cdot,\cdot,y)$, that is, $G_W=G_X=0$.
Since this holds for all $y$ in the range of Proposition~\ref{p:sp}, we can differentiate once more in $y$ to find
	\[
	\begin{pmatrix}
	G_{WW} & G_{WX}\\G_{XW} & G_{XX}
	\end{pmatrix}
	\begin{pmatrix} X'(y) \\ W'(y)
	\end{pmatrix} = - \begin{pmatrix}
	G_{Wy}\\G_{Xy}
	\end{pmatrix}\,.
	\]
On the other hand, recall that
 $-\ee(y)=\FF'(y) = G_y$. It follows that
	\[
	-\ee'(y)
	=\FF''(y)
	=\begin{pmatrix}
	G_{Xy} & G_{Wy}
	\end{pmatrix}
	\begin{pmatrix} X'(y) \\ W'(y)\end{pmatrix}
	+ G_{yy}\,.
	\]
Combining the above equations gives the claim.
\end{proof}
\end{lem}

To simplify the above,
note that $\dZ(W)$ is the moment-generating function of the number $\dL$ of $\set{\zro,\one}$-warnings incoming to a variable; while $\hZ(X)$ is the moment-generating function of the number $\hJ$ of $\free$-warnings incoming to a clause. Similarly, $\eZ(X,W)$ is the joint moment-generating function of the pair $(L,J)$ where $L$ is the indicator that the clause-to-variable warning on the edge is in $\set{\zro,\one}$, while $J$ is the indicator that the variable-to-clause warning on the edge is $\free$. We therefore have, for example,
	\[
	(\log\hZ)_{Xy}
	=\f{\hZ_{Xy}}{\hZ}
	-\bigg(\f{\hZ_X}{\hZ}\bigg)
	\bigg(\f{\hZ_y}{\hZ}\bigg)
	=\av{\hph\hJ}-\av{\hph}\av{\hJ}
	\equiv\Cov(\hph,\hJ)
	\]
where $\av{}$ is the average with respect to the measure $\hat{\nu}_y$ of \eqref{e:sp.clause.marginal}.
To simplify further we decompose $\dL=L_1+\ldots+L_d$ where $L_j$ is the indicator that the $j$-th incoming warning to the variable is in $\set{\zro,\one}$; and similarly
$\hJ=J_1+\ldots +J_k$ where $J_i$ is the indicator
that the $i$-th incoming warning to the clause is $\free$.
Then
	\[
	G_{Xy}
	=d\bigg\{ \f1k \Cov(\hph,\hJ)-\Cov(\eph,J)
	\bigg\}
	= d\bigg\{ \Cov(\hph,J_1)-\Cov(\eph,J)\bigg\}
	=0\,.
	\]
This allows us to simplify the result of Lemma~\ref{l:second.derivative}:
as long as $(G_{XW})^2-G_{XX}G_{WW}\ne0$,
we have
	\beq\label{e:second.derivative.simp}
	\FF''(y)
	=G_{yy}
	+\f{G_{XX}(G_{Wy})^2 }{(G_{XW})^2-G_{XX}G_{WW}}\,.
	\eeq
We next turn to the evaluation of $(G_{XW})^2-G_{XX}G_{WW}$.

\begin{lem}\label{l:determinant}
In the setting of Proposition~\ref{p:sp} we have
	\beq\label{e:GXW.GXX}
	G_{XW}
	=\f{(1-w)G_{XX}}{(k-1)x}
	= -d\Cov(J,L)
	\asymp -\CC^{1/2} k x\,.
	\eeq
The quantity $(G_XW)^2-G_{XX}G_{WW}$ is positive:
moreover, recalling $\branch\equiv(d-1)(k-1)$, we have
	\beq\label{e:determinant.ratio}
	\bigg|\f{G_{XX}G_{WW}}{(G_{XW})^2}\bigg|
	=\bigg|
	\f{\branch \Cov(L_1,L_2)\Cov(J_1,J_2)}
		{\Cov(J,L)^2}
	\bigg|
	\le 
	O\bigg( \f1{e^{\Omega(k)}}\bigg)\,.
	\eeq
In particular, positivity of $(G_XW)^2-G_{XX}G_{WW}$ justifies the formula \eqref{e:second.derivative.simp}.

\begin{proof}
Following similar notation as above, we have
	{\setlength{\jot}{0pt}\begin{align*}
	- G_{XW}
	&=  d(\log\eZ)_{XW}
	= d\Cov(J,L)\,,\\
	G_{WW}
	&=(\log\dZ)_{WW}- d(\log\eZ)_{WW}
	=\Var\dL - d\Var L
	=d(d-1)\Cov(L_1,L_2)\,,
	\\
	G_{XX}
	&=\alpha(\log\hZ)_{XX}- d(\log\eZ)_{XX}
	=\alpha\Var\hJ-d\Var J
	=d(k-1)\Cov(J_1,J_2)\,.
	\end{align*}}%
Recall  that
$\efz=1-w(1-x)(1-\AM)$. This allows us to evaluate
	\[
	\Cov(J,L)
	=-\Cov(J,1-L)
	= -\f{\dq(\free)\hq(\free)}{\efz}
	\bigg\{1 - \f1{\efz}\bigg\}
	= \f{xw(1-w)(1-x)(1-\AM)}{(\efz)^2}\,,
	\]
from which we conclude
$d\Cov(J,L)\asymp dxw(1-\AM)
\asymp \CC^{1/2} k x$, where the last step uses
\eqref{e:y.restriction}.
 Similarly,
	\[
	\Cov(J_1,J_2)
	=\f{\dq(\free)^2}{\hfz}
	\bigg\{1-\f1{\hfz}\bigg\}
	=
	-\f{\dq(\free)}{\hq(\free)}\Cov(J,L)
	=
	-\f{x}{1-w}\Cov(J,L)
	\,.
	\]
This implies \eqref{e:GXW.GXX}.
Combining the above calculations, we find
	\beq\label{e:second.partials.ratio.simp}
	\f{G_{XX}G_{WW}}{(G_{XW})^2}
	= \f{\branch \Cov(L_1,L_2)\Cov(J_1,J_2)}
		{\Cov(J,L)^2}
	= -\f{\branch \Cov(L_1,L_2)x}
		{\Cov(J,L)(1-w)}
	\asymp
	- \f{dk\Cov(L_1,L_2)}{w(1-\AM)}\,.
	\eeq
We next turn to the estimation of $\Cov(L_1,L_2)
=\Cov(1-L_1,1-L_2)$. Note that $\av{L_1}=\av{L_2}=\av{L}$. We have
	\begin{align}\label{e:av.one.minus.L}
	\av{1-L}
	&=\f{\hq(\free)}{\efz}
	=\f{1-w}{1-w(1-x)(1-\AM)}\,,\\
	\av{(1-L_1)(1-L_2)}
	&=\f{\hq(\free)^2(S_0 + 2S_{\ge1})}{\dfz}
	= \f{(1-w)^2(S_0 + 2S_{\ge1})}{\dfz}\,.
	\nonumber
	\end{align}
Recall from \eqref{e:dz.S} 
and \eqref{e:Zd.Zdminusone.reln} that $\dz=  S_0 + 2S_{\ge1}(1-w(1-\AM))$
and $\dfz/\dz = \efz=1-w(1-x)(1-\AM)$. Then
	\begin{align*}
	\f{\Cov(L_1,L_2)}{(1-w)^2}
	&= \f{1}{\dfz}
	\bigg\{
	S_0 + 2S_{\ge1}(1-w(1-\AM))
		+2S_{\ge1}w(1-\AM)
	\bigg\}
		- \bigg(\f{1}{\efz}\bigg)^2\\
	&=\f1{\efz}
	\bigg\{
	\f{2S_{\ge1} w(1-\AM)}{S_0 + 2S_{\ge1}(1-w(1-\AM))}
	-\f{w(1-x)(1-\AM)}{1-w(1-x)(1-\AM)}
	\bigg\}\\
	&=\f{(1- \acute{x})w(1-\AM)}
		{1-(1-\acute{x})w(1-\AM)}
	-\f{(1-x)w(1-\AM)}{1-(1-x)w(1-\AM)}
	\asymp w(1-\AM)(x-\acute{x})\,,
	\end{align*}
where $\acute{x}\equiv 2S_{\ge1}/(S_0+2S_{\ge1})$.
To compare $x$ with $\acute{x}$, note
	\[
	x
	= \f{\hq(\free) S_0 + 2\hq(\one) S_1/e^y}{\dz}
	= \f{(1-w) S_0 + w S_1/e^y}{
	S_0 + 2S_{\ge1}(1-w(1-\AM))
	}= \acute{x}\bigg\{
	1 + O(w (1-\AM))\bigg\}	\]
where the last estimate is easy if $y\asymp1$, and uses Lemma~\ref{l:S.estimates} if $y$ is small. Substituting into \eqref{e:second.partials.ratio.simp} gives
 	\[
	\bigg|\f{G_{XX}G_{WW}}{(G_{XW})^2}\bigg|
	\asymp dk
	w (1-\AM) x 
	\asymp \CC^{1/2} k^2 x\,,
	\]
where the last step uses \eqref{e:y.restriction} again.
The assertion
\eqref{e:determinant.ratio} follows since $\dq\in\MMstar$ by Proposition~\ref{p:sp}.
\end{proof}
\end{lem}

\begin{lem}
In the setting of Proposition~\ref{p:sp} we have
	\[
	0\le -\f{G_{XX}(G_{Wy})^2 }{(G_{XW})^2}
	\le \f{\CC^{1/2}}{e^y e^{\Omega(k)}}\,.
	\]

\begin{proof}
We start by noting that
if $\dw_1=\dWP(\hw_{2:d})$, then we can decompose
	\[
	\dph(\hw_{1:d})
	=\min\bigg\{
	\#\set{1\le i\le d : \hw_i=\zro},
	\#\set{1\le i\le d : \hw_i=\one}
	\bigg\}
	=\eph(\hw_1,\dw_1)
	+ \dph_{d-1}(\hw_{2:d})\,.
	\]
We abbreviate the above as
$\dph=\eph_1+\bph$ where
$\eph_1\equiv\eph(\hw_1,\dw_1)$
and $\bph\equiv \dph_{d-1}(\hw_{2:d})$. Then
	\[
	G_{Wy}
	=\Cov(\dL,\dph)
	-d\Cov(L,\eph)
	=d\bigg\{
	\Cov(L_1,\dph)-\Cov(L,\eph)
	\bigg\}
	=d\Cov(L,\bph)\,.
	\]
Recall from \eqref{e:arith.Z}
and \eqref{e:geom.Z} the definitions of $\LAM\equiv L_{d-1,\AM}$ and $\LGM\equiv L_{d-1,\GM}$.
Recalling Lemma~\ref{l:Ld.estimates} and 
\eqref{e:energy.vertex.term.AM}, denote
	\[
	e_\zro\equiv\E\bigg\{
	\E\bigg( X\,\bigg|\,X<\f{\LAM}{2}\bigg)
	\bigg\}
	= \f{(d-1)w}{2e^y}
	\bigg\{
	1 - O\bigg( \f1{\CC^{1/2}e^{\Omega(k)}} \bigg)
	\bigg\}
	\]
as well as $e_\free\equiv (\E\LGM)/2$. We then evaluate
	\begin{align*}
	\av{(1-L_1)\bph}
	&=\f{\hq(\free)}{\dfz}
	 \bigg(
	 2\dz_\zro e_\zro + \dz_\free e_\free
	 \bigg)
	= \f{(1-w)}{\efz}
	\bigg\{
	(1-x)e_\zro + x e_\free
	\bigg\}
	\,,\\
	\av{\bph}
	&= \f{2\hq(\zro)}{\dfz}
	\bigg( \dz_\zro e_\zro
		+\dz_\free e_\free+\f{\dz_\one e_\zro}{e^y}
	 \bigg)
	 +\f{\hq(\free)}{\dfz}
	 \bigg(
	 2\dz_\zro e_\zro + \dz_\free e_\free
	 \bigg)=
	\f{1}{\efz}
	\bigg\{
	(1-x) e_\zro (1-w(1-\AM))+x e_\free
	\bigg\}\,.
	\end{align*}
We already calculated $\av{1-L}$ above
in \eqref{e:av.one.minus.L}, and combining these gives
\begin{align*}
	\f{\Cov(1-L_1,\bph)}{(1-w)/\efz}
	&=
	\bigg\{(1-x) e_\zro + x e_\free\bigg\}
	-\f1{\efz}
	\bigg\{(1-x) e_\zro (1-w(1-\AM))+x e_\free\bigg\}\\
	&=\f{xw(1-x)(1-\AM)}{\efz}
	( e_\zro-e_\free)
	=O\bigg(xw(1-\AM) \f{\CC^{1/2}k}{e^{y/2}}\bigg)
	=O\bigg(
	\f{kxw}{e^{y/2}}
	\bigg)\,,
	\end{align*}
--- the second-to-last step is similar to \eqref{e:energy.vertex.term.GM.large}~and~\eqref{e:energy.vertex.term.GM.small}, and the last step is by \eqref{e:y.restriction}. Combining with \eqref{e:GXW.GXX} gives
	\[
	-\f{G_{XX}(G_{Wy})^2}{(G_{XW})^2}
	\asymp-\f{ kx (G_{Wy})^2}{G_{XW}}
	\asymp\f{(G_{Wy})^2}{\CC^{1/2}}
	\asymp
	\f{\CC^{3/2} k^4 x^2}{e^y}\,.
	\]
The result follows by recalling $\dq\in\MMstar$.
\end{proof}
\end{lem}

\begin{lem}\label{l:Gyy}
In the setting of Proposition~\ref{p:sp} we have
	\[
	G_{yy} 
	\asymp \f{d}{k} \Var\eph
	\asymp \f{\CC}{e^y}\,.
	\]
\begin{proof}
As before, we write $\dph$ for the random variable $\dph(\hw_{1:d})$ where $\hw_{1:d}$ is sampled from the measure $\dot{\nu}_y$ of \eqref{e:sp.var.marginal}
and write $\eph$ for the random variable $\eph(\dw,\hw)$ where $(\dw,\hw)$ is sampled from the measure $\bar{\nu}_y$ of \eqref{e:sp.edge.marginal}. We then have
	\[
	G_{yy}
	= (\log\dZ)_{yy}-d\bigg(1-\f1k\bigg) (\log\eZ)_{yy}
	= \Var \dph
	-d\bigg(1-\f1k\bigg)\Var \eph\,.
	\]
Let $\dF\equiv \Ind{\dWP(\hw_{1:d})=\free}$,
$F\equiv \Ind{\dw=\free}$,
$\dH\equiv 1-\dF$, and $H\equiv 1-F$. Then
	\[
	\dph
	= \dH \sum_{i=1}^d \eph_i
	+ \f{\dF}{2}\sum_{i=1}^d \eph_i\,,
	\]
while $\eph= F\eph + H\eph$. This allows us to expand
	\begin{align*}
	G_{yy}
	&=d\bigg\{\Var(H\eph)
	+(d-1)\Cov(\dH\eph_1,\dH\eph_2)
	+\f{1}{4}\Var(F\eph)
	+\f{(d-1)}{4}\Cov(\dF\eph_1,\dF\eph_2)
	- d \av{H\eph}\av{F\eph}\bigg\}\\
	&\qquad -d\bigg\{
	\Var(H\eph)+\Var(F\eph)-2\av{H\eph}\av{F\eph}
	\bigg\}
	+\f{d}{k}\Var\eph
	\end{align*}
Rearranging gives
	\begin{align*}
	\f{G_{yy}}{d}
	&=(d-1)\bigg\{
	\Cov(\dH\eph_1,\dH\eph_2)-\f{\av{H\eph}\av{F\eph}}{2}
	\bigg\}
	+\f{(d-1)}{4}
	\bigg\{
	\Cov(\dF\eph_1,\dF\eph_2)
	-2\av{H\eph}\av{F\eph}
	\bigg\}
	\\
	&\qquad -\f34 \Var(F\eph)
	+\av{H\eph}\av{F\eph}
	+\f1k\Var\eph\,.
	\end{align*}
Recall from \eqref{e:def.S.i}
and \eqref{e:def.S.geq.i} the definitions of $S_i$ and $S_{\ge1}$. We then evaluate
	\begin{align*}
	\av{\dH\eph_1\eph_2}
	&= \av{H\eph}
		- \f{w(1-w) S_{\ge2}}{e^y\dfz}
	= \f{w^2 S_{\ge3}}{2 e^{2y}\dfz}
	\,,\\
	\av{F\eph_1\eph_2}
	&= \av{F\eph}-\f{w(1-w)S_1}{e^y\dfz}
	= \f{w^2}{2 e^y}
	\bigg(\f{S_0}{\dfz} + \f{S_2}{e^y\dfz} \bigg)\,.
	\end{align*}
Recall from \eqref{e:dz.S}  and \eqref{e:Zd.Zdminusone.reln} that $\dz=  S_0 + 2S_{\ge1}(1-w(1-\AM))$ and $\dfz/\dz = \efz=1-w(1-x)(1-\AM)$. Recall also Lemma~\ref{l:S.estimates}, and note that combining \eqref{e:S.estimate.small.l}~with~\eqref{e:S.estimate.large.l} gives in general $S_i(\GM)^i \le O(\min\set{1,(\lgm)^i}) S_0$. We then have
$\av{H\eph}\asymp w/e^y$ and
$\av{F\eph}\asymp w^2 x/e^y$, so
	\[\f1k\Var\eph
	\asymp \f{\av{H\eph}}{k}
	\asymp \f{w}{k e^y}\,.\]
We will argue that this is the main contribution, so that $G_{yy} \ge \Omega(\alpha \Var\eph)$. To this end, note
	\[
	\f{\Var(F\eph)}{k^{-1}\Var\eph}
	\asymp \f{ \av{F\eph} }{k^{-1}\Var\eph}
	\asymp \f{ke^y}{w}\f{w^2 x}{e^y}
	\asymp kw x \le
	kw \le \f1{e^{\Omega(k)}}\,.
	\]
The remaining contributions to $G_{yy}$ require more careful analysis. First we expand
	\begin{align*}
	A&\equiv
	\Cov(\dH\eph_1,\dH\eph_2)-\f{\av{H\eph}\av{F\eph}}{2}
	=\f{w^2 S_{\ge3}}{2 e^{2y}\dfz}
	-\bigg\{1+O(w)\bigg\}
		\f{wS_{\ge2}}{e^y\dfz}
		\bigg\{
		\f{wS_{\ge2}}{e^y\dfz}
		+\f12
	\f{wS_1}{e^y\dfz} \bigg\} \\
	&=\f{w^2}{ e^{2y}\dfz}\bigg\{
	 \f{S_{\ge3}}{2}
	-\f{[1+O(w)]S_{\ge2}}{S_0 + 2 S_{\ge1}}
		\bigg(
		S_{\ge2}+\f{S_1}{2e^y} \bigg)
	\bigg\}
	= \f{w^2}{2 e^{2y}\dfz}
	\bigg\{
	S_1-S_2
	+\f{S_0}{2}
	-\f{S_1}{2e^y}
	+O(w \dz)
	\bigg\}
	\end{align*}
where the last step is obtained by expanding
$S_{\ge3}=S_{\ge1}-S_1-S_2$, etc., and simplifying. Similarly,
	\begin{align*}
	B&\equiv
	\Cov(\dF\eph_1,\dF\eph_2)
	-2\av{H\eph}\av{F\eph}
	=\f{w^2}{2 e^y}
	\bigg(\f{S_0}{\dfz} + \f{S_2}{e^y\dfz} \bigg)
	-\bigg\{1+O(w)\bigg\}
	\f{wS_1}{e^y\dfz}
	\bigg\{
	\f{wS_1}{e^y\dfz}+ \f{2w S_{\ge2}}{e^y\dfz}
	\bigg\}\\
	&=\f{w^2}{ e^y \dfz}
	\bigg\{
	\f{S_0}{2} + \f{S_2}{2e^y} 
	-\f{S_1}{e^y}
	+O(x(x+w) \dz)
	\bigg\}\,.
	\end{align*}
Recall Lemma~\ref{l:S.estimates}.
If $y\ge\Omega(1)$, then $\CC\asymp1$, and we can bound
	\begin{align*}
	-\f{(d-1) A }{k^{-1}\Var\eph}
	&\le O\bigg(
	\f{ke^y}{w}\f{d w^2}{e^{2y}\dfz}
	\bigg\{S_2+\f{S_1}{2e^y}+O(w \dz)
	\bigg\}\bigg)
	\le O\bigg(k^2
	\bigg\{\f{S_2}{e^y\dfz}+\f{w}{e^y}
	\bigg\}\bigg)\\
	&\le
	O\bigg(k^2
	\bigg\{x(\lgm)^2
	+\f{w}{e^y}
	\bigg\}\bigg)
	\le O\bigg( \f{k^2w}{e^y}
	\bigg\{k^2 wx
	+1
	\bigg\}\bigg)
	\le O\bigg( \f{k^2w}{e^y}\bigg)
	\le \f1{e^{\Omega(k)}}\,,\\
	-\f{(d-1) B }{k^{-1}\Var\eph}
	&\le O\bigg(
	\f{ke^y}{w} \f{dw^2 S_1}{e^{2y}\dfz}\bigg)
	\le O\bigg(
	\f{k^2}{e^{y/2}} \f{S_1}{e^{y/2}\dfz}\bigg)
	\le O\bigg(\f{k^2 \lgm x }{e^{y/2}} \bigg)
	\le O\bigg(\f{k^3  x }{e^y} \bigg)
	\le \f1{e^{\Omega(k)}}\,.
	\end{align*}
If $y$ is small (in which case $\lgm$ is large), we use the more precise estimate \eqref{e:S.estimate.large.l} to obtain
	\[\f{(d-1)A}{k^{-1}\Var\eph}
	= O\bigg(\f{ke^y}{w}
	\f{dw^2 S_0 y}{e^{2y}\dfz}\bigg)
	= O\bigg(\f{\CC k^2 x y}{e^y}\bigg)
	=O (\CC^{1/2} k^2 x)
	\le O\bigg( \f1{e^{\Omega(k)}}\bigg)\,,
	\]
where the last estimate uses that $\dq\in\MMstar$. Similarly we have
	\[\f{(d-1)B}{k^{-1}\Var\eph}
	= O\bigg(\f{ke^y}{w}\f{dw^2xy}{e^y}\bigg)
	= O\bigg( \CC^{1/2} k^2 x \bigg)
	\le O\bigg( \f1{e^{\Omega(k)}}\bigg)\,.
	\]
The claim follows.
\end{proof}
\end{lem}

\begin{proof}[Proof of Proposition~\ref{p:second.derivative}]
Follows by combining Lemmas~\ref{l:second.derivative}--\ref{l:Gyy}.
\end{proof}

{\raggedright
\bibliographystyle{alphaabbr}
\bibliography{refs}
}

\end{document}